\documentclass[reqno,11pt]{amsart}

\usepackage{pgfplots}

\usepackage{xcolor}
\usepackage{graphicx}
\usepackage[hidelinks]{hyperref}
\usepackage{amscd}

\usepackage{amsmath}
\usepackage{amsthm}
\usepackage{amssymb}
\usepackage{latexsym,array}
\usepackage{amsfonts}
\usepackage{amsbsy}
\usepackage{amssymb}
\usepackage{dsfont}
\usepackage{mathrsfs}
\usepackage{mathtools}
\usepackage{enumitem}
\usepackage{float}
\usepackage{tikz}
\usepackage{tikz-cd}

\usepackage{color}
\usepackage{graphicx}
\usepackage[hidelinks]{hyperref}
\usepackage{cleveref}
\usepackage{fullpage}
\usepackage{comment}

\usepackage{dsfont}

\usepackage{cleveref}

\numberwithin{equation}{section}

\newtheorem{theorem}{Theorem}[section]
\newtheorem{lemma}[theorem]{Lemma}
\newtheorem{corollary}[theorem]{Corollary}

\theoremstyle{definition}
\newtheorem{remark}[theorem]{{\bf Remark}}
\newtheorem{definition}[theorem]{Definition}
\newtheorem{problem}[theorem]{Problem}

\newcommand{\BR}{\mathbb{R}}

\newcommand{\bemph}[1]{\textbf{\textit{#1}}}
\newcommand{\midbar}{ \ \big\vert \ }
\newcommand{\p}[1]{\left( #1 \right)}

\newcommand{\set}[1]{\left\{ #1 \right\}}
\newcommand{\modulo}[1]{\left| #1 \right|}
\newcommand{\norm}[1]{\left\| #1 \right\|}
\newcommand{\dif}{\hspace*{0.5mm}d\hspace*{0.25mm}}

\newcommand{\Spin}{\textup{Spin}}
\newcommand{\spin}{\mathfrak{spin}}

\newcommand{\Sc}[1]{\textup{Sc}\p{#1}}

\newcommand{\inner}[1]{\left\langle  #1 \right\rangle }

\usepackage{amscd}
\usepackage{tikz-cd}
\usepackage{tikz}

\crefname{enumi}{}{}
\crefname{enumii}{}{}

\author[I. Beschastnyi]{Ivan Beschastnyi}
\address{(IB)
	Charge de Recherche, Université Côte d'Azur, LJAD / Centre Inria d'Université Côte d'Azur, MCTAO
} \email{ ivan.beschastnyi@inria.fr}

\title[]{The $S$-resolvent estimates for the Dirac operator on hyperbolic and spherical spaces}
\author[F. Colombo]{Fabrizio Colombo}
\address{(FC)
	Politecnico di Milano\\Dipartimento di Matematica\\Via E. Bonardi, 9\\20133
	Milano, Italy}
\email{fabrizio.colombo@polimi.it}

\author[Simão Andrade Lucas]{Simão Andrade Lucas}
\address{(SL)
	Politecnico di Milano\\Dipartimento di Matematica\\Via E. Bonardi, 9\\20133
	Milano, Italy
} \email{simao.lucas@polimi.it}

\author[I. Sabadini]{Irene Sabadini}
\address{(IS)
	Politecnico di Milano\\Dipartimento di Matematica\\Via E. Bonardi, 9\\20133
	Milano, Italy
} \email{irene.sabadini@polimi.it}

\date{}

\begin{document}
\maketitle

\begin{abstract}

This seminal paper marks the beginning of our investigation into on the spectral theory based on  $S$-spectrum applied to the Dirac operator on manifolds.
Specifically, we examine in detail the cases of the Dirac operator $\mathcal{D}_H$  on hyperbolic space and the Dirac operator $\mathcal{D}_S$ on the spherical space, where these operators, and their squares $\mathcal{D}_H^2$ and $\mathcal{D}_S^2$, can be written in a very explicit form.
This fact is very important for the application of the spectral theory on the $S$-spectrum.
In fact, let $T$ denote a (right) linear Clifford operator, the $S$-spectrum is associated with a second-order polynomial in the operator $T$, specifically the operator defined as
$
Q_s(T) := T^2 - 2s_0T + |s|^2.
$
This allows us to associate to the Dirac operator boundary conditions that can be of Dirichlet type but also of Robin-like type.
Moreover, our theory is not limited to Hilbert modules; it is applicable to Banach modules as well. The spectral theory based on the $S$-spectrum has gained increasing attention in recent years, particularly as it aims to provide quaternionic quantum mechanics with a solid mathematical foundation from the perspective of spectral theory.
This theory was extended to Clifford operators, and more recently, the spectral theorem has been adapted to this broader context.
The $S$-spectrum is crucial for defining the so-called $S$-functional calculus for quaternionic and Clifford operators in various forms. This includes bounded as well as unbounded operators, where suitable estimates of sectorial and bi-sectorial type for the $S$-resolvent operator are essential for the convergence of the Dunford integrals in this setting.

\end{abstract}
\vskip 1cm
\par\noindent
 AMS Classification: 47A10, 58J50.
\par\noindent
\noindent {\em Key words}: $S$-spectrum, Dirac operator on hyperbolic space, Dirac operator on the spherical space.

\medskip
\textbf{Acknowledgements:} F. Colombo and I. Sabadini are supported by MUR grant Dipartimento di Eccellenza 2023-2027. I. Beschastnyi was supported by the French government through the France 2030 investment plan managed by the National Research Agency (ANR), as part of the Initiative of Excellence Université Côte d’Azur under reference number ANR-15-IDEX-01

\tableofcontents

\section{Introduction }

This paper represents the onset of our investigation into the spectral theory based on the $S$-spectrum as applied to the Dirac operator on manifolds. In particular, we focus on the Dirac operator in hyperbolic space and spherical space. In these contexts, the Dirac operator  $\mathcal{D}_H$ on hyperbolic space and the Dirac operator $\mathcal{D}_S$ on spherical space, along with their squares $\mathcal{D}_H^2$ and $\mathcal{D}_S^2$, can be expressed in a very explicit form. For these specific cases, we provide explicit estimates that indicate the regions where the $S$-spectrum of Dirac operators, under various boundary conditions, is contained, along with suitable estimates for the $S$-resolvent operator.
More in general in this paper we illustrates the applicability of the new spectral theory based on the $S$-spectrum to the Dirac operator on manifolds, establishing a link between spectral theory and differential geometry, thereby giving rise to the $S$-spectral theory in Riemannian geometry.
Our theory on the $S$-spectrum associated with Dirac operators on manifolds works even when these operators are not assumed to be self-adjoint.

\medskip
The spectral theory, rooted in the concept of the $S$-spectrum and associated with slice hyperholomorphicity, began its development in 2006. This framework reveals that
it is far more than a mere extension of complex operator theory, especially concerning unbounded operators (see \cite{RD}). The identification of the $S$-spectrum was accomplished using techniques from hypercomplex analysis, with initial insights arising from quaternionic quantum mechanics, see \cite{BF},  and for more details see \cite{adler} and  \cite{JONAQSTUD}.

An exhaustive description on the discovery of the $S$-spectrum can be found in the introduction of the book \cite{CGK}. The spectral theory of the $S$-spectrum was developed in several papers and it is systematically organized in the books \cite{ACS2016,AlpayColSab2020,AlpayColSab2024,FJBOOK,CGK,ColomboSabadiniStruppa2011}, where one can find also the development of the theory of slice hyperholomorphic functions. It is also important to point out that the quaternionic and the Clifford spectral theorems are based on the notion of $S$-spectrum as well, see \cite{ACK,ColKim}. However, this spectral theory goes beyond the quaternionic and the Clifford setting, as shown in \cite{ADVCGKS}.
As an example we mention that in the context of Clifford algebras, the theory has been recently applied to investigate the spectral properties of the gradient operator with nonconstant coefficients in $n$-dimensions
\begin{equation}\label{Eq_T}
T = \sum_{i=1}^n e_i a_i(x) \frac{\partial}{\partial x_i}, \qquad x \in \Omega,
\end{equation}
and associated with various boundary conditions, see \cite{Gradient}.
Here, $e_1, \ldots, e_n$ represent the imaginary units of the Clifford algebra
$\mathbb{R}_n$, and the coefficients $a_1, \ldots, a_n$ are defined on a
 smooth domain $\Omega \subseteq \mathbb{R}^n$. We require that the coefficients $a_1, \ldots, a_n: \overline{\Omega} \to (0, \infty)$  belong to $C^1(\overline{\Omega})$ and satisfy the appropriate bounds.
Unlike in classical complex spectral theory, the definition of the $S$-spectrum of $T$ is tied to the invertibility of the second-order operator defined as
\begin{equation}\label{Eq_Qs}
Q_s(T) := T^2 - 2s_0 T + |s|^2,
\end{equation}
where $s=s_0+s_1e_1+\dots+s_ne_n$ is a paravector, $s_0$ its real part and $|s|^2$ is its module.
This fact is a significant distinction from classical operator theory.

\medskip
Let $V$ be a Clifford module and by $\mathcal{B}(V)$ the denote the space of all bounded right-linear operators $T:V\to V$. The concept of the $S$-spectrum is grounded in the sum of an analogue of the resolvent series expansion, where the bounded operator $T$ and the paravector $s$ do not necessarily commute.
Specifically, we have:
\begin{subequations}
\begin{align}
\sum_{n=0}^\infty T^ns^{-1-n}=(T^2-2s_0T+|s|^2)^{-1} (\overline{s}-T),\qquad|s|>\Vert T\Vert, \label{Eq_SL_series} \\
\sum_{n=0}^\infty s^{-1-n}T^n=(\overline{s}-T)(T^2-2s_0T+|s|^2)^{-1},\qquad|s|>\Vert T\Vert, \label{Eq_SR_series}
\end{align}
\end{subequations}
for $s\in\mathbb{R}_n$ of the form $s=s_0+s_1e_1+\dots+s_ne_n$.
That is, for every $T\in\mathcal{B}(V)$, the \textit{$S$-spectrum} and the \textit{$S$-resolvent set} of $T$ are defined as
\begin{equation*}
\rho_S(T):=\{s\in\mathbb{R}^{n+1} \ \ |\ \  Q_s(T)^{-1}\in\mathcal{B}(V)\}\qquad\text{and}\qquad\sigma_S(T):=\mathbb{R}^{n+1}\setminus\rho_S(T).
\end{equation*}
 For every $s\in\rho_S(T)$, the inverse $Q_s(T)^{-1}$ is called the \textit{pseudo $S$-resolvent operator}.
A second main difference between complex and $S$-spectral theory lies in the fact that
\begin{equation*}
 s\mapsto Q_s(T)^{-1}\ \ {\rm for}\  s\in \rho_S(T),
\end{equation*}
is not a (operator valued) slice hyperholomorphic function. The operator $Q_s(T)$ is used for the definition of the $S$-spectrum, while the \textit{left} and the \textit{right $S$-resolvent operators} are motivated by the series representations \eqref{Eq_SL_series} and \eqref{Eq_SR_series}
\begin{equation*}
S_L^{-1}(s,T):=Q_s(T)^{-1}\overline{s}-TQ_s(T)^{-1}\qquad\text{and}\qquad S_R^{-1}(s,T):=(\overline{s}-T)Q_s(T)^{-1},\qquad s\in\rho_S(T).
\end{equation*}
These resolvent operators now preserve the slice hyperholomorphicity and are used in the definition of the $S$-functional calculus for Clifford operators.

 Observe that for unbounded operators $T: D(T)\subset X\to X$ we have that
 $ Q_s(T):  D(Q_s(T))\subset D(T)\to X$ so that the $S$-resolvent operators are defined on $X$.

\medskip

Clifford operators that are sectorial, bi-sectorial or strip-operators, can be defined in analogy to complex operators, but in this case the estimates associated with the $S$-resolvents are
\begin{equation*}
\Vert S_L^{-1}(s,T)\Vert\leq\frac{C}{|s|}\qquad\text{and}\qquad\Vert S_R^{-1}(s,T)\Vert\leq\frac{C}{|s|},
\end{equation*}
where the parameter $s$ belongs to suitable sectorial, bi-sectorial or strip-type subsets of $\rho_S(T)$.

\medskip
{\em The content of the paper and description of the main results.}
This paper addresses potential interests for a dual audience of researchers: those working in spectral theory and those interested in differential geometry.
In Section \ref{sec_Lax_Milgram} we introduce the functional setting in Clifford modules over $\mathbb{R}_n$ on which we work while in Section \ref{FACTONMANIFOLD}, we recall the basic results on the Dirac operator on Riemannian manifolds
Specifically, we revisit the essential concepts related to the description of Dirac operators on manifolds
and the fact that one can  write the Dirac operator in terms of the covariant derivative $\nabla_{ X}^\tau$. {\color{black} Let $U$ be a coordinate neighborhood of a Riemannian manifold $(M,g)$ of dimension $n$. In this neighborhood we can find an orthonormal frame of vector fields $E_1,\dots,E_n$. Then the Dirac operator on $(U,g)$ can be written as a differential operator on $\mathcal{C}^\infty(U,\mathcal{H})$
\begin{equation}
\mathcal{D} = \sum_{i=1}^ne_i\nabla^\tau_{E_i}.
\end{equation}
}
In order to apply the spectral theory on the $S$-spectrum we need
to recall that $\mathcal{D}^2$ is given as the sum of a second-order Laplacian and a curvature operator, in accordance with the Bochner-Weitzenb\"ock theorem.
By Theorem \ref{QUADRATODID} on $\mathcal{C}^\infty(U,\mathcal{H})$
   the square of the Dirac operator is given by (\ref{Dquadrato}).

As we mentioned, the Dirac operator on hyperbolic and spherical spaces takes on a particularly elegant form, and the squares of these operators given by (\ref{Dquadrato}) also have a very nice expression.

\medskip
Section \ref{SubSec:Hyperbolic} we study Dirac operator on the hyperbolic space with Dirichlet boundary conditions.
In order to illustrate our results,
choosing a representation $d\tau$, see Lemma \ref{dtaurep}, the Dirac operator on the hyperbolic space becomes:
\begin{equation}
		\mathcal{D}_H = \sum_{i=1}^{n-1} e_iy\partial_{x_i} - \alpha_ne_{n} + e_{n}y\partial_y,\ \ \ {\rm where} \ \ \   \alpha_n:= \frac{n-1}{2}
	\end{equation}
and $\mathcal{D}_H^2$ turns out to be
		\begin{equation}
			\mathcal{D}_H^2 = -y^2\Delta + e_{n}\mathcal{D}_H - \beta_n,\ \ \ {\rm where} \ \ \  \Delta:=\sum_{i=1}^{n-1}\partial_{x_i}^2+\partial_y^2,\ \ \ {\rm and} \ \ \ \beta_n :=\alpha_n+\alpha_n^2.
		\end{equation}
Since $\mathcal{D}_H$ and $\mathcal{D}_H^2$ on $\BR^{n}_+ := \set{(x,y) \midbar x=\p{x_1,\cdots,x_{n-1}}\in\BR^{n-1}, \quad y>0}$ have unbounded coefficients,  we consider domains $\Omega\subset \mathbb{R}_+^{n}$ for which there exist $0< m<M<+\infty$ such that
$\inf_{(x,y)\in\Omega}y=m$ and $\sup_{(x,y)\in\Omega}y=M$.
We study the
elliptic boundary value problem
\begin{equation}\label{ellBVP}
		\begin{cases}
			Q_s(\mathcal{D}_H)F = f, & \\ F\vert_{\partial \Omega} = 0
		\end{cases}
	\end{equation}
	and the invertibility of the operator $Q_s(\mathcal{D}_H)=\mathcal{D}_H^2-2s_0\mathcal{D}_H+|s|^2$ in the weak sense. In Theorem \ref{Th:Main_Hyper}, we prove that the invertibility is assured when the spectral parameter $s\in \mathbb{R}^{n+1}$ satisfies the two
explicit inequalities, given by (\ref{CONDWITHPUTCP}),
\begin{align*}
\Big||s|^2 - \beta_n\Big| - \alpha_n\sqrt{1+ 4s_0^2}+m^2>0,
\end{align*}
and
\begin{align*}
&
4m^2\Big||s|^2 - \beta_n\Big|-4nM^2 s_0^2
 - 4\Big(m^2\alpha_n +M^2(1+\alpha_n)\sqrt{n}\Big)\sqrt{1+4s_0^2}\nonumber
>M^2\Big(4(1+\alpha_n)^2+n\Big),
\end{align*}
for $|s|^2 - \beta_n\not=0$, where $\beta_n$ and $\alpha_n$  the constants $M>m>0$ are given above.
		Thanks to those estimates, for every $f\in L^2(\Omega)$, there exists a unique $F_f\in H_0^1(\Omega)$ such that
		\begin{equation}
			q_s(F_f,G) = \inner{f,G}_2, \qquad \text{for all } G\in H_0^1(\Omega)
		\end{equation}
where $q_s(F_f,G)$ is the  sesquilinear form, given by (\ref{FORMWITHSCALPROD}), associated with the boundary value problem (\ref{ellBVP}).
		Moreover, this solution satisfies the bounds
		\begin{equation*}
			\|Q_s(\mathcal{D}_H)^{-1}f\|_2\leq \frac{1}{K_{n,m,M}(s)}\, \norm{f}_2,
		\ \ \ {\rm and}\ \
			\|Q_s(\mathcal{D}_H)^{-1}f\|_D\leq \frac{1 }{K_{n,m,M}(s)}\,\norm{f}_2,
		\end{equation*}
where  $K_{n,m,M}(s)$ is given by (\ref{KAPPAnm}) and $\|\cdot\|_2$ is the $L^2$ norm of the function and  $\|\cdot\|_D$ is the $L^2$ norm of the derivatives.
Moreover in Corollary \ref{STIMASRESOL} we deduce the estimate  of the right $S$-resolvent operator
for $\mathcal{D}_H$ as a direct consequence of estimates (\ref{Eq:norm_est_hyper}). Precisely, we have
$$
\|S_R^{-1}(s,\mathcal{D}_H)f\|_2\leq \frac{1}{K_{n,m,M}(s)} \Big(|s|+\alpha_n+ M\sqrt{n} \Big)\,\,\norm{f}_2
$$
for $s\in \mathbb{R}^{n+1}$ satisfies the two explicit inequalities reported above given in (\ref{CONDWITHPUTCP}).

\medskip
In Section \ref{sec_Robin_HYPERBOLIC}
we give  the weak formulation of the spectral problem $Q_s(\mathcal{D}_H))F=f$ under suitable Robin-like boundary conditions for the operator $\mathcal{D}_H$, more precisely we investigate the problem
\begin{equation}
\begin{cases}
Q_s(\mathcal{D}_H)F = f \qquad\text{on } \Omega,
&
\\
\sum_{i=1}^{n-1}\nu_iy^2\partial_{x_i}F+\nu_ny^2\partial_yF +bF=0,\qquad\text{on }\partial\Omega
\end{cases}
\end{equation}
where $b$ is a given real valued function such that $b\in L^\infty(\partial\Omega)$.

In Theorem \ref{THFORTHEHYPERB} we prove that for $s\in\mathbb{R}^{n+1}$ varying in the suitable set described by condition \eqref{SECODESTROBPRIMEH} the problem admits a unique solution $F_f\in H^1(\Omega)$ for  every $f\in L^2(\Omega)$.

\medskip
In the second part of the paper we consider the Dirac operator $\mathcal{D}_S$ on the spherical space, which is computed  in \cite{DiracHarm} page 275, and by choosing the representation given by the left multiplication is given by
$$
\mathcal{D}_S=(1+|x|^2)\sum_{i=1}^ne_i\partial_{x_i}-nx.
$$

In Section \ref{sphericalDirichlet}, we study the Dirac operator on spherical space with Dirichlet boundary conditions. More precisely, we first examine the elliptic boundary value problem of the form \eqref{ellBVP} for the operator $ \mathcal{D}_S $. Using some technical lemmas, we prove in Theorem \ref{THFORTHESPHER} that, when $s \in \mathbb{R}^{n+1}$ varies in a suitable set, the boundary value problem for the Dirac operator on spherical spaces admits a unique solution $ F_f \in H^1_0(\Omega) $ for any $ f \in L^2(\Omega)$. Additionally, as we did for the hyperbolic space, we discuss the bounds satisfied by the solution $F_f$.

 Finally in Section \ref{sphericalRobin} we consider the
 Dirac operator on the spherical space with the associated Robin-like boundary conditions,
see Problem \ref{RobinSphere}. In Theorem \ref{ROBENTHFORTHESPHER} we prove conditions under which we have existence and uniqueness of the solution in $H^1(\Omega)$ for any datum in $L^2(\Omega)$.

We conclude the paper with Section \ref{sec_Concluding_remarks} that contains some
 concluding remarks on the sectorial and bi-sectorial estimates obtained in this paper in order to give a panorama on the related papers on this topic.

\section{Functional setting in Clifford modules over $\mathbb{R}_n$}\label{sec_Lax_Milgram}

Clifford algebras are a class of algebras that generalize the concept of complex numbers and quaternions. They are formed from a vector space equipped with a quadratic form, where the key feature is the relation between the basis vectors, typically expressed through anticommutation relations. In the following we denote by  $\mathbb{R}_n$, for $n\in\mathbb{N}$, $n\geq 2$ the
Clifford algebra generated by $n$ \textit{imaginary units} $e_1,\dots,e_n$ which satisfy the relations
\begin{equation*}
e_i^2=-1\qquad\text{and}\qquad e_ie_j=-e_je_i,\qquad i\neq j\in\{1,\dots,n\}.
\end{equation*}
More precisely, $\mathbb{R}_n$ is given by
\begin{equation}
\mathbb{R}_n:=\left\{\sum\nolimits_{A\in\mathcal{A}}x_Ae_A\ \  |\ \  x_A\in\mathbb{R},\,A\in\mathcal{A}\right\},
\end{equation}
using the index set
\begin{equation*}
\mathcal{A}:=\{(i_1,\dots,i_r)\  | \ r\in\{0,\dots,n\},\ \ 1\leq i_1<\dots<i_r\leq n\},
\end{equation*}
and the \textit{basis vectors} $e_A:=e_{i_1}\dots e_{i_r}$. Note, for $A=\emptyset$ the empty product of imaginary units is the real number $e_\emptyset:=1$. Moreover, we will consider the set of all \textit{paravectors}
\begin{equation*}
\mathbb{R}^{n+1}:=\left\{x_0+\sum\nolimits_{i=1}^nx_ie_i \ \ |\ \  x_0,x_1,\dots,x_n\in\mathbb{R}\right\}.
\end{equation*}
For any Clifford number $x\in\mathbb{R}_n$, we define
\begin{align*}
{\rm Sc}(x)&:=x_\emptyset=x_0, && \textit{(scalar part)} \\
\overline{x}&:=\sum\nolimits_{A\in\mathcal{A}}x_A\overline{e_A}, && \textit{(conjugate)} \\
|x|&:=\Big(\sum\nolimits_{A\in\mathcal{A}}|x_A|^2\Big)^{\frac{1}{2}}=({\rm Sc}(x\overline{x}))^{\frac{1}{2}}=({\rm Sc}(\overline{x}x))^{\frac{1}{2}}, && \textit{(norm)}
\end{align*}
where $\overline{e_A}=\overline{e_{i_r}}\dots\overline{e_{i_1}}$ and $\overline{e_i}=-e_i$. It is now obvious that for any Clifford number $x\in\mathbb{R}_n$ one can calculate its coefficients $x_A$ by
\begin{equation}\label{Eq_xA}
{\rm Sc}(x\overline{e_A})={\rm Sc}\Big(\sum\nolimits_{B\in\mathcal{A}}x_Be_B\overline{e_A}\Big)=\sum\nolimits_{B\in\mathcal{A}}x_B{\rm Sc}(e_B\overline{e_A})=x_A,
\end{equation}
where in the last equation we used that
\begin{equation}\label{Eq_SceBeA}
{\rm Sc}(e_B\overline{e_A})=\begin{cases} 1, & \text{if }B=A, \\ 0, & \text{if }B\neq A. \end{cases}
\end{equation}
\subsection{Clifford Hilbert modules}

Clifford (Hilbert) modules arise in the context of Clifford algebras. Specifically, a Clifford module is a vector space equipped with a linear action of a Clifford algebra.
This action is typically defined such that the elements of the Clifford algebra act on the module in a way that respects the algebra's multiplication rules.
Clifford modules are particularly important to provide a framework for studying the representations of the Dirac operator.
For any real Hilbert space $V_\mathbb{R}$ with inner product $\langle\cdot,\cdot\rangle_\mathbb{R}$ and norm $\Vert\cdot\Vert_\mathbb{R}^2=\langle\cdot,\cdot\rangle_\mathbb{R}$, we define the \textit{Clifford module}
\begin{equation*}
V:=\left\{\sum\nolimits_{A\in\mathcal{A}}F_A\otimes e_A\ \  |\ \  F_A\in V_\mathbb{R}\right\}.
\end{equation*}
For any vector $F=\sum_{A\in\mathcal{A}}F_A\otimes e_A\in V$ and any Clifford number $x=\sum_{A\in\mathcal{A}}x_Ae_A\in\mathbb{R}_n$, we equip this space with a left and a right scalar multiplication
\begin{subequations}
\begin{align}
xF:=&\sum\nolimits_{A,B\in\mathcal{A}}(x_BF_A)\otimes(e_Be_A), && \textit{(left-multiplication)} \\
Fx:=&\sum\nolimits_{A,B\in\mathcal{A}}(F_Ax_B)\otimes(e_Ae_B). && \textit{(right-multiplication)} \label{Eq_Right_multiplication}
\end{align}
\end{subequations}
Moreover, we define the \textit{inner product}
\begin{equation}\label{Eq_Inner_product}
\langle F,G\rangle:=\sum\nolimits_{A,B\in\mathcal{A}}\langle F_A,G_B\rangle_\mathbb{R}\,\overline{e_A}e_B,\qquad F,G\in V,
\end{equation}
and the \textit{norm}
\begin{equation}\label{Eq_Norm}
\Vert F\Vert:=\Big(\sum\nolimits_{A\in\mathcal{A}}\Vert F_A\Vert_\mathbb{R}^2\Big)^{\frac{1}{2}}=\big({\rm Sc}\langle F,F\rangle\big)^{\frac{1}{2}},\qquad F\in V.
\end{equation}
The sesquilinear form \eqref{Eq_Inner_product} is clearly right-linear in the second, and right-antilinear in the first argument, i.e. for every $F,G,H\in V$, $x\in\mathbb{R}_n$, there holds
\begin{align*}
\langle F,G+H\rangle&=\langle F,G\rangle+\langle F,H\rangle, && \langle G,Hx\rangle=\langle G,H\rangle x, \\
\langle F+H,G\rangle&=\langle F,G\rangle+\langle H,G\rangle, && \langle Gx,H\rangle=\overline{x}\langle G,H\rangle.
\end{align*}
Moreover, there also holds
\begin{equation}\label{Eq_Inner_product_property}
\langle G,xH\rangle=\langle\overline{x}G,H\rangle.
\end{equation}
Next we recall some more basic properties of the inner product \eqref{Eq_Inner_product} and the norm \eqref{Eq_Norm}.

\begin{lemma}\label{lem_Properties}
For every $F,G\in V$, $x\in\mathbb{R}_n$, we have:

\begin{enumerate}
\item[i)] $\Vert Fx\Vert\leq 2^{\frac{n}{2}}|x|\Vert F\Vert$
\qquad and\qquad$\Vert xF\Vert\leq 2^{\frac{n}{2}}|x|\Vert F\Vert$,
\item[ii)] $\Vert Fx\Vert=\Vert xF\Vert=|x|\Vert F\Vert$
 for $x\in\mathbb{R}^{n+1}$,
\item[iii)] $|\langle F,G\rangle|\leq 2^{\frac{n}{2}}\Vert F\Vert\,\Vert G\Vert$,
\item[iv)] $|{\rm Sc}\langle F,G\rangle|\leq\Vert F\Vert\,\Vert G\Vert$.
\end{enumerate}
\end{lemma}

\begin{remark}
In the vector space $V$, we consider two inner products that serve different purposes.
The first inner product, $\langle \cdot,\cdot\rangle: V \times V \to \mathbb{R}_n $ is the one in (\ref{Eq_Inner_product}), while the second is its scalar part, namely $\mathrm{Sc}\langle \cdot,\cdot\rangle: V \times V \to \mathbb{R}$.
The inner product $\langle \cdot,\cdot\rangle$ is used in the
 Riesz's representation theorem, that is,
 every linear and continuous functional $\varphi: V \to \mathbb{R}_n$  is
  represented by an element $F_\varphi \in V$ such that
  $\varphi(G) = \langle F_\varphi, G \rangle$ for every $G\in V$.
The second inner product is important since, unlike the first one, it gives rise to a norm in the classical sense.
\end{remark}

We now recall the Lax-Milgram lemma for Clifford modules. This lemma is well known in complex functional analysis while its Clifford algebra version can be found in \cite{Gradient}.
We will need this result in the later sections to prove the unique solvability of the spectral problem for the $S$-spectrum and, in particular, to detect regions of $S$-resolvent set and related estimates of the $S$-resolvent operator.

\begin{lemma}[Lax-Milgram lemma]\label{lem_Lax_Milgram}
Let $q:V\times V\to\mathbb{R}_n$ be right-linear in the second and right-antilinear in the first argument. Moreover, we assume that there exist constants $C\geq 0$ and $\kappa>0$, such that
\begin{equation*}
{\rm Sc} (q(F,F))\geq\kappa\Vert F\Vert^2\qquad\text{and}\qquad|q(F,G)|\leq C\Vert F\Vert\Vert G\Vert,\qquad F,G\in V.
\end{equation*}
Then, for every right-linear bounded functional $\varphi:V\to\mathbb{R}_n$, there exists a unique vector $F_\varphi\in V$, satisfying
\begin{equation*}
\Vert F_\varphi\Vert\leq\frac{1}{\kappa}\Vert{\rm Sc}(\varphi)\Vert\qquad\text{and}\qquad q(F_\varphi,G)=\varphi(G),\qquad\text{for all }G\in V.
\end{equation*}
\end{lemma}

\subsection{Clifford-Sobolev spaces}

The spectral theory on the $S$-spectrum is designed to work in Banach modules; however, the simpler setting to investigate boundary value problems is the $L^2$ spaces.
In this setting the $L^2$ space is composed of square-integrable functions which means that the integral of $Sc(\overline{F(x)}F(x))$, over a specified domain $\Omega \subseteq \mathbb{R}^{n}$, converges to a finite value.
Precisely, let $\Omega \subseteq \mathbb{R}^{n}$ be an open set with smooth boundary.
We define
$$
L^2(\Omega) \coloneqq L^2(\Omega,\mathbb{R}_n) \coloneqq \Bigl\{ F\colon\Omega\to\mathbb{R}_n\ \ | \ \ \int_\Omega Sc(\overline{F(x)}F(x))\,dx<+\infty\Bigr\}
 $$
with the sesquilinear form
$$
 \inner{F,G}_2\coloneqq \inner{F,G}_{L^2(\Omega,\mathbb{R}_n)} \coloneqq \int_\Omega \overline{F(x)}G(x)\,dx,
 $$
where $dx$ is the Lebesgue measure and $F = \sum_A F_A e_A$ and $G= \sum_A G_A e_A$, with $F_A,G_A: \Omega \to \mathbb{R}$.
Clearly
$
L^2(\Omega,\mathbb{R}_n) = L^2(\Omega,\mathbb{R}) \otimes \mathbb{R}_n
$
and it is a Clifford module.
Moreover, we introduce the Sobolev space:
$$
H^1(\Omega,\mathbb{R}_n) = H^1(\Omega,\mathbb{R}) \otimes \mathbb{R}_n
$$
 equipped with the sesquilinear form:
$$
\inner{F,G}_{H^1(\Omega,\mathbb{R}_n)} \coloneqq \inner{F,G}_{H^1(\Omega,\mathbb{R}_n)}
\coloneqq \inner{F,G}_{L^2(\Omega,\mathbb{R}_n)} + \sum_{j=1}^n \inner{\partial_{x_j}F,\partial_{x_j}G}_{L^2}
$$
where $\partial_{x_j} u$  denotes the $j$-th weak derivative of $u$ and so
$H^1(\Omega,\mathbb{R}_n) = H^1(\Omega,\mathbb{R}) \otimes \mathbb{R}_n$ is a Clifford module.

We define $H^1_0(\Omega) \coloneqq H^1_0(\Omega,\mathbb{R}_n)$ to be the closure of $C^\infty_c(\Omega,\mathbb{R}_n)$ in $H^1(\Omega,\mathbb{R}_n)$. As in the scalar case,
 see \cite[Theorem 9.17]{BREZIS}, if we suppose that $\partial\Omega$ is of class $\mathcal{C}^1$ and we assume that
\[
 F\in H^1(\Omega,\mathbb{R}_n)\cap \mathcal{C}(\overline{\Omega},\mathbb{R}_n)
\]
 then the condition $u=0$ on $\partial \Omega$ is equivalent to $F\in H_0^1(\Omega,\mathbb{R}_n)$.
In the general case, for arbitrary functions $F\in H^1(\Omega,\mathbb{R}_n)$, we have to consider the trace operator  $u\mapsto F|_{\partial \Omega}$,
and the space $H_0^1(\Omega,\mathbb{R}_n)$ turns out to be the kernel of the trace operator, i.e.
  \[
 H_0^1(\Omega,\mathbb{R}_n) := \left\{ F\in H^1(\Omega,\mathbb{R}_n): F|_{\partial \Omega} = 0 \right\},
 \]
 where $F|_{\partial \Omega} := \mathrm{tr}\, F$ is the Dirichlet trace, also denoted by $\tau_D$,
  has to be understood in the sense of the  trace operator, see \cite[p. 315]{BREZIS}.
Finally,  $H_0^1(\Omega,\mathbb{R}_n)=H_0^1(\Omega,\mathbb{R})\otimes\mathbb{R}_n$ is a subspace of $H^1(\Omega,\mathbb{R}_n)$ that is a Clifford Hilbert module itself.
For the sake of brevity we will also use the notation $L^2$ for $L^2(\Omega,\mathbb{R}_n)$ and similarly  $H^1$ and $H_0^1$ for $H^1(\Omega,\mathbb{R}_n)$ and $H_0^1(\Omega,\mathbb{R}_n)$, respectively.
We finally state two basic lemmas, which will be crucial throughout the paper, namely the Sobolev and the Poincar\'e inequality, for more details see \cite{Gradient}.
In order to state these inequalities we will need for every function $F\in H^1(\Omega)$ the \textit{Sobolev-seminorm}
\begin{equation*}
\Vert F\Vert_D:=\bigg(\sum_{i=1}^n\Big\Vert
\partial_{x_i}F\Big\Vert_{L^2(\Omega,\mathbb{R}_n)}^2\bigg)^{1/2}.
\end{equation*}
\begin{lemma}[Poincar\'e inequality]\label{lem_Poincare_inequality}\cite[Corollary 9.19 \& Remark 21]{BREZIS}
Let $\Omega\subset\mathbb{R}^n$ be open and either bounded in one direction or of finite measure. Then there exists a constant $C_P(\Omega)>0$, such that
\begin{equation}\label{Eq_Poincare_inequality}
\Vert F\Vert_{L^2}\leq C_P(\Omega)\Vert F\Vert_D,\qquad F\in H_0^1(\Omega).
\end{equation}
\end{lemma}

\section{Basic results on the Dirac operator on manifolds}\label{FACTONMANIFOLD}

In this section we collect the necessary concepts of the  description of Dirac operators on manifolds, for more details on this topic see e.g. \cite{Ni},  and  for the proofs of the theorems stated in this section  see the book \cite{DiracHarm}.
To give a precise definition of the $S$-spectrum of the Dirac operator on $M$, we need to the following steps:
\begin{itemize}
\item
Definition of the Dirac operator $\mathcal{D}$ on $M$,
\item
Identification of $\mathcal{D}^2$ as the sum of a second-order Laplacian and a curvature operator (Bochner-Weitzenb\"ock theorem).
\end{itemize}

The precise expression of the operator $\mathcal{D}^2$ is of crucial importance in order to define the $S$-spectrum because it is associated with the operator
$
Q_s(\mathcal{D}):=\mathcal{D}^2-2s_0\mathcal{D}+|s|^2.
$

	A way to treat functions on a Riemannian manifold $(M,g)$ that take values on a Clifford module (see Definition \ref{Def:Cliff_module}) is through the theory of fiber bundles. One uses the fact that the space of such functions is in bijection with the sections of a vector bundle with such Clifford module as fiber. Moreover, to define a Dirac operator on $(M, g)$, so that when $(M,g)$ is $(\BR^n,g_{\text{eucl}})$ we get the standard Dirac operator in Clifford analysis, one needs several things:
	
	\begin{enumerate}[label=\bemph{(\arabic*)}]
		\item A bundle of Clifford algebras $\mathcal{C}\longrightarrow M$ such that $\mathcal{C}_x\cong\BR_n$, $\forall x\in M$;
		\item A fiberwise injective morphism of vector bundles $\iota: \mathcal{T}^\ast M\hookrightarrow\mathcal{C}$ such that, $\forall x\in M$
		\begin{equation*}
			\set{\iota(u),\iota(v)}_{\mathcal{C}} =-2g(u,v), \quad \forall u,v\in \mathcal{T}_x^\ast M,
		\end{equation*}
       where $g$ is the co-metric on the cotangent bundle;
		\item A bundle of Clifford modules, i.e. a vector bundle $\mathcal{E}\longrightarrow M$, together with a morphism $\mathbf{c}:\mathcal{C}\longrightarrow\textup{End}(\mathcal{E})$ whose restrictions to the fibers are morphism of algebras;
		\item A connection on $\mathcal{E}$.		
	\end{enumerate}

 One way to define a connection exploits the fact that the spin group $\Spin(n)$, seen as a subgroup of the Clifford algebra $\mathbb R_n$, is a $2$-fold cover of the special orthogonal group $SO(n)$ with the covering map $\sigma: \Spin(n) \to SO(n)$. This can be exploited by introducing additional structure on a Riemannian manifold $(M,g)$.
    \begin{definition}
        Consider a Riemannian manifold $(M,g)$ and denote by $\pi: SO(\mathcal{T}M) \to M$ the principal bundle of oriented orthonormal frames of $M$. A \textit{spin structure} on a Riemannian manifold $(M,g)$ is a principle $\Spin(n)$-bundle $\rho: P_\Spin \to M$ together with a double-covering map $\mu: P_\Spin \to SO(\mathcal{T}M)$ such that $\rho = \pi \circ \mu$ and
        $$
        \mu(g\cdot p) = \sigma(g)\cdot \mu(p), \qquad \forall p\in P_\Spin, \,\forall g\in \Spin(n).
        $$
    \end{definition}

    We recall that neither existence nor uniqueness is guaranteed for existence of spin structures. Nevertheless, existence is completely determined by a topological invariant (the vanishing of the second Stiefel-Whitney class of $M$~\cite{Spin}). A manifold that allows a spin structure on its tangent bundle $\mathcal{T}M$ is then called a \emph{spin manifold}. We shall, in this article, restrict our studies to the case when $(M,g)$ is a spin manifold. Additionally, as for the moment we are not concerned with the dependence of our studies on the spin structure, we shall consider that our manifold $(M,g)$ is simply connected.

    In order to define the necessary ingredients recall, that given a $G$-principle bundle $P$ with a principle connection and a representation $\tau: G \to GL(V)$, one can construct the associated vector bundle $\mathcal{E} := P \times_\tau V$, which will have a connection induced from $P$. In our example, a principal connection on $SO(\mathcal{T}M)$ can be lifted to a principal connection on the Spin bundle $P_\Spin$. Given a representation $\tau:\Spin(n) \to GL(V)$, we obtain the associated bundle and the induced connection. We can take the $V$ to be the Clifford algebra $\mathbb R_n$ and the $\Spin(n)\subset \mathbb R_n$ representation on $\mathbb R_n$. Thus we can define the Dirac operator in an invariant way as follows
	\begin{definition} \label{Def:Dirac_op_inv}
		The Dirac operator $\mathcal{D}$ is the first-order differential operator defined by the composition of the Clifford multiplication with the covariant derivative $\nabla^\tau$,
		\begin{equation*}
			\mathcal{D} := \mathbf{c}\circ\nabla^\tau : C^\infty(\mathcal{E}) \overset{\nabla^\tau}{\longrightarrow} C^\infty(T^*M\otimes \mathcal{E}) \cong C^\infty(TM\otimes \mathcal{E}) \overset{\mathbf{c}}{\longrightarrow} C^\infty(\mathcal{E}).
		\end{equation*}
	\end{definition}

For simplicity of exposition, we introduce all the tools locally on a simply connected subset $U\subset\BR^n$. Let $g:U\to \mathbb{R}^{n \times n}$, where $g(x)=[g_{ij}(x)]$
is a smooth matrix-valued function defined on the open set $U$ in $\mathbb{R}^n$, and that $g(x)$ will always be taken to be positive-definite and symmetric.
Then
\begin{equation}\label{1.2i}
	g_x(u,v)=\sum_{i,j=1}^ng_{ij}(x)u_iv_j, \ \ u,\ v\in \mathbb{R}^n
\end{equation}
is a positive-definite inner product on $\mathbb{R}^n$ and
\begin{equation}\label{1.2ii}
	g_x(X,Y)=\sum_{i,j=1}^ng_{ij}(x)a_ib_j,
	\ \ {\rm where}\ \
	X=\sum_{i=1}^n a_{i}(x)\partial_{x_i},\ \ \ Y=\sum_{j=1}^n b_{j}(x)\partial_{x_j},
\end{equation}
defines a positive-definite inner product space on the tangent space $\mathcal{T}_x(U)$ to $U$ at $x$.
So $U$ can be seen as a coordinate neighborhood for the Riemannian manifold $M$ taking $x=(x_1,...,x_n)$ as coordinates and  (\ref{1.2ii}) as inner product, making it possible to study Dirac operators on $M$, by introducing it as a nonconstant coefficients nonhomogeneous first-order systems of differential operators on $\mathcal{C}^\infty(U,\mathcal{H})$. Here $\mathcal{H}$ should be considered as a Clifford module in the following sense:
\begin{definition} \label{Def:Cliff_module}
	A finite-dimensional real or complex Hilbert space $\mathcal{H}$ is said to be a \emph{Clifford module} if there exist skew-adjoint operators $e_1,\dots,e_n$ in $\textup{Aut}(\mathcal{H})$ such that
	\begin{equation} \label{Eq:Cliff_relation}
		e_je_k + e_ke_j = -2\delta_{jk}e_0 \qquad (1\leq j,k\leq n),
	\end{equation}
	where $e_0 \ (=I)$ is the identity operator.
\end{definition}
Let now
\begin{equation}\label{1.3i}
g^{-1}:U\to \mathbb{R}^{n\times n}, \ \ g^{-1}(x)=[g^{ij}(x)]
\end{equation}
be the inverse matrix-valued function inverse to $g$ and let
$$
\gamma^{}(x)=[\gamma_{ij}(x)],\ \ \gamma^{-1}(x)=[\gamma^{ij}(x)]:U\to \mathbb{R}^{n\times n},
$$
be the unique square roots of $g$ and $g^{-1}$, respectively.

	\begin{definition}\label{Def:iota_emb}
		Let $e_1,\dots,e_n$ be the skew-adjoint operators on $\mathcal{H}$
		satisfying the Clifford relation (\ref{Eq:Cliff_relation}), and set
		\begin{equation*}
			\iota(\sum_{j=1}^n \gamma^{ij}(x)\dif x_j) := e_i(x)=\sum_{j=1}^n \gamma^{ij}(x)e_j, \quad x\in U, \quad \text{for every } i=1,...,n,
		\end{equation*}		
		where $\dif x_j$ are just the duals to $\partial_{x_j}$.
	\end{definition}
	By definition we have that
	\begin{equation}\label{1.6}
		e_j(x)e_k(x)+e_k(x)e_j(x)=-2g^{jk}(x),\ \ \ x\in U.
	\end{equation}

	With respect to $g$ one has an orthonormal frame, i.e., a set of vector fields that form an orthonormal basis at each point $x\in U$. These are given by
	\begin{equation*}
		E_i(x)=\sum_{j=1}^n\gamma^{ij}(x)\partial_{x_j},\ \ \ x\in U \ \ i=1,...,n,
	\end{equation*}
	where $\partial_{x_i}$ are the basis with respect to the local coordinates $x = \p{x_1,\dots,x_n}$.

	On a manifold one differentiates vector fields with respect to another vector field through the notion of \emph{connection}. Let $\mathcal{C}^\infty(U,\mathcal{T}(U))$ be the space of smooth vector fields $X = \sum_i a_i(x)\partial_{x_i}$ on $U$.
	
	\begin{definition}\label{1.8}
		A connection on $(U,g)$ is any bilinear mapping $\nabla:(X,Y)\to \nabla_X(Y)$ from $\mathcal{C}^\infty(U,\mathcal{T}(U))\times \mathcal{C}^\infty(U,\mathcal{T}(U))\to\mathcal{C}^\infty(U,\mathcal{T}(U))$ which satisfy the conditions
		\begin{itemize}
			\item[(i)]
			$\nabla_{f X + g Y}(Z)=f\nabla_{ X}+g\nabla_{Y}(Z)$,
			\item[(ii)]
			$\nabla_X(fY)=(Xf)Y+f\nabla_{X}(Y)$,
		\end{itemize}
		for all $X,Y,Z\in \mathcal{C}^\infty(U,\mathcal{T}(U))$ and for all $f,g\in \mathcal{C}^\infty(U)$.
	\end{definition}

	Of particular interest is the Levi-Civita connection, which is the unique connection that is torsion-free and compatible with the metric. This is stated in the following theorem.
	
	\begin{theorem}\label{th1.10}
		There exists a unique connection on $(U,g)$ having the further properties
		\begin{itemize}
			\item[(i)]
			$\nabla_X(Y)-\nabla_Y(X)=[X,Y],$
			\item[(ii)]
			$Xg(Y,Z)=g(\nabla_X(Y),Z)+g(Y,\nabla_X(Z))$
		\end{itemize}
		for all $X,Y,Z\in \mathcal{C}^\infty(U,\mathcal{T}(U))$.
	\end{theorem}

	A connection when applied to two vector fields is again a vector field, i.e., $\nabla_X(Y)\in\mathcal{C}^\infty(U,T(U))$. This makes it possible to write the connection as the linear combination of the elements of the frame.

	\begin{definition}
		The functions $\nu_{ij}^k\in C^\infty(U)$ given
		by
		\begin{equation}
		\nabla_{E_i}(E_j(x))=\sum_{k=1}^n\nu_{ij}^k(x)E_k(x)
		\end{equation}
		are called \emph{connection coefficients}.
	\end{definition}

	There is a rather straightforward way of computing such functions.
	
	\begin{theorem}[ ] \label{Th:ChrisSymb}
		The connection coefficients associated to an orthonormal frame $E_1,\dots,E_n$ are given by
		\begin{equation*}
			\nu_{ij}^k = \frac{1}{2}\p{c_{ij}^k - c_{jk}^i + c_{ki}^j}, \qquad i,j,k = 1,\dots,n,
		\end{equation*}
		where $c_{ij}^k$ are structure functions defined by the identities $[E_i,E_j] = \sum_{k=0}^n c_{ij}^k E_k$, for $i,j = 1,\dots,n$.
	\end{theorem}

	We now wish to be able not only to differentiate vector fields but also $\mathcal{H}$-valued functions. For that we need to extend the Levi-Civita connection on $(U,g)$ to a connection $(X,F)\longrightarrow\nabla_X(F)$ $\mathcal{C}^\infty(U,\mathcal{T}(U))\times \mathcal{C}^\infty(U,\mathcal{H})\to \mathcal{C}^\infty(U,\mathcal{H})$. For that, we define connection coefficients that are not $\BR$ but $\spin(n)$-valued, the Lie algebra of the Lie group $\Spin(n)$, in terms of the Levi-Civita connection coeffiecients.
	
	\begin{definition}
		Let $\nu_{ij}^k$ be the functions defined above. The function $\nu_i(x)$ given by
		\begin{equation}\label{1.13iiinu}
			\nu_i(x)=\frac{1}{2}\sum_{j<k}\nu_{ij}^k(x)e_je_k=\frac{1}{4}\sum_{j,k=1}^n\nu_{ij}^k(x)e_je_k
		\end{equation}
		is a smooth $\spin(n)$-valued function on $U$.
	\end{definition}
	
	Let $d\tau :\spin(n)\to \textup{Aut}(\mathcal{H})$ be the representation of the Lie algebra of $\Spin(n)$ derived from the representation $\tau: \Spin(n)\to \textup{Aut}(\mathcal{H})$.
	
	\begin{definition}
		The map $\nabla^\tau:(X,F)\to \nabla_X(F)$ from $\mathcal{C}^\infty(U,\mathcal{T}(U))\times \mathcal{C}^\infty(U,\mathcal{H})\to\mathcal{C}^\infty(U,\mathcal{H})$ defines a connection, which is expressed as a covariant derivative $\nabla_X^\tau$ by
		\begin{equation*}
			\nabla_{X}^\tau = X(x) + d\tau(\nu_i(x))
		\end{equation*}
		acting on $\mathcal{C}^\infty(U,\mathcal{H})$.
	\end{definition}

In analogy with Definition \ref{1.8} it has the properties
the conditions
\begin{itemize}
\item[(i)]
$\nabla^\tau_{fX+gY}(F)=f\nabla^\tau_{ X}(F)+g\nabla^\tau_{Y}(F)$,
\item[(ii)]
$\nabla^\tau_X(fF)=(Xf)F+f\nabla^\tau_{X}(F)$,
\end{itemize}
for all $F,G\in \mathcal{C}^\infty(U,\mathcal{H})$ and for all $X\in \mathcal{C}^\infty(U,\mathcal{T}(U))$ and $f,g\in \mathcal{C}^\infty(U)$.
\begin{remark}
{\rm
Since $ \mathcal{C}^\infty(U,\mathcal{H})$ is a $\mathcal{C}^\infty(U)$-module
condition (i) in Theorem \ref{th1.10} does not make sense but (ii) in Theorem \ref{th1.10} becomes
\begin{equation}\label{1.18iii}
Xg_x(F(x),G(x))=(\nabla_X^\tau F(x),G(x))+(F(x),\nabla_X^\tau G(x)),
\end{equation}
where $(\cdot,\cdot)$
denotes the inner product on $\mathcal{H}$.
}
\end{remark}
The following theorem shows how to write the Dirac operator locally in terms of the covariant derivative $\nabla_{E_i}^\tau$ with respect to a local orthonormal basis $(E_1,\dots,E_n)$, which is a consequence of the embedding $\iota$ in Definition \ref{Def:iota_emb} and Definition \ref{Def:Dirac_op_inv}.
\begin{theorem}\label{Diraccovar}
As a differential operator on $\mathcal{C}^\infty(U,\mathcal{H})$ the Dirac operator on $(U,g)$ becomes
\begin{equation}
\mathcal{D} = \sum_{i=1}^ne_i\nabla^\tau_{E_i}
\end{equation}
and, more explicitly, we can write
\begin{equation}
\mathcal{D} = \sum_{i=1}^ne_i \Big(E_i(x)+d\tau(\nu_i(x))\Big),
\end{equation}
where $\nu_i(x)$ are given by (\ref{1.13iiinu}).
\end{theorem}

\begin{remark}
{\rm
Since the covariant derivative $\nabla^\tau_{E_i}$ becomes simply the vector field $\partial_{x_i}$ on
$\mathcal{C}^\infty(U,\mathcal{H})$ when $(U,g)$ is an open set of $\mathbb{R}^n$ with the usual Euclidean structure, it is clear that the operator $\mathcal{D}$ in Theorem \ref{Diraccovar} reduces to the standard Dirac operator.
}
\end{remark}
In order to define the $S$-spectrum it is necessary to write an explicit
form for the operator $\mathcal{D}^2$.

\begin{theorem}\label{QUADRATODID}
As a differential operator on $\mathcal{C}^\infty(U,\mathcal{H})$ the square of the Dirac operator is given by
\begin{equation}\label{Dquadrato}
\begin{split}
\mathcal{D}^2&=-\sum_{j=1}^n\Big( \nabla^\tau_{E_j}\nabla^\tau_{E_j} -\Big( \sum_{k=1}^n\nu_{jj}^k \nabla^\tau_{E_k} \Big) \Big)
+\sum_{i<j}e_ie_j\big(\nabla^\tau_{E_i}\nabla^\tau_{E_j}-
\nabla^\tau_{E_j}\nabla^\tau_{E_i}-\nabla^\tau_{[E_i,E_j]}
  \big).
\end{split}
\end{equation}
\end{theorem}

\begin{remark}
{\rm
The formula for $\mathcal{D}^2$ is at the heart of the property of $\mathcal{D}$ on manifold and in the case the operator $\mathcal{D}$ is not self-adjoint this formula is
of crucial importance for the definition of the $S$-spectrum.
We point out that the first term generalizes the Laplace operator while
the second term plays the role of a curvature operator.

The definition of the second order Laplacian $\Delta_\tau$
follows the classical $\Delta={\rm div} ({\rm grad})$ form of the classical Laplacian on $\mathbb{R}^n$
where in the present context the gradient operator is represented by the total derivative
\begin{equation}\label{1.23}
\nabla_\tau:F\to \sum_{j=1}^n\nabla_{E_j}^\tau F(x)\otimes e_j \ \ {\rm for}\ \ \ F\in \mathcal{C}^\infty(U,\mathcal{H})
\end{equation}
mapping
$\mathcal{C}^\infty(U,\mathcal{H})$ into $\mathcal{C}^\infty(U,\mathcal{H}\otimes \mathbb{R}^n)$.
As $\mathcal{H}\otimes \mathbb{R}^n$ is an $\mathcal{U}_n$-module on which $\tau\otimes\sigma$ defined a representation of $Spin(n)$
a linear connection $(X,G)\to \nabla_X^{\tau\otimes\sigma}$ on
$\mathcal{H}\otimes \mathbb{R}^n$-valued functions is obtained by replacing $\tau$ by
$\tau\otimes\sigma$ throughout the covariant derivative.
}
\end{remark}

\begin{definition}
The definition of the divergence operator from
$\mathcal{C}^\infty(U,\mathcal{H}\otimes \mathbb{R}^n)$ to $\mathcal{C}^\infty(U,\mathcal{H})$ is given by
\begin{equation}
G\to \sum_{j=1}^n\Big(\nabla^\tau_{E_j}(G_j)-\sum_{k=1}^n\nu_{jj}^k(x)G_k  \Big),\ \ \ \ {\rm where} \ \  \ G=\sum_{j=1}^nG_j\otimes e_j.
\end{equation}
This operator is the adjoint $\nabla_\tau^*$ of the divergence operator defined in (\ref{1.23}).
\end{definition}

\begin{definition}
The Laplace operator $\Delta_\tau$
is defined by
\begin{equation}
\Delta_\tau F=\nabla_\tau^*\nabla_\tau F
 =\sum_{j=1}^n\Big(\nabla^\tau_{E_j}\nabla^\tau_{E_j}(F)-\sum_{k=1}^n\nu_{jj}^k(x) \nabla^\tau_{E_k}F  \Big),
\end{equation}
for $F\in \mathcal{C}^\infty(U,\mathcal{H})$.
\end{definition}
Clearly $\Delta_\tau$ is given by
$$
\Delta_\tau =\sum_{ij=1}^n g^{ij}(x)\partial_{x_i}\partial_{x_j}+{\rm \ lower order\  terms},
$$
so every $\Delta_\tau$ has the same principal symbol, in particular $\Delta_\tau$ is elliptic. We now write the part of the operator related to the curvature.
Heuristically speaking the curvature is a measure of the failure of $\nabla_X$, $\nabla_Y$, to commute as operators on
$\mathcal{C}^\infty(U,\mathcal{T}(U))$, while the curvature on $\mathcal{H}$
 measures of the failure of $\nabla_X^\tau$, $\nabla_Y^\tau$, to commute as operators on $\mathcal{C}^\infty(U,\mathcal{H})$.

 \begin{definition}[Curvature operators]
 The curvature operators are defined as
 $$
 R(X,Y)=\nabla_X\nabla_Y-\nabla_Y\nabla_X-\nabla_{[X,Y]},\ \ \ \ {\rm for\  all}\ \ \ X,Y\in \mathcal{C}^\infty(U,\mathcal{T}(U))
 $$
 and
 $$
 R^\tau(X,Y)=\nabla_X^\tau\nabla_Y^\tau-\nabla_Y^\tau\nabla_X^\tau-\nabla^\tau_{[X,Y]},\ \ \ \ {\rm for\  all}\ \ \ X,Y\in \mathcal{C}^\infty(U,\mathcal{H})
 $$
 \end{definition}
So we have the important theorem on the structure of the operator $\mathcal{D}^2$.
\begin{theorem}\label{DdueStruc}
On $\mathcal{C}^\infty(U,\mathcal{H})$ we have
\begin{equation}
\begin{split}
\mathcal{D}^2&=-\Delta_\tau+\sum_{i<j}e_ie_jR^\tau (E_i,E_j)
\\
&
=-\Delta_\tau+ \frac{1}{2} \sum_{i<j}e_ie_j \Big(\sum_{k<\ell} R_{ijk\ell}(x)\dif\tau (e_ke_\ell)\Big),
\end{split}
\end{equation}
where $\Delta_\tau =\nabla_\tau^*\nabla_\tau$ is the Laplace operator on $\mathcal{C}^\infty(U,\mathcal{H})$ and
$$
R_{ijk\ell}(x) := g_x(R(E_i,E_j)E_k,E_l)
$$
are the components of the Riemann curvature tensor with respect to the orthonormal frame $E_1,\dots,E_n$.
\end{theorem}
	
In some special case the curvature term simplifies.
We recall that the scalar curvature of $(U,g)$ is defined by
\begin{equation} \label{Eq:Scalar_curv}
\kappa(x)=-\sum_{i,j=1}^n R_{ijij}(x).
\end{equation}
When we assume $\mathcal{H}=\Theta_n$ (see \cite{DiracHarm}) and $\tau$ is the Spin representation of $\Spin(n)$ on $\Theta_n$, $n=2m$, the second order operator
$\mathcal{D}^2$ is often called spinor Laplacian is
$$
\mathcal{D}^2=-\Delta_S+\frac{1}{4}\kappa(x).
$$
 We conclude this section by presenting results that are essential for explicitly formulating the Dirac operator on hyperbolic and spherical spaces, which will serve our purposes in the subsequent sections. Precisely, we need to select a representation of $\tau$
appearing in the Dirac operator, see (\ref{Eq:Dirac_Hyper_A}) below.
To this end, as representation $\tau$ of $\Spin(n)$ on $\BR_n$ we choose the left multiplication on $\BR_n$. More precisely, the map
	\begin{equation} \label{Eq:Spin_Rep}
		\tau: \Spin(n) \longrightarrow \text{End}(\BR_n),
	\end{equation}
	is defined as  $\tau(a)(u):=au$ for any $a\in\Spin(n)$ and $u\in\BR_n$. In the next lemma we determine $d\tau$ for a given representation $\tau$.
	
	\begin{lemma}\label{dtaurep}
		Let $\tau$ be the representation of $\Spin(n)$ on $\BR_n$ given by left multiplication. Then $\tau$ induces a representation $\dif\tau$
		\begin{equation*}
			\dif\tau:\spin(n)\longrightarrow\mathfrak{end}(\BR_n)
		\end{equation*}
		of $\spin(n)$ on $\BR_n$ given by $\dif\tau \equiv e_ie_j u$, for $u\in\BR_n$.
	\end{lemma}
	\begin{proof}
		There is a canonical embedding $\Lambda^2\BR^n\subset\BR_n$. Using Proposition 6.1 in \cite{Spin}, we know that the Lie algebra of $(\BR_n,[\cdot,\cdot])$ corresponding to the subgroup $\Spin(n)\subset\BR_n^{\times}$ is
		\begin{equation*}
			\spin(n) = \Lambda^2\BR^n.
		\end{equation*}
		This implies that any element $X\in\spin(n)$ is a linear combination of $e_{i}e_{j}$. In particular, we can consider the basis elements $X_{ij} = e_{i}e_{j}$
		and compute the represention of Lie algebra $\spin(n)$ on $\BR_n$ as follows,
		\begin{align*}
			\dif\tau(X_{ij})(u) = \frac{\dif}{\dif t} \tau(\exp tX_{ij})(u)\Big\vert_{t=0} = \frac{\dif}{\dif t} (\cos t + e_{i}e_j\sin t)u\Big\vert_{t=0} = e_{i}e_{j}u.
		\end{align*}
		This then induces as representation of the whole Lie algebra $\spin(n)$ as the left multiplication by $e_ie_j$.
	\end{proof}

	\section{Dirac operator on the hyperbolic space with Dirichlet  boundary conditions} \label{SubSec:Hyperbolic}
	
	Let us consider the Dirac operator $\mathcal{D}_H$ on the hyperbolic space.
	The operator $\mathcal{D}_H$ is computed in \cite{DiracHarm}, p. 273, for the whole half-space $\BR^{n}_+$ for $n\in \mathbb{N}$ and with $n\geq2$
		\begin{equation*}
			\BR^{n}_+ := \set{(x,y) \midbar x=\p{x_1,\cdots,x_{n-1}}\in\BR^{n-1}, \quad y>0}
		\end{equation*}
	with hyperbolic metric $g$ given by
		\begin{equation*}
			g(x,y) = \frac{1}{y^2}\p{\sum_{i=1}^{n-1} \dif x_i^2 + \dif y^2},
		\end{equation*}
$\mathcal{D}_H$ becomes
	\begin{equation} \label{Eq:Dirac_Hyper_A}
		\mathcal{D}_H = \sum_{i=1}^{n-1} e_i\p{y\partial_{x_i} + \frac{1}{2}d\tau(e_ie_{n}) } + e_{n}y\partial_y.
	\end{equation}

An immediate consequence of the  Lemma \ref{dtaurep} is the following:	
\begin{corollary}
Let $n\in \mathbb{N}$, $n\geq2$, and $\tau$ be the representation of $\spin(n)$ given by th eleft multiplication. Then, the Dirac operator (\ref{Eq:Dirac_Hyper_A}) on the hyperbolic space turned out to be
\begin{equation} \label{Eq:Dirac_HyperA}
		\mathcal{D}_H = \sum_{i=1}^{n-1} e_i\p{y\partial_{x_i} + \frac{1}{2}e_ie_{n}} + e_{n}y\partial_y
	\end{equation}
which can be written as:
\begin{equation} \label{Eq:Dirac_Hyper}
		\mathcal{D}_H = \sum_{i=1}^{n-1} e_iy\partial_{x_i} - \alpha_ne_{n} + e_{n}y\partial_y,
	\end{equation}
where we have set
\begin{equation}\label{alphan}
\alpha_n:= \frac{n-1}{2}.
\end{equation}
\end{corollary}
	The spectral problem we consider in this section refers to a second-order problem
\begin{equation} \label{Eq:Q_operator}
	Q_s(\mathcal{D}_H) = \mathcal{D}_H^2 - 2s_0\mathcal{D}_H + |s|^2,
\end{equation}
in which the square of the Dirac operator $\mathcal{D}_H$ also appears. Since we now have an explicit expression for the Dirac operator in $\mathcal{D}_H$ as shown in (\ref{Eq:Dirac_Hyper}), we will compute $\mathcal{D}_H^2$ formally when the operators act on smooth functions.

	\begin{lemma}\label{HYPDDUE}
Let $n\in \mathbb{N}$ and with $n\geq2$.
		The square of $\mathcal{D}_H$, in (\ref{Eq:Dirac_Hyper}), is given by
		\begin{equation}
			\mathcal{D}_H^2 = -y^2\Delta_n + e_{n}\mathcal{D}_H +2\alpha_ny\partial_yF - \beta_n,
		\end{equation}
		where
\begin{equation}
\Delta_n=\sum_{i=1}^{n-1}\partial_{x_i}^2+\partial_y^2
\end{equation}
and
\begin{equation}\label{BETAN}
\beta_n :=\alpha_n+\alpha_n^2= \frac{n-1}{2} + \frac{(n-1)^2}{4}.
\end{equation}
	\end{lemma}
	\begin{proof}
		The result follows from direct computations, indeed for $F\in C^2(\Omega,\BR_n)$, we have
		\begin{align*}
			\mathcal{D}_H^2F & = \p{\sum_{i=1}^{n-1} e_iy\partial_{x_i} - \alpha_ne_{n} + e_{n}y\partial_y}\p{\sum_{j=1}^{n-1} e_jy\partial_{x_j} - \alpha_ne_{n} + e_{n}y\partial_y}F \\ & = \p{\sum_{i=1}^{n-1} e_iy\partial_{x_i}}\p{\sum_{j=1}^{n-1} e_jy\partial_{x_j}}F - \p{\sum_{i=1}^{n-1} e_iy\partial_{x_i}}\p{\alpha_ne_{n}}F + \p{\sum_{i=1}^{n-1} e_iy\partial_{x_i}}\p{e_{n}y\partial_y}F
\\
& \hspace*{0.5cm} - \p{\alpha_ne_{n}}\p{\sum_{j=1}^{n-1} e_jy\partial_{x_j}}F + \p{\alpha_ne_{n}}\p{\alpha_ne_{n}}F - \p{\alpha_ne_{n}}\p{e_{n}y\partial_y}F
\\
& \hspace*{1cm} + \p{e_{n}y\partial_y}\p{\sum_{j=1}^{n-1} e_jy\partial_{x_j}}F - \p{e_{n}y\partial_y}\p{\alpha_ne_{n}}F + \p{e_{n}y\partial_y}\p{e_{n}y\partial_y}F
 \end{align*}
 so that
 \begin{align*}
			\mathcal{D}_H^2F & =
  - \sum_{i=1}^{n-1} y^2\partial_{x_i}^2F + \sum_{i=1}^{n-1} e_ie_{n}y^2\partial_{x_i}\partial_yF + e_{n}y\partial_y\p{\sum_{j=1}^{n-1} e_jy\partial_{x_j}}F - \alpha_n^2F + \alpha_ny\partial_yF \\ & \hspace*{1cm} + y\alpha_n\partial_yF - y^2\partial_y^2F - y\partial_yF .
 \end{align*}
 With some further simplifications we get
 \begin{align*}
			\mathcal{D}_H^2F & = - \sum_{i=1}^{n-1} y^2\partial_{x_i}^2F - y^2\partial_y^2F
+ e_{n}\p{e_{n}y\partial_yF + y \sum_{i=1}^{n-1} e_i\partial_{x_i}F} - \alpha_n^2F
+2\alpha_ny\partial_yF
\\
&
= - \sum_{i=1}^{n-1} y^2\partial_{x_i}^2F - y^2\partial_y^2F + e_{n}\p{\mathcal{D}_H + \alpha_ne_{n}}F - \alpha_n^2F +2\alpha_ny\partial_yF
\\ &
= -y^2\Delta_n F + e_{n}\mathcal{D}_HF - \beta_nF+2\alpha_ny\partial_yF ,
		\end{align*}
where $\beta_n$ are given by (\ref{BETAN}) and this concludes the proof.
	\end{proof}

\begin{remark}\label{DOMAINS}
We will consider the hyperbolic Dirac operator on some domains that can be bounded or unbounded,
as an example of unbounded domain in one direction like the infinite strip  $\BR_{a,b}^{n}$, is defined as
		\begin{equation*}\label{HYPSTRP}
			\BR_{a,b}^{n} := \set{(x,y) \midbar x=\p{x_1,\cdots,x_{n-1}}\in\BR^{n-1}, \quad 0<a<y<b}
		\end{equation*}
 where $0<a<b< +\infty$ are two given constants.
We can also consider a bounded domain with smooth boundary $\Omega\subset \mathbb{R}_+^{n}$ denoted by $\partial \Omega$.
Since the coefficients of $\mathcal{D}_H$ and of the second order operator in $\mathcal{D}_H^2$ are unbounded functions on $\mathbb{R}_+^{n}$ in order to apply our techniques we assume that $\Omega$ is such that there exist constants $0< m<M<+\infty$ such that
\begin{equation}\label{CONTnM1}
\inf_{(x,y)\in\Omega}y=m,\ \ \ \ \sup_{(x,y)\in\Omega}y=M,
\end{equation}

Observe that in the case of the strip $\BR_{a,b}^{n}$ it is $m=a$ and $b=M$.

	\end{remark}	
We will study the elliptic boundary value problem with homogeneous Dirichlet boundary condition
\begin{equation} \label{Eq:BVProblem}
		\begin{cases}
			Q_s(\mathcal{D}_H)F = f,\ \ {\rm in} \ \ \Omega & \\ F\vert_{\partial \Omega} = 0
		\end{cases}
	\end{equation}
	and the invertibility of the operator $Q_s(\mathcal{D}_H)$ in (\ref{Eq:Q_operator}) in the weak
	sense.
 That is, we ask whether the problem (\ref{Eq:BVProblem}) admits a unique
	weak solution for every $f\in L^2(\Omega, \mathbb{R}_n)$.
Since the operator $Q_s(\mathcal{D}_H)$ is elliptic, by standard properties of elliptic equations, the weak solution is also a strong solution.

\subsection{Weak formulation of the Dirichlet boundary value problem (\ref{Eq:BVProblem})}
We work in the space of square-integrable functions on $\Omega$ with values in $\mathbb{R}_n$. For the current section we use the Euclidean volume $\dif x := \dif x_1 \cdots\dif x_{n-1}\dif y$.
We introduce the inner product
\begin{equation}\label{SCLAL2}
			\inner{F,G}_2 := \inner{F,G}_{L^2(\Omega,\mathbb{R}_n)} = \int_{\Omega} \overline{F}G \: \dif x.
		\end{equation}
and the  scalar product
\begin{equation}\label{ScSCLAL2}
			{\rm Sc}\inner{F,G}_2 := {\rm Sc}\inner{F,G}_{L^2(\Omega,\mathbb{R}_n)} = {\rm Sc}\int_{\Omega} \overline{F}G \: \dif x.
		\end{equation}
The first one will be used for the weal formulation of the boundary
 value problem and the second one is used to define $L^2(\Omega,\mathbb{R}_n)$ and
the norm squared in $L^2(\Omega,\mathbb{R}_n)$ will be denoted by
$$
\|F\|_2^2={\rm Sc}\inner{F,F}_{L^2(\Omega,\mathbb{R}_n)}:={\rm Sc}\int_{\Omega}\overline{F}F \dif x.
$$
We denote by $C^\infty_c(\Omega,\mathbb{R}_n)$ set of Clifford-valued infinitely differentiable functions with compact support on $\Omega$.
Because of the homogeneous boundary conditions in problem (\ref{Eq:BVProblem}) we  will work on the space $H^1_0(\Omega,\mathbb{R}_n)$ which is the closure of $C^\infty_c(\Omega,\mathbb{R}_n)$  in the norm
 $$
\|F\|_{H^1(\Omega,\mathbb{R}_n)}^2 \coloneqq {\rm Sc}\inner{F,F}_{H^1(\Omega,\mathbb{R}_n)}
\coloneqq {\rm Sc}\inner{F,F}_{L^2(\Omega,\mathbb{R}_n)} + {\rm Sc}\sum_{j=1}^n \inner{\partial_{x_j}F,\partial_{x_j}F}_{L^2(\Omega,\mathbb{R}_n)}.
$$
This norm associated with the scalar product
$$
{\rm Sc}\inner{F,G}_{H^1(\Omega,\mathbb{R}_n)} \coloneqq {\rm Sc}\inner{F,G}_{H^1(\Omega,\mathbb{R}_n)}
\coloneqq {\rm Sc}\inner{F,G}_{L^2} + {\rm Sc}\sum_{j=1}^n \inner{\partial_{x_j}F,\partial_{x_j}G}_{L^2(\Omega,\mathbb{R}_n)}.
$$
We will use the scalar product
 $$
\inner{F,G}_{H^1} \coloneqq \inner{F,G}_{H^1(\Omega,\mathbb{R}_n)}
\coloneqq \inner{F,G}_{L^2} + \sum_{j=1}^n \inner{\partial_{x_j}F,\partial_{x_j}G}_{L^2}
$$
for the formulation of the boundary value problem.
In $H_0^1(\Omega, \mathbb{R}_n)$ we denote the norm as
$$
\|F\|_{H_0^1}:= \|F\|_{H_0^1(\Omega, \mathbb{R}_n)}:= \norm{F}_2^2+\norm{F}_D^2 ,
 $$
where we use the notation $\norm{F}_2^2$ for   $\norm{F}_{L^2}^2$ and
	$
		\norm{F}_D^2 = \sum_{i=1}^{n} \norm{\partial_{x_i} F}_2^2.
	$
\begin{remark}\label{INTBYPARTFORM} We recall the integration by part formula associated with the operator $A$ defined by
$$
AF:=-\sum_{i=1}^n a_i\partial_{x_i}\Big(a_i\partial_{x_i}F\Big)
$$
where the coefficients $a_i(x)$ are real-valued smooth functions of the variables $x=(x_1,...,x_n)\in \Omega\subseteq \mathbb{R}^{n}$.
Denoting by $\vec{\nu}=(\nu_1,\dots,\nu_n)$ the outer unit normal vector to the boundary smooth boundary  $\partial\Omega$ of $\Omega$, we have
\begin{align*}
-\sum_{i=1}^n\Big\langle a_i\partial_{x_i}
\Big(a_i\partial_{x_i}F\Big),G\Big\rangle_{L^2(\Omega, \mathbb{R}_n)}
=\sum_{i=1}^n\Big\langle a_i\partial_{x_i}F,\partial_{x_i}(a_iG)\Big\rangle_{L^2(\Omega, \mathbb{R}_n)}-\sum_{i=1}^n\Big\langle\nu_ia_i
\partial_{x_i}F,a_iG\Big\rangle_{L^2(\partial\Omega, \mathbb{R}_n)} .
\end{align*}
\end{remark}
To obtain the weak formulation of the boundary value problem
we integrate by parts the second order terms of $\mathcal{D}_H^2$ and for $F,G\in C^\infty_c(\Omega,\mathbb{R}_n)$, we get
	\begin{align*}
		-\sum_{i=1}^{n-1}\int_{\Omega} y^2\partial_{x_i}^2\overline{F}G \dif x &- \int_{\Omega} y^2\partial_y^2\overline{F}G \dif x
\\
&
=-\sum_{i=1}^{n-1}\int_{\Omega} y\partial_{x_i} (y\partial_{x_i}\overline{F}) G \dif x - \int_{\Omega} \partial_y^2\overline{F} (y^2G) \dif x
\\
& =
\sum_{i=1}^{n-1}\int_{\Omega} y^2\partial_{x_i}\overline{F}\partial_{x_i}G \dif x
+ \int_{\Omega} y^2\partial_y\overline{F}\partial_yG \dif x  + \int_{\Omega} 2y\partial_y\overline{F}G \dif x.
	\end{align*}
Now for $F,G\in H_0^1(\Omega, \mathbb{R}_n)$ and for every, $s\in\BR^{n+1}$ we consider the  sesquilinear form associated with the boundary value problem (\ref{Eq:BVProblem}) and we obtain:
	\begin{align}\label{THEFORM}
		q_s(F,G) & = \int_{\Omega} \sum_{i=1}^{n-1} y^2\partial_{x_i}\overline{F}\partial_{x_i}G \dif x + \int_{\Omega} y^2\partial_y\overline{F}\partial_yG \dif x
+ 2(1+\alpha_n)\int_{\Omega} y\partial_y\overline{F}G \dif x \nonumber
\\
&
 \hspace*{5cm} + \int_{\Omega} \overline{(e_{n} - 2s_0)\mathcal{D}_HF}G \dif x + \int_{\Omega} (\modulo{s}^2 - \beta_n)\overline{F}G \dif x.
	\end{align}
	\begin{problem}
Let $\mathcal{D}_H$ be the Dirac operator in (\ref{Eq:Dirac_Hyper}), let $\beta_n$ be the constants defined in (\ref{BETAN}) and
	using the scalar product $\inner{F,G}_2$ defined in (\ref{SCLAL2}) we write the  sesquilinear form
$q_s(F,G)$ defined in (\ref{THEFORM}) as
\begin{align}\label{FORMWITHSCALPROD}
		q_s(F,G) & = \inner{ \sum_{i=1}^{n-1} y^2\partial_{x_i}F,\partial_{x_i}G}_2
+ \inner{ y^2\partial_yF,\partial_yG}_2
 + 2(1+\alpha_n)\inner{ y\partial_yF,G}_2\nonumber
  \\ & \hspace*{5cm}
  + \inner{ (e_{n} - 2s_0)\mathcal{D}_HF,G}_2
  + \inner{ (\modulo{s}^2 - \beta_n)F,G}_2,
	\end{align}
with $\textup{dom } (q_s):= H^1_0(\Omega, \mathbb{R}_n) \times H^1_0(\Omega, \mathbb{R}_n)$.
Show that for some values of the spectral parameter
$s\in\mathbb{R}^{n+1},$ for every $f\in L^2(\Omega, \mathbb{R}_n)$ there exists a unique solution $F_f\in H^1_0(\Omega, \mathbb{R}_n)$ such that
	\begin{equation*}
		q_s(F_f,G) = \inner{f,G}_2, \quad \text{for all} \ G\in H^1_0(\Omega, \mathbb{R}_n).
	\end{equation*}
	Furthermore, determine $L^2$- and $D$-estimates of $F_f$, depending on the parameter $s\in \mathbb{R}^{n+1}$.
\end{problem}
We need the following preliminary lemma.
\begin{lemma}\label{STIMADF}
Let $\mathcal{D}_H$ be the Dirac operator in (\ref{Eq:Dirac_Hyper}). Then, we have
\begin{equation} \label{ESTIM:Dirac_Hyper}
		\|\mathcal{D}_HF\|_2=({\rm Sc}\inner{\mathcal{D}_HF,\mathcal{D}_HF}_2)^{1/2}
\leq  M\sqrt{n}\|F\|_D + \alpha_n\|F\|_2,
	\end{equation}
where $\alpha_n$ is given by (\ref{alphan}) and $M$ is as in \eqref{CONTnM1}.
\end{lemma}
\begin{proof}
It follows from
\begin{align*}
{\rm Sc}\inner{\mathcal{D}_HF,\mathcal{D}_HF} &= {\rm Sc}\inner{\sum_{i=1}^{n-1} e_iy\partial_{x_i}F + e_{n}y\partial_yF - \alpha_ne_{n}F,\sum_{i=1}^{n-1} e_iy\partial_{x_i}F + e_{n}y\partial_yF - \alpha_ne_{n}F}_{2}
\\
&
=
{\rm Sc}\inner{\sum_{i=1}^{n-1} e_iy\partial_{x_i}F + e_{n}y\partial_yF  ,\sum_{i=1}^{n-1} e_iy\partial_{x_i}F + e_{n}y\partial_yF }_{2}
\\
&
-{\rm Sc}\inner{\sum_{i=1}^{n-1} e_iy\partial_{x_i}F + e_{n}y\partial_yF , \alpha_ne_{n}F}_{2}
\\
&
-{\rm Sc} \inner{  \alpha_ne_{n}F,\sum_{i=1}^{n-1} e_iy\partial_{x_i}F + e_{n}y\partial_yF }_{2}
\\
&
+{\rm Sc}\inner{ \alpha_ne_{n}F, \alpha_ne_{n}F}_{2}
\end{align*}
so that using the properties of the scalar product in Lemma \ref{lem_Properties} for the four terms we have
\begin{align*}
\Big|{\rm Sc}\inner{\sum_{i=1}^{n-1} e_iy\partial_{x_i}F + e_{n}y\partial_yF  ,\sum_{i=1}^{n-1} e_iy\partial_{x_i}F + e_{n}y\partial_yF }_{2}\Big|
&\leq \|\sum_{i=1}^{n-1} e_iy\partial_{x_i}F + e_{n}y\partial_yF \|_2^2
\\
& \leq  nM^2\|F\|_D^2
\end{align*}
which is obtained by observing that
\begin{align*}
\|\sum_{i=1}^{n-1} e_iy\partial_{x_i}F + e_{n}y\partial_yF \|_2&\leq
\sum_{i=1}^{n-1}\| e_iy\partial_{x_i}F\|_2 +\| e_{n}y\partial_yF \|_2
\\
&
=\sum_{i=1}^{n-1}|e_iy|\| \partial_{x_i}F\|_2 +|ye_{n}|\|\partial_yF \|_2
\leq\\
&
\leq M(\sum_{i=1}^{n-1}\| \partial_{x_i}F\|_2 +\|\partial_yF \|_2)
\\
&
\leq M (\sum_{i=1}^{n}\| \partial_{x_i}F\|_2)  \ \ \ setting  \ \ y=x_n
\\
&
\leq  M (\sum_{i=1}^{n}1^2)^{1/2}  \Big(\sum_{i=1}^{n}\|\partial_{x_i}F\|^2_2\Big))^{1/2}
\\
&
\leq M\sqrt{n}\|F\|_D.
\end{align*}

The second term is estimated by
\begin{align*}
\Big|{\rm Sc}\inner{\sum_{i=1}^{n-1} e_iy\partial_{x_i}F + e_{n}y\partial_yF , \alpha_ne_{n}F}_{2}\Big|
&
\leq  \|\sum_{i=1}^{n-1} e_iy\partial_{x_i}F + e_{n}y\partial_yF\|_2\|\alpha_ne_{n}F\|_2
\\
&
\leq  \sqrt{n}\alpha_nM \|F\|_D\|e_{n}F\|_2
\\
&
=\sqrt{n}\alpha_nM \|F\|_D\|F\|_2.
\end{align*}
Similarly we have
\begin{align*}
\Big| {\rm Sc}\inner{  \alpha_ne_{n}F,\sum_{i=1}^{n-1} e_iy\partial_{x_i}F + e_{n}y\partial_yF }_{2}\Big|
\leq \sqrt{n}\alpha_nM \|F\|_D\|F\|_2
\end{align*}
and finally
\begin{align*}
\Big|{\rm Sc}\inner{ \alpha_ne_{n}F, \alpha_ne_{n}F}_{2}\Big|\leq\alpha_n^2\|F\|_2^2.
\end{align*}
The estimates for $\|\mathcal{D}_HF\|_2^2$, using the above computations, becomes
\begin{align*}
\|\mathcal{D}_HF\|_2^2&\leq  nM^2\|F\|_D^2+2\sqrt{n}\alpha_nM \|F\|_D\|F\|_2+  \alpha_n^2\|F\|_2^2
=\Big(\sqrt{n} M\|F\|_D^2+ \alpha_n\|F\|_2^2\Big)^2,
\end{align*}
from which we get the statement.
\end{proof}

\subsection{Estimates without Poincaré inequality}
Using Lax-Milgram techniques we obtain some sufficient
 conditions in order to identify part of the $S$-resolvent set.

 Some on these inequalities can be expressed solely in terms of the constants $m$ and $M$, as well as the dimension of the space $n$. Therefore, we need to distinguish between two cases.

 More precisely, in the  sesquilinear form
(\ref{FORMWITHSCALPROD}) we consider the term $\inner{ (\modulo{s}^2 - \beta_n)F,G}_2$ in which we replace $G=F$, and then we have:
\begin{itemize}
\item[(I)]
 For $\modulo{s}^2 - \beta_n\not=0$ we can use in $H_0^1$ the norm
 $\|F\|_{H_0^1(\Omega, \mathbb{R}_n)}:= \norm{F}_2^2+\norm{F}_D^2 $.
 \item[(II)]
 For $\modulo{s}^2 - \beta_n=0$ we cannot use in $H_0^1$ the norm
 $\|F\|_{H_0^1(\Omega, \mathbb{R}_n)}:= \norm{F}_2^2+\norm{F}_D^2 $,
 but we have to use the Poincaré inequality in which the Poincaré constant appears. This constant is not explicitly computed for arbitrary domains, so the estimates we obtain depend on a positive constant that is not a priori known.
 The advantage is that the term $\norm{F}_D^2 $ gives a norm on $H_0^1$.
\end{itemize}

	\begin{theorem} \label{Th:Main_Hyper}
		Let $\Omega\in \mathbb{R}^{n}_+$ be as in Remark \ref{DOMAINS}
		 and let $\mathcal{D}_H$ be the Dirac operator on the hyperbolic space as in (\ref{Eq:Dirac_Hyper}).
Let $s\in\BR^{n+1}$ be such that
\begin{align}\label{CONDWITHPUTCP}
&\Big||s|^2 - \beta_n\Big| - \alpha_n\sqrt{1+ 4s_0^2}+m^2>0,
\\
&
4m^2\Big||s|^2 - \beta_n\Big|-4nM^2 s_0^2
 - 4\Big(m^2\alpha_n +M^2(1+\alpha_n)\sqrt{n}\Big)\sqrt{1+4s_0^2}\nonumber
\\
& >M^2\Big(4(1+\alpha_n)^2+n\Big),\nonumber
\end{align}
for $|s|^2 - \beta_n\not=0$,
where $\beta_n:=\alpha_n+ \alpha_n^2$ and $\alpha_n$ are the constants given in (\ref{alphan}),  and the constants $M>m>0$ are explicitly given by (\ref{CONTnM1}).
		Then, for every $f\in L^2(\Omega)$ there exists a unique $F_f\in H_0^1(\Omega)$ such that
		\begin{equation}
			q_s(F_f,G) = \inner{f,G}_2, \qquad \text{for all } G\in H_0^1(\Omega),
		\end{equation}
where $q_s(F_f,G)$ is the sesquilinear form given by (\ref{FORMWITHSCALPROD}) associated with the Dirichlet problem.
		Moreover, this solution satisfies the bounds
		\begin{equation} \label{Eq:norm_est_hyper}
			\|Q_s(\mathcal{D}_H)^{-1}f\|_2\leq \frac{1}{K_{n,m,M}(s)}\, \norm{f}_2,
		\ \ \ {\rm and}\ \
			\|Q_s(\mathcal{D}_H)^{-1}f\|_D\leq \frac{1 }{K_{n,m,M}(s)}\,\norm{f}_2,
		\end{equation}
where
\begin{align}\label{KAPPAnm}
K_{n,m,M}(s):&=\frac{1}{2}\Big||s|^2 - \beta_n\Big| -\frac{1}{2} \alpha_n\sqrt{1+ 4s_0^2}
+\frac{1}{2}m^2\nonumber
\\
&
-\frac{1}{2}\sqrt{\Big( \Big||s|^2 - \beta_n\Big| - \alpha_n\sqrt{1+ 4s_0^2}-m^2\Big)^2
+M^2\Big(\textcolor{black}{ 2(1+\alpha_n)}+\textcolor{black}{\sqrt{n}}\sqrt{1+ 4s_0^2}\Big)^2}.
\end{align}
	\end{theorem}

	\begin{proof}
		We will verify that the  sesquilinear form $q_s$ satisfies the assumptions of Lax-Milgram, see Lemma \ref{lem_Lax_Milgram}.

\medskip
{\em Step 1}. The continuity of the  sesquilinear form $q_s$ follows from the estimates		
		\begin{align}\label{Eq:Sc_cont_hyper}
		|q_s(F,G)| & \leq \Big| \inner{ \sum_{i=1}^{n-1} y^2\partial_{x_i}F,\partial_{x_i}G}_2\Big|
+ |\inner{ y^2\partial_yF,\partial_yG}_2|
 +  2(1+\alpha_n)|\inner{ y\partial_yF,G}_2|\nonumber
  \\ & \hspace*{5cm}
  + |\inner{ (e_{n} - 2s_0)\mathcal{D}_HF,G}_2|
  + |\inner{ (\modulo{s}^2 - \beta_n)F,G}_2|,
	\end{align}
using Lemma \ref{lem_Properties} we have
\begin{align}\label{Eq:Sc_cont_hyper_TWO}
		|q_s(F,G)| &
\leq 2^{n/2}\sum_{i=1}^{n-1} \| y^2\partial_{x_i}F\|_2 \|\partial_{x_i}G\|_2
+ 2^{n/2}\|y^2\partial_yF\|_2 \|\partial_yG\|_2
 + 2^{1+n/2}(1+\alpha_n)\|y\partial_yF\|_2 \|G\|_2 \nonumber
  \\
  &
  +2^{n/2}\|(e_{n} - 2s_0)\mathcal{D}_HF\|_2 \| G\|_2
  + 2^{n/2}\|(\modulo{s}^2 - \beta_n)F\|_2 \| G\|_2
	\end{align}
and by Lemma \ref{STIMADF} we obtain
\begin{align}\label{Eq:Sc_cont_hyper_3}
		|q_s(F,G)| &
\leq 2^{n/2}M^2\sum_{i=1}^{n-1}\|\partial_{x_i}F\|_2 \|\partial_{x_i}G\|_2
+ 2^{n/2}M^2\|\partial_yF\|_2 \|\partial_yG\|_2
 + 2^{1+n/2}M(1+\alpha_n)\|\partial_yF\|_2 \|G\|_2 \nonumber
  \\
  &
  +2^{n/2}\sqrt{1+ 4s_0^2}\, \big(\sqrt{n}M\|F\|_D + \alpha_n\|F\|_2\big)\| G\|_2
  + 2^{n/2}\Big|\modulo{s}^2 - \beta_n\Big|\|F\|_2 \| G\|_2.
	\end{align}
Since
$\|\partial_{x_i}F\|_2 \leq \|F\|_D$ and $\|\partial_{y}F\|_2 \leq \|F\|_D$,
we get
\begin{align}\label{Eq:Sc_cont_hyper_4}
|q_s(F,G)| &
\leq  n 2^{n/2}M^2\|F\|_D\|G\|_D
 +  2^{1+n/2}M (1+\alpha_n)\|F\|_D \|G\|_2 \nonumber
  \\
  &
  +2^{n/2}\sqrt{n}M\sqrt{1+ 4s_0^2}\,\|F\|_D \| G\|_2
  +2^{1+n/2}\alpha_n\sqrt{1+ 4s_0^2}\, \|F\|_2\| G\|_2
  \\
  &
  + 2^{n/2}(\modulo{s}^2 + \beta_n)|\|F\|_2 \| G\|_2,		
	\end{align}
from which we deduce that there exists a positive constant
$C(s,n,M)$ such that
\begin{equation}\label{CONTIN}
|q_s(F,G)| \leq C(s,n,M)\|F\|_{H_0^1}\|G\|_{H_0^1},
\end{equation}
where we have set $\|F\|^2_{H_0^1}:= \norm{F}_2^2+\norm{F}_D^2$,
so the  sesquilinear form is continuous from $H^1_0(\Omega, \mathbb{R}_n) \times H^1_0(\Omega, \mathbb{R}_n)$ to $\mathbb{R}_n$.

\medskip
{\em Step 2}. For the coercivity of the scalar part $\Sc{q_s(F,F)}$ of $q_s(F,F)$ we use again the  inequalities
of Lemma \ref{lem_Properties} in this specific settings and the lower bound $m>0$ to obtain the inequalities
		\begin{align} \label{Eq:Coer_Hyper}
			\Sc{q_s(F,F)} & = {\rm Sc} \inner{ \sum_{i=1}^{n-1} y^2\partial_{x_i}F,\partial_{x_i}F}_2
+ {\rm Sc}\inner{ y^2\partial_yF,\partial_yF}_2
 + \textcolor{black}{ 2(1+\alpha_n)}{\rm Sc}\inner{ y\partial_yF,F}_2\nonumber
  \\
  &
  + {\rm Sc}\inner{ (e_{n} - 2s_0)\mathcal{D}_HF,F}_2
  + {\rm Sc}\inner{ (\modulo{s}^2 - \beta_n)F,F}_2\nonumber
\\
&
\geq m^2\norm{F}_D^2 - \textcolor{black}{ 2(1+\alpha_n)}\|y\partial_yF\|_2\| F\|_2\nonumber
-
\|(e_{n} - 2s_0)\mathcal{D}_HF\|_2\|F\|_2 + \Big|\modulo{s}^2 - \beta_n\Big|\norm{F}_2^2 \nonumber
\\
&
 \geq m^2\norm{F}_D^2 - \textcolor{black}{ 2(1+\alpha_n)}M\norm{F}_D\norm{F}_2 - \sqrt{1+ 4s_0^2}\,\norm{\mathcal{D}_HF}_2\norm{F}_2 + \Big|\modulo{s}^2 - \beta_n\Big|\norm{F}_2^2 .
 \end{align}
 So using Lemma \ref{STIMADF} inequality (\ref{Eq:Coer_Hyper})
  becomes
 \begin{align}
\Sc{q_s(F,F)}  &\geq m^2\|F\|_D^2 - \textcolor{black}{ 2(1+\alpha_n)}M\|F\|_D\|F\|_2\nonumber
\\
&
 - \sqrt{1+ 4s_0^2}\, \Big(\textcolor{black}{\sqrt{n}}M\|F\|_D + \alpha_n\|F\|_2\Big)\|F\|_2
  + ||s|^2 - \beta_n|\|F\|_2^2\nonumber
\end{align}
 and  some more computations
 \begin{align}
\Sc{q_s(F,F)}  &\geq m^2\|F\|_D^2 - \textcolor{black}{ 2(1+\alpha_n)}M\|F\|_D\|F\|_2\nonumber
\\
&
 - \textcolor{black}{\sqrt{n}}\sqrt{1+ 4s_0^2}\, M\|F\|_D \|F\|_2
 - \sqrt{1+ 4s_0^2}\,  \alpha_n\|F\|_2^2
  + \Big||s|^2 - \beta_n\Big|\|F\|_2^2\nonumber
\end{align}
so that
 \begin{align}\label{FINALSCQ}
\Sc{q_s(F,F)}  &\geq m^2\|F\|_D^2 - M\Big(\textcolor{black}{ 2(1+\alpha_n)}+\textcolor{black}{\sqrt{n}}\sqrt{1+ 4s_0^2}\Big)\|F\|_D\|F\|_2\nonumber
 \\
 &
  + \Big(\Big||s|^2 - \beta_n\Big| - \alpha_n\sqrt{1+ 4s_0^2}\Big)\|F\|_2^2.
\end{align}
 We use now Young's inequality, for  $\delta>0$,
		\begin{equation*}
			\norm{F}_D\norm{F}_2 \leq \frac{1}{2\delta}\norm{F}_D^2 + \frac{\delta}{2}\norm{F}_2^2,
		\end{equation*}
		to further estimates for (\ref{FINALSCQ}) give
		\begin{align} \label{Eq:Lower_bound_Hyper}
			\Sc{q_s(F,F)} &\geq
\p{m^2 -  M\Big(\textcolor{black}{ 2(1+\alpha_n)}+\textcolor{black}{\sqrt{n}}\sqrt{1+ 4s_0^2}\Big)\frac{1}{2\delta}}\norm{F}_D^2\nonumber
\\
&
+
\p{\Big||s|^2 - \beta_n\Big| - \alpha_n\sqrt{1+ 4s_0^2} -  M\Big(\textcolor{black}{ 2(1+\alpha_n)}+\textcolor{black}{\sqrt{n}}\sqrt{1+ 4s_0^2}\Big)\frac{\delta}{2}}\norm{F}_2^2.
		\end{align}
Let us choose $\delta>0$ such that the two coefficients of $\norm{F}_D^2$ and $\norm{F}_2^2$ be the same:

\begin{align}\label{EQUA_FOR_DELTA_DIR_HYP}
m^2 -  M\Big(\textcolor{black}{ 2(1+\alpha_n)}&+\textcolor{black}{\sqrt{n}}\sqrt{1+ 4s_0^2}\Big)\frac{1}{2\delta}\nonumber
\\
&=
\Big||s|^2 - \beta_n\Big| - \alpha_n\sqrt{1+ 4s_0^2} -  M\Big(\textcolor{black}{ 2(1+\alpha_n)}+\textcolor{black}{\sqrt{n}}\sqrt{1+ 4s_0^2}\Big)\frac{\delta}{2}
\end{align}
and also
\begin{align*}
M\Big(\textcolor{black}{ 2(1+\alpha_n)}+\textcolor{black}{\sqrt{n}}\sqrt{1+ 4s_0^2}\Big)\delta^2
&-2\delta\Big( \Big||s|^2 - \beta_n\Big| - \alpha_n\sqrt{1+ 4s_0^2}-m^2\Big)
\\
&
-  M\Big(\textcolor{black}{ 2(1+\alpha_n)}+\textcolor{black}{\sqrt{n}}\sqrt{1+ 4s_0^2}\Big)=0.
\end{align*}
Identifying coefficients
\begin{align*}
a = M\Big(\textcolor{black}{ 2(1+\alpha_n)}+\textcolor{black}{\sqrt{n}}\sqrt{1+ 4s_0^2}\Big),\ \ \
b = -2\Big( \Big||s|^2 - \beta_n\Big| - \alpha_n\sqrt{1+ 4s_0^2}-m^2\Big),\ \ \
c = -a,
\end{align*}
by applying the quadratic formula, so
$$
\frac{\Delta}{4}:=\frac{b^2 - 4ac}{4}
=\Big( \Big||s|^2 - \beta_n\Big| - \alpha_n\sqrt{1+ 4s_0^2}-m^2\Big)^2
+M^2\Big(\textcolor{black}{ 2(1+\alpha_n)}+\textcolor{black}{\sqrt{n}}\sqrt{1+ 4s_0^2}\Big)^2
$$
and observing that $\Delta/4$ is always positive, we obtain the two solutions
\begin{equation}\label{HYP_DELTAPIU}
\delta^\pm:=\frac{\Big( \Big||s|^2 - \beta_n\Big| - \alpha_n\sqrt{1+ 4s_0^2}-m^2\Big)\pm\sqrt{\Delta/4}}{M\Big(\textcolor{black}{ 2(1+\alpha_n)}+\textcolor{black}{\sqrt{n}}\sqrt{1+ 4s_0^2}\Big)}.
\end{equation}
We note that $\delta^-$ cannot be considered because it is negative, so we remain with

$$
M\Big(\textcolor{black}{ 2(1+\alpha_n)}+\textcolor{black}{\sqrt{n}}\sqrt{1+ 4s_0^2}\Big)\delta^+:=\Big( \Big||s|^2 - \beta_n\Big| - \alpha_n\sqrt{1+ 4s_0^2}-m^2\Big)+\sqrt{\Delta/4}.
$$
We consider the coefficient of $\norm{F}_2^2$ that for the value $\delta^+$ is equal to the coefficient of $\norm{F}_D^2$
so that the term
$$
K_{n,m,M}(s):=\Big||s|^2 - \beta_n\Big| - \alpha_n\sqrt{1+ 4s_0^2} -  M\Big(\textcolor{black}{ 2(1+\alpha_n)}+\textcolor{black}{\sqrt{n}}\sqrt{1+ 4s_0^2}\Big)\frac{\delta^+}{2}
$$
becomes
$$
K_{n,m,M}(s):=\Big||s|^2 - \beta_n\Big| - \alpha_n\sqrt{1+ 4s_0^2} - \frac{1}{2}
\Big[\Big( \Big||s|^2 - \beta_n\Big| - \alpha_n\sqrt{1+ 4s_0^2}-m^2\Big)+\sqrt{\Delta/4}\Big]
$$
so we finally get
$$
K_{n,m,M}(s):=\frac{1}{2}\Big||s|^2 - \beta_n\Big| -\frac{1}{2} \alpha_n\sqrt{1+ 4s_0^2}
+\frac{1}{2}m^2-\frac{1}{2}\sqrt{\Delta/4}
$$
and with the substitution of the explicit value of  $\Delta/4$ we have the explicit expression

\begin{align}\label{COERCONST}
K_{n,m,M}(s):&=\frac{1}{2}\Big||s|^2 - \beta_n\Big| -\frac{1}{2} \alpha_n\sqrt{1+ 4s_0^2}
+\frac{1}{2}m^2
\\
&
-\frac{1}{2}\sqrt{\Big( \Big||s|^2 - \beta_n\Big| - \alpha_n\sqrt{1+ 4s_0^2}-m^2\Big)^2
+M^2\Big(\textcolor{black}{ 2(1+\alpha_n)}+\textcolor{black}{\sqrt{n}}\sqrt{1+ 4s_0^2}\Big)^2}.
\end{align}
Now we have to impose that the coefficient $K_{n,m,M}(s)$ is positive. Now we observe that
 the
 inequality with the square root
\begin{align}\label{irrtineq}
\Big||s|^2 - \beta_n\Big|& - \alpha_n\sqrt{1+ 4s_0^2}
+m^2\nonumber
\\
&
>\sqrt{\Big( \Big||s|^2 - \beta_n\Big| - \alpha_n\sqrt{1+ 4s_0^2}-m^2\Big)^2
+M^2\Big(\textcolor{black}{ 2(1+\alpha_n)}+\textcolor{black}{\sqrt{n}}\sqrt{1+ 4s_0^2}\Big)^2}
\end{align}
is equivalent to the system of two inequalities
\begin{align}\label{QULI11}
\Big||s|^2 - \beta_n\Big| - \alpha_n\sqrt{1+ 4s_0^2}
+m^2
>0
\end{align}
and
\begin{align}\label{QULI2}
\Big(\Big||s|^2 - \beta_n\Big|& - \alpha_n\sqrt{1+ 4s_0^2}
+m^2\Big)^2
\\
&
>\Big( \Big||s|^2 - \beta_n\Big| - \alpha_n\sqrt{1+ 4s_0^2}-m^2\Big)^2
+M^2\Big(\textcolor{black}{ 2(1+\alpha_n)}+\textcolor{black}{\sqrt{n}}\sqrt{1+ 4s_0^2}\Big)^2,
\end{align}
since the term under the square root is $\Delta/4$ that is always positive.
With further simplifications, inequality (\ref{QULI2}) becomes
\textcolor{black}{
\begin{align*}
4m^2\Big||s|^2 - \beta_n\Big|-4nM^2 s_0^2
 - 4\Big(m^2\alpha_n +M^2(1+\alpha_n)\sqrt{n}\Big)\sqrt{1+4s_0^2}
 >M^2\Big(4(1+\alpha_n)^2+n\Big).
\end{align*}
}
So the coercive estimate (\ref{FINALSCQ}) becomes
\begin{align}
\Sc{q_s(F,F)}  &\geq K_{n,m,M}(s)(\|F\|_D^2+\|F\|_2^2)
\end{align}
with coercive constant $K_{n,m,M}(s)$ given by  (\ref{COERCONST}).

\medskip
{\em Step 3}.  Conclusions from continuity and coercivity conditions.
Hence we have proven that $\Sc{q_s}$ is coercive in $H_0^1(\Omega)\times H_0^1(\Omega)$. Fixing now any function $f\in L^2(\Omega)$, we can consider the corresponding functional
		\begin{equation} \label{Eq:test_funct_hyper}
			\varphi_f(G) := \inner{f,G}_{L^2}, \quad G\in H_0^1(\Omega).
		\end{equation}
		Then by Lemma \ref{lem_Properties} iii), this linear functional is bounded on $H_0^1(\Omega)$, i.e.,
		\begin{equation*}
			\modulo{\varphi_f(G)} = \modulo{\inner{f,G}_{L^2}} \leq 2^{\frac{n}{2}}\norm{f}_2\norm{G}_2 \leq 2^{\frac{n}{2}}\norm{f}_2\norm{G}_{H_0^1}, \quad G\in H_0^1(\Omega).
		\end{equation*}
		Hence, the assumptions of Lemma \ref{lem_Lax_Milgram} are satisfied and there exists a unique weak solution $F_f\in H_0^1(\Omega)$ such that
		\begin{equation*}
			q_s(F_f,G) = \varphi_f(G) = \inner{f,G}_2, \quad \text{for all } G\in H_0^1(\Omega).
		\end{equation*}
		For the $L^2$-estimate in (\ref{Eq:norm_est_hyper}), we test (\ref{Eq:test_funct_hyper}) with $G = F_f$. Using also Lemma \ref{lem_Properties} iv), this gives
		\begin{equation} \label{Eq:norm_sc_bound_hyper}
			\Sc{q_s(F_f,F_f)} = \Sc{\inner{f,F_f}_{L^2}} \leq \norm{f}_2\norm{F_f}_2.
		\end{equation}
		Combining this inequality with the coercivity estimate (\ref{Eq:Lower_bound_Hyper}),
so the coercive estimate (\ref{FINALSCQ}) becomes
\begin{align}
\Sc{q_s(F_f,F_f)}  &\geq K_{n,m,M}(s)(\|F_f\|_D^2+\|F_f\|_2^2)
\end{align}
with coercive constant $K_{n,m,M}(s)$ given by  (\ref{COERCONST}), furthermore gives
		\begin{equation*}
			K_{n,m,M}(s)(\|F_f\|_D^2+\|F_f\|_2^2)\leq \Sc{q_s(F_f,F_f)} \leq \norm{f}_2\norm{F_f}_2.
		\end{equation*}
Recall that $Q_s(\mathcal{D}_H)F = f$ so $F=Q_s(\mathcal{D}_H)^{-1}f$ we get the estimates for the pseudo $S$-resolvent
\begin{equation*}
			K_{n,m,M}(s)(\|Q_s(\mathcal{D}_H)^{-1}f\|_D^2+\|Q_s(\mathcal{D}_H)^{-1}f\|_2^2)\leq \norm{f}_2\norm{Q_s(D)^{-1}f}_2,
		\end{equation*}
from which we get
\begin{equation*}
			K_{n,m,M}(s)\|Q_s(\mathcal{D}_H)^{-1}f\|_2^2\leq \norm{f}_2\norm{Q_s(\mathcal{D}_H)^{-1}f}_2.
		\end{equation*}
and we obtain the estimate
\begin{equation}\label{UNODDD}
			\|Q_s(\mathcal{D}_H)^{-1}f\|_2\leq \frac{1}{K_{n,m,M}(s)}\, \norm{f}_2.
		\end{equation}
We also have
\begin{equation}\label{UNODDDGGG}
			K_{n,m,M}(s)\|Q_s(\mathcal{D}_H)^{-1}f\|_D^2\leq \norm{f}_2 \norm{Q_s(\mathcal{D}_H)^{-1}f}_2.
		\end{equation}
from which, by replacing (\ref{UNODDD}) in (\ref{UNODDDGGG}) we get
\begin{equation}\label{UNODDDGGGHHH}
			K_{n,m,M}(s)\|Q_s(\mathcal{D}_H)^{-1}f\|_D^2\leq  \frac{1}{K_{n,m,M}(s)}\norm{f}_2^2,		\end{equation}
we simplify and we get the other estimate.
\end{proof}

\begin{remark} The inequalities \eqref{CONDWITHPUTCP} define a non-empty set, indeed if we consider $s$ such that $s_0=0$ and $s_1$ is large, they are both satisfied.
Special choices of $n$, $m$, $M$ indicate that the second inequality in \eqref{CONDWITHPUTCP} may imply the first one. In other words, in some special cases the set of points satisfying the second inequality is contained in the set defined by the first inequality. A fine study of the solutions to \eqref{CONDWITHPUTCP} looks rather complicated in the general case and a case-by-case study seems to be more appropriate.
\end{remark}
We now show that the estimates for the pseudo $S$-resolvent operator $Q_s(\mathcal{D}_H)^{-1}$  directly yield the estimates for the right $S$-resolvent operator.
Similarly, we can apply these to obtain the same estimate for the left $S$-resolvent operator for unbounded operators, which has to be defined as
$$
S_L^{-1}(s,T):=Q_s(T)^{-1}\overline{s}-TQ_s(T)^{-1}
$$
in order to be defined on the entire space.

\begin{corollary}[$S$-resolvent estimate]\label{STIMASRESOL}
Let $\Omega\in \mathbb{R}^{n}_+$ be as in Remark \ref{DOMAINS}
		 and let $\mathcal{D}_H$ be the Dirac operator on the hyperbolic space as in (\ref{Eq:Dirac_Hyper}).
Let $s\in\BR^{n+1}$ be such that the two inequalities
(\ref{CONDWITHPUTCP}) hold.
Then, we have the estimate
$$
\|S_R^{-1}(s,\mathcal{D}_H)f\|_2\leq  \Big(|s|+\alpha_n+ M\sqrt{n} \Big)\frac{1}{K_{n,m,M}(s)}\,\,\norm{f}_2,
$$
where the constant $K_{n,m,M}(s)$ is given by
(\ref{KAPPAnm}).
\end{corollary}
\begin{proof}
We estimate directly the right $S$-resolvent operators
$$
S_R^{-1}(s,\mathcal{D}_H)f=(\overline{s}-\mathcal{D}_H)Q_s(\mathcal{D}_H)^{-1}f
$$
of the Dirac operator on the hyperbolic space:
$$
\mathcal{D}_H = \sum_{i=1}^{n-1} e_iy\partial_{x_i} - \alpha_ne_{n} + e_{n}y\partial_y.
$$
So we get the estimates
\begin{align}
\|S_R^{-1}(s,\mathcal{D}_H)f\|_2 &\leq \|\overline{s}Q_s(\mathcal{D}_H)^{-1}f\|_2+\|\mathcal{D}_HQ_s(\mathcal{D}_H)^{-1}f\|_2
\\
&
\leq|s|\|Q_s(\mathcal{D}_H)^{-1}f\|_2+\|(\sum_{i=1}^{n-1} e_iy\partial_{x_i} - \alpha_ne_{n} + e_{n}y\partial_y)Q_s(\mathcal{D}_H)^{-1}f\|_2
\\
&
\leq|s|\|Q_s(\mathcal{D}_H)^{-1}f\|_2+\|(\sum_{i=1}^{n-1} e_iy\partial_{x_i} + e_{n}y\partial_y)Q_s(\mathcal{D}_H)^{-1}f\|_2
+\| \alpha_ne_{n}Q_s(\mathcal{D}_H)^{-1}f\|_2
\\
&
\leq |s|\|Q_s(\mathcal{D}_H)^{-1}f\|_2+
M\sqrt{n}\|Q_s(\mathcal{D}_H)^{-1}f\|_D
+\alpha_n\| Q_s(\mathcal{D}_H)^{-1}f\|_2
\end{align}
where we have used Lemma \ref{STIMADF} replacing $F$ by $Q_s(\mathcal{D}_H)^{-1}f$.
Now we use the estimates for the pseudo $S$-resolvent operator (\ref{Eq:norm_est_hyper}), precisely we replace the estimates
$$
\|Q_s(\mathcal{D}_H)^{-1}f\|_2\leq \frac{1}{K_{n,m,M}(s)}\, \norm{f}_2,\ \ \
\|Q_s(\mathcal{D}_H)^{-1}f\|_D\leq \frac{1 }{K_{n,m,M}(s)}\,\norm{f}_2,
$$
where $C_P(\Omega)$ is the Poincaré constant and the constant $K_{n,m,M}(s)$ is given by
(\ref{KAPPAnm}), so we get
\begin{align}
\|S_R^{-1}(s,\mathcal{D}_H)f\|_2&\leq
 (|s|+\alpha_n)\|Q_s(\mathcal{D}_H)^{-1}f\|_2+
M\sqrt{n}\|Q_s(\mathcal{D}_H)^{-1}f\|_D
\\
&
\leq (|s|+\alpha_n)\frac{1}{K_{n,m,M}(s)}\, \norm{f}_2+
M\sqrt{n}\frac{1 }{K_{n,m,M}(s)}\,\norm{f}_2
\\
&
\leq \Big(|s|+\alpha_n+ M\sqrt{n} \Big)\frac{1}{K_{n,m,M}(s)}\,\,\norm{f}_2.
\end{align}
\end{proof}

\subsection{Using Poincar\'e inequality}

As we have already discussed, Poincar\'e inequality is necessary for the definition of the norm in $H_0^1$
 when the term $\modulo{s}^2 - \beta_n$ is zero because we do not have equivalence of the norm
$\|F\|_{H_0^1(\Omega, \mathbb{R}_n)}$ defined just by the term $\norm{F}_D^2$
and the norm
 $\|F\|_{H_0^1(\Omega, \mathbb{R}_n)}:= \norm{F}_2^2+\norm{F}_D^2 $.

So we use the Poincar\'e inequality in Lemma \ref{lem_Poincare_inequality}
that holds when  $\Omega\subset\mathbb{R}^n$ is an open set bounded in one direction. Then there exists a constant $C_P(\Omega)>0$, such that
\begin{equation}
\Vert F\Vert_{2}\leq C_P(\Omega)\Vert F\Vert_D,\qquad F\in H_0^1(\Omega).
\end{equation}

\begin{theorem} \label{Th:Main_HyperPOINC}
		Let $\Omega\in \mathbb{R}^{n}_+$ be as in Remark \ref{DOMAINS}
and let $\mathcal{D}_H$ be the Dirac operator on the hyperbolic space as in (\ref{Eq:Dirac_Hyper}). Let $s\in\BR^{n+1}$ be such that
\begin{align}\label{KPnmMs}
K_{P,n,m,M}(s):&=
m^2 - C_PM\Big({ 2(1+\alpha_n)}+\textcolor{black}{\sqrt{n}}\sqrt{1+ 4s_0^2}\Big)\nonumber
 \\
 &
  - \Big|\Big||s|^2 - \beta_n\Big| - \alpha_n\sqrt{1+ 4s_0^2}\Big|C_P^2>0,
\end{align}
where $\beta_n:=\alpha_n+ \alpha_n^2$ and $\alpha_n$ is the constant given in (\ref{alphan}),  and the constants $M>m>0$ are explicitly given by (\ref{CONTnM1}) and $C_P(\Omega)$ Poincar\'e constant of $\Omega$.
		Then, for every $f\in L^2(\Omega)$ there exists a unique $F_f\in H_0^1(\Omega)$ such that
		\begin{equation}
			q_s(F_f,G) = \inner{f,G}_2, \qquad \text{for all } G\in H_0^1(\Omega)
		\end{equation}
where $q_s(F_f,G)$ is the sesquilinear form given by (\ref{FORMWITHSCALPROD}) associated with the Dirichlet problem.
		Moreover, this solution satisfies the bounds
		\begin{equation} \label{Eq:norm_est_hyper2}
			\|Q_s(\mathcal{D}_H)^{-1}f\|_D\leq \frac{C_P(\Omega)}{K_{P,n,m,M}(s)}\|f\|_2,
		\ \ \ {\rm and}\ \
			\Vert Q_s(\mathcal{D}_H)^{-1}f\Vert_{L^2}\leq \frac{(C_P(\Omega))^2}{K_{P,n,m,M}(s)}\|f\|_2,.
		\end{equation}
	\end{theorem}
	\begin{proof}
Consider $q_s(F,G)$ is the sesquilinear form given by \eqref{FINALSCQ} associated with the Dirichlet problem.
The continuity is simple we study the coercivity. We have
\begin{align}
\Sc{q_s(F,F)}  &\geq m^2\|F\|_D^2 - M\Big(\textcolor{black}{ 2(1+\alpha_n)}+\textcolor{black}{\sqrt{n}}\sqrt{1+ 4s_0^2}\Big)\|F\|_D\|F\|_2\nonumber
 \\
 &
  + \Big(\Big||s|^2 - \beta_n\Big| - \alpha_n\sqrt{1+ 4s_0^2}\Big)\|F\|_2^2
  \end{align}
 which gives
  \begin{align}
\Sc{q_s(F,F)}  &\geq m^2\|F\|_D^2 - C_PM\Big({ 2(1+\alpha_n)}+\textcolor{black}{\sqrt{n}}\sqrt{1+ 4s_0^2}\Big)\|F\|_D^2\nonumber
 \\
 &
  - \Big|\Big||s|^2 - \beta_n\Big| - \alpha_n\sqrt{1+ 4s_0^2}\Big|C_P^2\Vert F\Vert_D^2
\end{align}
and also
\begin{align}
\Sc{q_s(F,F)}  &\geq \Big(m^2 - C_PM\Big({ 2(1+\alpha_n)}+\textcolor{black}{\sqrt{n}}\sqrt{1+ 4s_0^2}\Big)\nonumber
 \\
 &
  - \Big|\Big||s|^2 - \beta_n\Big| - \alpha_n\sqrt{1+ 4s_0^2}\Big|C_P^2\Big)\Vert F\Vert_D^2
  \\
  &
  =
K_{P,n,m,M}(s)\|F\|_D^2,
\end{align}
where $K_{P,n,m,M}(s)$  is as in (\ref{KPnmMs}).
For the estimates we reason as in the previous theorem.
The we observe that from
$$
q_s(F_f,G) = \inner{f,G}_2, \qquad \text{for all } G\in H^1(\Omega),
$$
setting we get $G=F_f$ we have
$$
K_{P,n,m,M}(s)\|F\|_D^2\leq \Sc{q_s(F_f,F_f)}= \Sc{ \inner{f,G}_2} \leq \|f\|_2\|F_f\|_2
$$
and recalling that
 $Q_s(\mathcal{D}_H)F = f$ so $F=Q_s(\mathcal{D}_H)^{-1}f$ we obtain
 $$
K_{P,n,m,M}(s)\|Q_s(\mathcal{D}_H)^{-1}f\|_D^2\leq \|f\|_2\|Q_s(\mathcal{D}_H)^{-1}f\|_2
$$
 because of Poincar\'e inequality in Lemma \ref{lem_Poincare_inequality}
 there exists a constant $C_P(\Omega)>0$, such that
\begin{equation*}
\Vert Q_s(\mathcal{D}_H)^{-1}f\Vert_{L^2}\leq C_P(\Omega)\Vert Q_s(\mathcal{D}_H)^{-1}f\Vert_D,
\end{equation*}
so that
$$
K_{P,n,m,M}(s)\|Q_s(\mathcal{D}_H)^{-1}f\|_D^2\leq C_P(\Omega)\|f\|_2\Vert Q_s(\mathcal{D}_H)^{-1}f\Vert_D
$$
simplifying we have the first inequality in (\ref{Eq:norm_est_hyper2})
$$
\|Q_s(\mathcal{D}_H)^{-1}f\|_D\leq \frac{C_P(\Omega)}{K_{P,n,m,M}(s)}\|f\|_2
$$
and  we also have using again Poincar\'e inequality, we got the second inequality in (\ref{Eq:norm_est_hyper2})
\begin{equation*}
\Vert Q_s(\mathcal{D}_H)^{-1}f\Vert_{L^2}\leq \frac{(C_P(\Omega))^2}{K_{P,n,m,M}(s)}\|f\|_2.
\end{equation*}
\end{proof}
	
\begin{remark}
Observe that the inequalities (\ref{CONDWITHPUTCP}) that define part of the $S$-resolvent set contain only explicit constants, such as $\beta_n := \alpha_n + \alpha_n^2$ and $\alpha_n$, which are specified in (\ref{alphan}). The constants $M > m > 0$ are also explicitly provided, despite the need to exclude the points on the circle where $|s|^2 - \beta_n = 0$. In condition (\ref{KPnmMs}), the restriction $|s|^2 - \beta_n \neq 0$ is removed, but the inequality that defines part of the $S$-resolvent set still includes the Poincaré constant $C_P(\Omega)$ , which is not explicitly known for most domains.
\end{remark}

\begin{remark}
Repeating the consideration done in
Corollary \ref{STIMASRESOL} we can deduce also the estimates for the $S$-resolvent operator.
We omit the computations since they are very similar.
\end{remark}

\section{Dirac operator on the hyperbolic space with
Robin-like boundary conditions}\label{sec_Robin_HYPERBOLIC}

In this section we study the weak formulation of the spectral problem  associated with  the equation $Q_s(\mathcal{D}_H)F = f$ under suitable
		 Robin-like boundary conditions for the Dirac operator $\mathcal{D}_H$ defined in (\ref{Eq:Dirac_Hyper}).
To understand what are the natural Robin-like boundary conditions we mentioned, in Remark \ref{INTBYPARTFORM},
 that the integration by part formula associated with the operator $A$ defined by
$$
AF:=-\sum_{i=1}^n a_i\partial_{x_i}\Big(a_i\partial_{x_i}F\Big)
$$
is given by
\begin{align*}
-\sum_{i=1}^n\Big\langle a_i\partial_{x_i}
\Big(a_i\partial_{x_i}F\Big),G\Big\rangle_{L^2(\Omega, \mathbb{R}_n)}
&=\sum_{i=1}^n\Big\langle a_i\partial_{x_i}F,\partial_{x_i}(a_iG)\Big\rangle_{L^2(\Omega, \mathbb{R}_n)}
\\
&
-\sum_{i=1}^n\Big\langle\nu_ia_i
\partial_{x_i}F,a_iG\Big\rangle_{L^2(\partial\Omega, \mathbb{R}_n)},
\end{align*}
where the coefficients $a_i(x)$ are smooth functions real-valued of the variables $x=(x_1,...,x_n)\in \Omega\subset \mathbb{R}^{n}$ and we denote by $\vec{\nu}=(\nu_1,\dots,\nu_n)$ outer unit normal vector to the boundary smooth boundary  $\partial\Omega$ of $\Omega$.
The type of Robin boundary conditions turn out to be
\begin{equation}\label{Eq_Robin_boundary_conditions}
\sum_{i=1}^n\nu_ia_i^2\partial_{x_i} F(x)+bF(x)=0,\qquad\text{on }\partial\Omega,
\end{equation}
with a natural choice of regularity for the real valued function $b\in L^\infty(\partial\Omega)$.

\medskip
In order to obtain the weak formulation of the boundary value problem
in the hyperbolic space we integrate by parts the second order terms of $\mathcal{D}_H^2$, given in
Lemma \ref{HYPDDUE}, for $F,G\in C^\infty(\Omega,\mathbb{R}_n)$.
In this case the boundary condition play a role and explicitly,
for the term $-\sum_{i=1}^{n-1}\int_{\Omega} y^2\partial_{x_i}^2\overline{F}G \dif x$,
we have
\begin{align*}
	-\sum_{i=1}^{n-1}\int_{\Omega} y^2\partial_{x_i}^2\overline{F}G \dif x&=	-\sum_{i=1}^{n-1}\int_{\Omega} y\partial_{x_i} (y\partial_{x_i}\overline{F})G \dif x
\\
&
=
\sum_{i=1}^{n-1}\int_{\Omega} y^2\partial_{x_i}\overline{F}\partial_{x_i}G \dif x
-\sum_{i=1}^{n-1}\int_{\partial\Omega} \nu_i y\partial_y\overline{F} \ (yG) d\Gamma
	\end{align*}
where $d\Gamma$ is the infinitesimal boundary of $\partial \Omega$.
Moreover, when we consider the term $- \int_{\Omega} y^2\partial_y^2\overline{F}G \dif x $ we have
\begin{align*}
		- \int_{\Omega} y^2\partial_y^2\overline{F}G \dif x & =- \int_{\Omega} \partial_y (\partial_y\overline{F}) (y^2G) \dif x
\\
&
=\int_{\Omega} \partial_y\overline{F} \partial_y(y^2G) \dif x
-\int_{\partial\Omega} \nu_n \partial_y\overline{F} \ (y^2G) d\Gamma
\\
&
 =\int_{\Omega} y^2\partial_y\overline{F}\partial_yG \dif x  + \int_{\Omega} 2y\partial_y\overline{F}G \dif x
 -\int_{\partial\Omega} \nu_n \partial_y\overline{F} \ y^2G d\Gamma.
	\end{align*}
Finally, taking into account the two above relations,  we get
	\begin{align*}
		-\sum_{i=1}^{n-1}\int_{\Omega} y^2\partial_{x_i}^2\overline{F}G \dif x& - \int_{\Omega} y^2\partial_y^2\overline{F}G \dif x
\\
&
=
\sum_{i=1}^{n-1}\int_{\Omega} y^2\partial_{x_i}\overline{F}\partial_{x_i}G \dif x
+ \int_{\Omega} y^2\partial_y\overline{F}\partial_yG \dif x  + \int_{\Omega} 2y\partial_y\overline{F}G \dif x
\\
&
-\int_{\partial\Omega} \Big(\sum_{i=1}^{n-1}\nu_iy^2\partial x_i\overline{F}+\nu_ny^2\partial_y\overline{F}\Big)  \ G d\Gamma.
	\end{align*}
Thus the Robin-type boundary value problem that we investigate in this section is:

\begin{equation}\label{Eq_Robin_boundary_conditions HYPERB}
\begin{cases}
Q_s(\mathcal{D}_H)F = f \qquad\text{on } \Omega,
&
\\
\sum_{i=1}^{n-1}\nu_iy^2\partial_{x_i}F+\nu_ny^2\partial_yF +bF=0,\qquad\text{on }\partial\Omega
\end{cases}
\end{equation}
with  $b\in L^\infty(\partial\Omega)$ which is a given function.

The  sesquilinear form $q_s(F,G)$, defined in (\ref{THEFORM}), has to be modified
taking into account the Robin-like boundary conditions in the hyperbolic space, as in (\ref{Eq_Robin_boundary_conditions HYPERB}). The form becomes
\begin{align*}
q^R_s(F,G)&=q_s(F,G)-\int_{\partial\Omega} \Big(\sum_{i=1}^{n-1}\nu_iy^2\partial x_i\overline{F}+\nu_ny^2\partial_y\overline{F}\Big)  \ G d\Gamma.
\\
&
=q_s(F,G)+\langle b\,\tau_DF,\tau_DG\rangle_{L^2(\partial\Omega)},
\end{align*}
defined on $\textup{dom }(q_s^R):=H^1(\Omega)\times H^1(\Omega)$.
The operator $\tau_{D}:H^1(\Omega)\rightarrow L^2(\partial\Omega)$
is the bounded Dirichlet trace operator, see \cite[Equation (4)]{M87}.

\begin{problem}
Let $\mathcal{D}_H$ be the Dirac operator in (\ref{Eq:Dirac_Hyper}), let $\beta_n$ be the constants in (\ref{BETAN}) and
	using the scalar product $\inner{F,G}_2$ in (\ref{SCLAL2}) we define the Robin  sesquilinear form $q^R_s(F,G)$ as
\begin{align}\label{FORMWITHSCALPROD1}
q^R_s(F,G):=q_s(F,G)+\langle b\,\tau_DF,\tau_DG\rangle_{L^2(\partial\Omega)}
\end{align}
with $\textup{dom } (q_s^R):= H^1(\Omega, \mathbb{R}_n) \times H^1(\Omega, \mathbb{R}_n)$
where $q_s(F,G)$ defined in (\ref{THEFORM}) and the boundary term $b$ was introduced in (\ref{Eq_Robin_boundary_conditions HYPERB}).
Show that for some values of the spectral parameter
$s\in\mathbb{R}^{n+1},$ for every $f\in L^2(\Omega, \mathbb{R}_n)$ there exists a unique solution $F_f\in H^1(\Omega, \mathbb{R}_n)$ such that
	\begin{equation}\label{Eq_Solution_Robin HYPER}
		q_s^R(F_f,G) = \inner{f,G}_2, \quad \text{for all} \ G\in H^1(\Omega, \mathbb{R}_n).
	\end{equation}
	Furthermore, determine $L^2$- and $D$-estimates of $F_f$, depending on the parameter $s\in \mathbb{R}^{n+1}$.
\end{problem}

\begin{theorem} \label{THFORTHEHYPERB}
		Let $\Omega\subset \mathbb{R}^{n}_+$ be a bounded set with smooth boundary and let $\mathcal{D}_H$ be the Dirac operator in (\ref{Eq:Dirac_Hyper}).
Let us consider $\Omega$ for which the constants  defined by
\begin{equation*}
\inf_{x\in\Omega}|x|=m,\ \ \ \ \sup_{x\in\Omega}|x|=M,
\end{equation*}
are such that $0<m<M<+\infty$,
and set
\begin{align}
\Lambda_{n,m,b}(p):
=pm^2-2^{\frac{n}{2}-1}\Vert b\Vert_{L^\infty(\partial\Omega)}\Vert\tau_D\Vert^2 .
\end{align}
Let $s\in\BR^{n+1}$ be such that
\begin{align}\label{QULI1}
\Big||s|^2 - \beta_n\Big| - \alpha_n\sqrt{1+ 4s_0^2}
+m^2-2^{\frac{n}{2}-1}\Vert b\Vert_{L^\infty(\partial\Omega)}\Vert\tau_D\Vert^2
>0
\end{align}
and
\begin{align}\label{SECODESTROBPRIMEH}
2\Lambda_{n,m,b}(2)
\Big||s|^2 - \beta_n\Big|-4n{\color{black} M^2}s_0^2\nonumber
&-2\Big[ \alpha_n\Lambda_{n,m,b}(2)+2M^2(1+\alpha_n)\sqrt{n}\Big]\sqrt{1+ 4s_0^2}\nonumber
\\
&
>m^4+{\color{black} M^2(n+4(1+\alpha_n)^2)}-\Lambda_{n,m,b}(1)^2.
\end{align}

Then, for every $f\in L^2(\Omega)$ there exists a unique $F_f\in H^1(\Omega)$ such that
		\begin{equation}
			q^R_s(F_f,G) = \inner{f,G}_2, \qquad \text{for all } G\in H^1(\Omega).
		\end{equation}
		Moreover, this solution satisfies the bounds
		\begin{equation} \label{Eq:norm_est_hyper6}
			\|Q_s(\mathcal{D}_H)^{-1}f\|_2\leq \frac{1}{K^R_{n,m,M}(s)}\, \norm{f}_2,
		\ \ \ {\rm and}\ \
			\|Q_s(\mathcal{D}_H)^{-1}f\|_D\leq \frac{1}{K^R_{n,m,M}(s)} \,\norm{f}_2,
		\end{equation}
where
\begin{align}\label{COERCONST1}
K^R_{n,m,M}(s) & :=\frac{1}{2}\Big||s|^2 - \beta_n\Big| -\frac{1}{2} \alpha_n\sqrt{1+ 4s_0^2}
+\frac{1}{2}m^2
-\frac{1}{2} 2^{\frac{n}{2}-1}\Vert b\Vert_{L^\infty(\partial\Omega)}\Vert\tau_D\Vert^2\nonumber
\\
&
-\frac{1}{2}\sqrt{\Big( \Big||s|^2 - \beta_n\Big| - \alpha_n\sqrt{1+ 4s_0^2}-m^2\Big)^2
+M^2\Big(\textcolor{black}{ 2(1+\alpha_n)}+\textcolor{black}{\sqrt{n}}\sqrt{1+ 4s_0^2}\Big)^2}.
\end{align}

	\end{theorem}

\begin{proof}
 We will verify that the  sesquilinear form $q_s^R$, defined in \eqref{FORMWITHSCALPROD1},  satisfies the assumptions of the Lemma~\ref{lem_Lax_Milgram}.
We start by recalling that the boundary term in \eqref{FORMWITHSCALPROD1} is estimated by
\begin{align}\label{ESTIM_bound_TERMTAU}
|\langle b\,\tau_DF,\tau_DG\rangle_{L^2(\partial\Omega)}|&\leq 2^{\frac{n}{2}}\Vert b\Vert_{L^\infty(\partial\Omega)}\Vert\tau_DF
\Vert_{L^2(\partial\Omega)}\Vert\tau_DG\Vert_{L^2(\partial\Omega)}
\\
&
\leq 2^{\frac{n}{2}}\Vert b\Vert_{L^\infty(\partial\Omega)}\Vert\tau_D\Vert^2\Vert F\Vert_{H^1}\Vert G\Vert_{H^1}.\nonumber
\end{align}
{\em Step 1.}
For the continuity we observe that the sesquilinear form $q^R_s(F,G)$
in (\ref{FORMWITHSCALPROD1}) is the sum of two terms
the sesquilinear form $q_s(F,G)$ and the boundary term $\langle b\,\tau_DF,\tau_DG\rangle_{L^2(\partial\Omega)}$. The estimate for
$q_s(F,G)$ has already been obtained
(\ref{CONTIN}), and it holds for the space $H^1(\Omega)$.
So it follows that
\begin{equation}\label{CONTINHUNO}
|q_s(F,G)| \leq C(s,n,M)\|F\|_{H^1}\|G\|_{H^1},
\end{equation}
and from (\ref{FORMWITHSCALPROD1}) we have
\begin{align*}	
|q^R_s(F,G)|\leq |q_s(F,G)|&+|\langle b\,\tau_DF,\tau_DG\rangle_{L^2(\partial\Omega)}|
\\
&
\leq\Big(C(s,n,M) +2^{\frac{n}{2}}\Vert b\Vert_{L^\infty(\partial\Omega)}\Vert\tau_D\Vert^2\Big)\|F\|_{H^1}\|G\|_{H^1}.
\end{align*}

\medskip
{\em Step 2.}
For the coercivity we take advantage of the inequality (\ref{Eq:Lower_bound_Hyper}) already obtained  for $\Sc{q_s(F,F)}$
and the estimate of the boundary term
 in \eqref{Eq_Solution_Robin HYPER} by
\begin{align*}
|\langle b\,\tau_DF,\tau_DF\rangle_{L^2(\partial\Omega)}|&\leq 2^{\frac{n}{2}}\Vert b\Vert_{L^\infty(\partial\Omega)}\Vert\tau_D\Vert^2\Vert F\Vert_{H^1}^2\nonumber
\\
&
=
2^{\frac{n}{2}}\Vert b\Vert_{L^\infty(\partial\Omega)}\Vert\tau_D\Vert^2\Big[\Vert F\Vert_{D}^2+\Vert F\Vert_{L^2}^2\Big].
\end{align*}
So the coercivity condition  for $q^R_s(F,F)$ becomes
\begin{align}	\label{Eq:Lower_bound_HyperROB}
\Sc{q^R_s(F,F)}&\geq \Sc{q_s(F,F)}- \Sc{\langle b\,\tau_DF,\tau_DF\rangle_{L^2(\partial\Omega)}}
\nonumber
\\
&
\geq
\p{m^2
-  M\Big(\textcolor{black}{ 2(1+\alpha_n)}+\textcolor{black}{\sqrt{n}}\sqrt{1+ 4s_0^2}\Big)\frac{1}{2\delta}}\norm{F}_D^2\nonumber
\\
&
+
\p{\Big||s|^2 - \beta_n\Big| -
\alpha_n\sqrt{1+ 4s_0^2} -
 M\Big(\textcolor{black}{ 2(1+\alpha_n)}+\textcolor{black}{\sqrt{n}}\sqrt{1+ 4s_0^2}\Big)\frac{\delta}{2}}\norm{F}_2^2\nonumber
 \\
 &
 -2^{\frac{n}{2}}\Vert b\Vert_{L^\infty(\partial\Omega)}\Vert\tau_D\Vert^2\Big[\Vert F\Vert_{D}^2+\Vert F\Vert_{L^2}^2\Big],
\end{align}
and also we have
to further elaborate (\ref{Eq:Lower_bound_HyperROB}) collecting some terms
\begin{align}	\label{Eq:Lower_bound_HyperROBDUE}
&\Sc{q^R_s(F,F)}\geq \Sc{q_s(F,F)}- \Sc{\langle b\,\tau_DF,\tau_DF\rangle_{L^2(\partial\Omega)}}\nonumber
\\
&
\geq
\p{m^2-2^{\frac{n}{2}}\Vert b\Vert_{L^\infty(\partial\Omega)}\Vert\tau_D\Vert^2 -  M\Big(\textcolor{black}{ 2(1+\alpha_n)}+\textcolor{black}{\sqrt{n}}\sqrt{1+ 4s_0^2}\Big)\frac{1}{2\delta}}\norm{F}_D^2\nonumber
\\
&
+
\p{\Big||s|^2 - \beta_n\Big| -2^{\frac{n}{2}}\Vert b\Vert_{L^\infty(\partial\Omega)}\Vert\tau_D\Vert^2- \alpha_n\sqrt{1+ 4s_0^2} -  M\Big(\textcolor{black}{ 2(1+\alpha_n)}+\textcolor{black}{\sqrt{n}}\sqrt{1+ 4s_0^2}\Big)\frac{\delta}{2}}\norm{F}_2^2.
\end{align}

Let us choose $\delta>0$ such that the two coefficients of $\norm{F}_D^2$ and $\norm{F}_2^2$ are equal and observe that we get the same relation as in
(\ref{EQUA_FOR_DELTA_DIR_HYP}) that is
\begin{equation}
m^2 -  M\Big(\textcolor{black}{ 2(1+\alpha_n)}+\textcolor{black}{\sqrt{n}}\sqrt{1+ 4s_0^2}\Big)\frac{1}{2\delta}=
\Big||s|^2 - \beta_n\Big| - \alpha_n\sqrt{1+ 4s_0^2} -  M\Big(\textcolor{black}{ 2(1+\alpha_n)}+\textcolor{black}{\sqrt{n}}\sqrt{1+ 4s_0^2}\Big)\frac{\delta}{2}
\end{equation}
because the terms $2^{\frac{n}{2}}\Vert b\Vert_{L^\infty(\partial\Omega)}\Vert\tau_D\Vert^2$ appear in both hand sides. When we
 consider the coefficient of $\norm{F}_2^2$ that for the value $\delta^+$ (given in (\ref{HYP_DELTAPIU})) is equal to the coefficient of $\norm{F}_D^2$ in the term
 $K^R_{n,m,M}(s)$ given by
$$
K^R_{n,m,M}(s):=\Big||s|^2 - \beta_n\Big|-2^{\frac{n}{2}}\Vert b\Vert_{L^\infty(\partial\Omega)}\Vert\tau_D\Vert^2 - \alpha_n\sqrt{1+ 4s_0^2} -  M\Big(\textcolor{black}{ 2(1+\alpha_n)}+\textcolor{black}{\sqrt{n}}\sqrt{1+ 4s_0^2}\Big)\frac{\delta^+}{2},
$$
it clearly appears the contribution of the boundary term
\begin{align*}
K^R_{n,m,M}(s):&=\Big||s|^2 - \beta_n\Big|-2^{\frac{n}{2}}\Vert b\Vert_{L^\infty(\partial\Omega)}\Vert\tau_D\Vert^2 - \alpha_n\sqrt{1+ 4s_0^2} \nonumber
\\
&
- \frac{1}{2}
\Big[\Big( \Big||s|^2 - \beta_n\Big| - \alpha_n\sqrt{1+ 4s_0^2}-m^2\Big)+\sqrt{\Delta/4}\Big] .
\end{align*}
We finally get
$$
K^R_{n,m,M}(s):=\frac{1}{2}\Big||s|^2 - \beta_n\Big|-2^{\frac{n}{2}}\Vert b\Vert_{L^\infty(\partial\Omega)}\Vert\tau_D\Vert^2 -\frac{1}{2} \alpha_n\sqrt{1+ 4s_0^2}
+\frac{1}{2}m^2-\frac{1}{2}\sqrt{\Delta/4}
$$
and with the substitution of the explicit value of  $\Delta/4$ we obtain the explicit expression

\begin{align}\label{COERCONST2}
K^R_{n,m,M}(s):&=\frac{1}{2}\Big||s|^2 - \beta_n\Big| -\frac{1}{2} \alpha_n\sqrt{1+ 4s_0^2}
+\frac{1}{2}m^2
-\frac{1}{2} 2^{\frac{n}{2}-1}\Vert b\Vert_{L^\infty(\partial\Omega)}\Vert\tau_D\Vert^2\nonumber
\\
&
-\frac{1}{2}\sqrt{\Big( \Big||s|^2 - \beta_n\Big| - \alpha_n\sqrt{1+ 4s_0^2}-m^2\Big)^2
+M^2\Big(\textcolor{black}{ 2(1+\alpha_n)}+\textcolor{black}{\sqrt{n}}\sqrt{1+ 4s_0^2}\Big)^2}.
\end{align}
Now we have to impose that the coefficient $K^R_{n,m,M}(s)$ is positive. We observe that
 the
 inequality with the square root
\begin{align}
\Big||s|^2 - \beta_n\Big|& - \alpha_n\sqrt{1+ 4s_0^2}+m^2
-2^{\frac{n}{2}-1}\Vert b\Vert_{L^\infty(\partial\Omega)}\Vert\tau_D\Vert^2\nonumber
\\
&
>\sqrt{\Big( \Big||s|^2 - \beta_n\Big| - \alpha_n\sqrt{1+ 4s_0^2}-m^2\Big)^2
+M^2\Big(\textcolor{black}{ 2(1+\alpha_n)}+\textcolor{black}{\sqrt{n}}\sqrt{1+ 4s_0^2}\Big)^2}
\end{align}
is equivalent to the system of two inequalities
\begin{align*}
\Big||s|^2 - \beta_n\Big| - \alpha_n\sqrt{1+ 4s_0^2}
+m^2-2^{\frac{n}{2}-1}\Vert b\Vert_{L^\infty(\partial\Omega)}\Vert\tau_D\Vert^2
>0
\end{align*}
which is (\ref{QULI1})
and
\begin{align}\label{QULI21}
\Big(\Big||s|^2 - \beta_n\Big|& - \alpha_n\sqrt{1+ 4s_0^2}
+m^2-2^{\frac{n}{2}-1}\Vert b\Vert_{L^\infty(\partial\Omega)}\Vert\tau_D\Vert^2\Big)^2\nonumber
\\
&
>\Big( \Big||s|^2 - \beta_n\Big| - \alpha_n\sqrt{1+ 4s_0^2}-m^2\Big)^2
+M^2\Big(\textcolor{black}{ 2(1+\alpha_n)}+\textcolor{black}{\sqrt{n}}\sqrt{1+ 4s_0^2}\Big)^2,
\end{align}
since the term under the square root is $\Delta/4$ that is always positive.
With simplifications, the inequality (\ref{QULI21}) becomes
\begin{align}\label{SECODESTROB}
&
2\Big(2m^2-2^{\frac{n}{2}-1}\Vert b\Vert_{L^\infty(\partial\Omega)}\Vert\tau_D\Vert^2\Big)
\Big||s|^2 - \beta_n\Big|-4n{\color{black} M^2}s_0^2\nonumber
\\
&
-2\Big[ \alpha_n\Big(2m^2-2^{\frac{n}{2}-1}\Vert b\Vert_{L^\infty(\partial\Omega)}\Vert\tau_D\Vert^2\Big)+2M^2(1+\alpha_n)\sqrt{n}\Big]\sqrt{1+ 4s_0^2}\nonumber
\\
&
>m^4+M^2(n+4{\color{black} (1+\alpha_n)^2})-\Big(m^2-2^{\frac{n}{2}-1}\Vert b\Vert_{L^\infty(\partial\Omega)}\Vert\tau_D\Vert^2\Big)^2.
\end{align}
Defining the function
\begin{align}
\Lambda_{n,m,b}(p):=pm^2-2^{\frac{n}{2}-1}\Vert b\Vert_{L^\infty(\partial\Omega)}\Vert\tau_D\Vert^2
\end{align}
the estimate (\ref{SECODESTROB}) can be written in a more compact way as
\begin{align}\label{SECODESTROBPRIME}
2\Lambda_{n,m,b}(2)
\Big||s|^2 - \beta_n\Big|-4n{\color{black} M^2}s_0^2\nonumber
&-2\Big[ \alpha_n\Lambda_{n,m,b}(2)+2M^2(1+\alpha_n)\sqrt{n}\Big]\sqrt{1+ 4s_0^2}\nonumber
\\
&
>m^4+{\color{black} M^2(n+4(1+\alpha_n)^2)}-\Lambda_{n,m,b}(1)^2.
\end{align}
The above discussion shows that the coercive estimate (\ref{FINALSCQ}) becomes
\begin{align}
\Sc{q_s(F,F)}  &\geq K^R_{n,m,M}(s)(\|F\|_D^2+\|F\|_2^2)
\end{align}
with coercive constant $K^R_{n,m,M}(s)$ given by  (\ref{COERCONST1}).

\medskip
{\em Step 3}.  Conclusions from continuity and coercivity conditions.
Hence we have proven that $\Sc{q_s^R}$ is coercive in $H^1(\Omega)\times H^1(\Omega)$. Fixing now any function $f\in L^2(\Omega)$, we can consider the corresponding functional
		\begin{equation} \label{Eq:test_funct_hyper1}
			\varphi_f(G) := \inner{f,G}_{L^2}, \quad G\in H^1(\Omega).
		\end{equation}
		Then by Lemma \ref{lem_Properties} iii), this linear functional is bounded on $H^1(\Omega)$, i.e.,
		\begin{equation*}
			\modulo{\varphi_f(G)} = \modulo{\inner{f,G}_{L^2}} \leq 2^{\frac{n}{2}}\norm{f}_2\norm{G}_2 \leq 2^{\frac{n}{2}}\norm{f}_2\norm{G}_{H^1}, \quad G\in H^1(\Omega).
		\end{equation*}
		Hence, the assumptions of Lemma \ref{lem_Lax_Milgram} are satisfied and there exists a unique weak solution $F_f\in H^1(\Omega)$ such that
		\begin{equation*}
			q_s^R(F_f,G) = \varphi_f(G) = \inner{f,G}_2, \quad \text{for all } G\in H^1(\Omega).
		\end{equation*}

		For the $L^2$-estimate in (\ref{Eq:norm_est_hyper6}), we test (\ref{Eq:test_funct_hyper1}) with $G = F_f$. Using also Lemma \ref{lem_Properties} iv), this gives
		\begin{equation} \label{Eq:norm_sc_bound_hyper1}
			\Sc{q_s^R(F_f,F_f)} = \Sc{\inner{f,F_f}_{L^2}} \leq \norm{f}_2\norm{F_f}_2.
		\end{equation}
		Combining this inequality with the coercivity estimate (\ref{Eq:Lower_bound_Hyper}),
so the coercive estimate (\ref{FINALSCQ}) becomes
\begin{align}
\Sc{q_s^R(F_f,F_f)}  &\geq K^R_{n,m,M}(s)(\|F_f\|_D^2+\|F_f\|_2^2)
\end{align}
with coercive constant $K_{n,m,M}(s)$ given by  (\ref{COERCONST}), furthermore gives
		\begin{equation}\label{GENERALE}
			K_{n,m,M}^R(s)(\|F_f\|_D^2+\|F_f\|_2^2)\leq \Sc{q_s^R(F_f,F_f)} \leq \norm{f}_2\norm{F_f}_2.
		\end{equation}
Recall that $Q_s(\mathcal{D}_H)F = f$ so $F=Q_s(\mathcal{D}_H)^{-1}f$ we get the estimates for the pseudo $S$-resolvent
\begin{equation}\label{GENERALE1}
			K_{n,m,M}^R(s)(\|Q_s(\mathcal{D}_H)^{-1}f\|_D^2+\|Q_s(\mathcal{D}_H)^{-1}f\|_2^2)\leq \norm{f}_2\norm{Q_s(\mathcal{D}_H)^{-1}f}_2,
		\end{equation}
from which we get
\begin{equation*}
			K_{n,m,M}^R(s)\|Q_s(\mathcal{D}_H)^{-1}f\|_2^2\leq \norm{f}_2\norm{Q_s(\mathcal{D}_H)^{-1}f}_2.
		\end{equation*}
and we obtain the estimate
\begin{equation}\label{ELLEDUE}
			\|Q_s(\mathcal{D}_H)^{-1}f\|_2\leq \frac{1}{K_{n,m,M}^R(s)}\, \norm{f}_2.
		\end{equation}
From (\ref{GENERALE}) we also have
\begin{equation}\label{ELLEDUEPERD}
			K_{n,m,M}^R(s)\|Q_s(\mathcal{D}_H)^{-1}f\|_D^2\leq \norm{f}_2 \norm{Q_s(\mathcal{D}_H)^{-1}f}_2
		\end{equation}
so replacing estimate (\ref{ELLEDUE}) in (\ref{ELLEDUEPERD})
 we get
\begin{equation*}
			K_{n,m,M}^R(s)\|Q_s(\mathcal{D}_H)^{-1}f\|_D^2\leq \norm{f}_2 \norm{Q_s(\mathcal{D}_H)^{-1}f}_2
\leq  \frac{1}{K_{n,m,M}^R(s))}\norm{f}_2^2
		\end{equation*}
so finally we get
\begin{equation*}
			\|Q_s(\mathcal{D}_H)^{-1}f\|_D\leq
\frac{1}{K_{n,m,M}^R(s)}\, \norm{f}_2.
		\end{equation*}

\end{proof}

\begin{remark}
Condition \eqref{SECODESTROBPRIMEH} defines a nonempty set when the datum $b\in L^\infty(\partial\Omega)$ is suitably chosen compared with the constants $m$, $M$ which depend on $\Omega$.
\end{remark}

\section{Dirac operator on the spherical space with Dirichlet  boundary conditions}\label{sphericalDirichlet}

	Let us now consider the Dirac operator $\mathcal{D}_S$ on the spherical space $S$. The operator $\mathcal{D}_S$ is computed on page 275 in \cite{DiracHarm}. let us consider a realisation of the spherical space $S$ by defining the metric
	\begin{equation*}
		g(x) = \p{\frac{1}{1+\modulo{x}^2}}^2\sum_{j=1}^n \dif x_j^2
	\end{equation*}
	on $\BR^n$. Then the Dirac operator on the spherical space $S$ is given by

	\begin{equation}\label{SFERICDIRACBOOK}
		\mathcal{D}_S=(1+|x|^2)\sum_{i=1}^ne_i\Big(\partial_{x_i}+\frac{1}{1+|x|^2} d\tau(e_ix)\Big),
	\end{equation}
	for $x\in\mathbb{R}^n$.

Choosing  the representation given by the left multiplication, as we did for the hyperbolic Dirac operator, we have $d\tau(e_ix)=e_ix$, where $x=\sum_{i=1}^ne_ix_i$, thanks to Lemma \ref{dtaurep}.
So the Dirac on the spherical space becomes
$$
\mathcal{D}_S=(1+|x|^2)\sum_{i=1}^ne_i\partial_{x_i}-nx
$$
and if we denote by
\begin{equation}\label{DIRAC}
\mathcal{D}_e:=\sum_{i=1}^ne_i\partial_{x_i}
\end{equation}
the Euclidean Dirac operator, the Dirac operator $\mathcal{D}_S$  rewrites as
\begin{equation}\label{DIREUC}
\mathcal{D}_S=(1+|x|^2)\mathcal{D}_e-nx.
\end{equation}
We will need the following lemma in order to compute the operator $\mathcal{D}^2_S$ explicitly.
\begin{lemma}
Let $\mathcal{D}_e$ be the Euclidean Dirac defined in (\ref{DIRAC}) and
we denote by $E$ the Euler operator
\begin{equation}\label{EULERO}
E:=\sum_{i=1}^nx_i\partial_{x_i}.
\end{equation}
Then, the following identity
\begin{equation}\label{IDENTDIREULER}
\mathcal{D}_e(xF)=-2EF-nF-x\mathcal{D}_eF
\end{equation}
holds, for every smooth functions $F:\Omega\subseteq \mathbb{R}^n\to \mathbb{R}_n$.
\end{lemma}
\begin{proof} It is a well known fact and it follows from direct computations.
\end{proof}

\begin{theorem}
Let $\mathcal{D}_S$ be the Dirac operator on the spherical space given by (\ref{DIREUC}),
$\mathcal{D}_e$ the Euclidean Dirac defined in (\ref{DIRAC}) and $E$ the Euler operator
as in (\ref{EULERO}).
Then, the square of the Dirac operator $\mathcal{D}_S$ is given by:
\begin{align}\label{DDUEQUAD}
\mathcal{D}_S^2F&
=(1+|x|^2)^2\mathcal{D}_e^2F
+2x(1+|x|^2)\mathcal{D}_eF\nonumber
\\
&
+2n(1+|x|^2)EF+n^2F
\end{align}
for every smooth functions $F:\Omega\subseteq \mathbb{R}^n\to \mathbb{R}_n$.
\end{theorem}
\begin{proof}
Consider smooth functions $F:\Omega\subseteq \mathbb{R}^n\to \mathbb{R}_n$ where $\Omega$ is an open set. We observe that
\begin{align}\label{DESSEIN}
\mathcal{D}_S^2F&=\Big((1+|x|^2)\mathcal{D}_e-nx)\Big)\Big((1+|x|^2)\mathcal{D}_eF-nxF)\Big)\nonumber
\\
&
=(1+|x|^2)\mathcal{D}_e\Big((1+|x|^2)\mathcal{D}_eF\Big)
-n(1+|x|^2)\mathcal{D}_e\Big(xF\Big)
-nx(1+|x|^2)\mathcal{D}_eF
-n^2|x|^2F,
\end{align}
we observe that, by a direct computation, we obtain
\begin{align*}
(1+|x|^2)\mathcal{D}_e\Big((1+|x|^2)\mathcal{D}_eF\Big)=2x(1+|x|^2)\mathcal{D}_eF+(1+|x|^2)^2\mathcal{D}_e^2F
\end{align*}
and, using the identity (\ref{IDENTDIREULER}), we have
\begin{align*}
-n(1+|x|^2)\mathcal{D}_e\Big(xF\Big)
=n(1+|x|^2)\Big(2EF+nF+x\mathcal{D}_eF\Big).
\end{align*}

Now we replace the above two relation in (\ref{DESSEIN}) and we obtain
\begin{align*}
\mathcal{D}_S^2F&
=2x(1+|x|^2)\mathcal{D}_eF+(1+|x|^2)^2\mathcal{D}_e^2F
\\
&
+n(1+|x|^2)\Big(2EF+nF+x\mathcal{D}_eF \Big)
-nx(1+|x|^2)\mathcal{D}_eF
-n^2|x|^2F
\end{align*}
and also
\begin{align*}
\mathcal{D}_S^2F&
=(1+|x|^2)^2\mathcal{D}_e^2F
+2x(1+|x|^2)\mathcal{D}_eF
\\
&
+2n(1+|x|^2)EF+n^2(1+|x|^2)F+n(1+|x|^2)x\mathcal{D}_eF
-nx(1+|x|^2)\mathcal{D}_eF
-n^2|x|^2F
\end{align*}
so
\begin{align*}
\mathcal{D}_S^2F
=(1+|x|^2)^2\mathcal{D}_e^2F
+2x(1+|x|^2)\mathcal{D}_eF
+2n(1+|x|^2)EF+n^2F,
\end{align*}
and we finally get (\ref{DDUEQUAD}).
\end{proof}

\subsection{The weak formulation}
We consider the formal operator
$$
Q_s(\mathcal{D}_S) := \mathcal{D}_S^2 - 2s_0 \mathcal{D}_S + |s|^2.
$$
We integrate by parts the second order term
in (\ref{DDUEQUAD}), i.e. $(1+|x|^2)^2\mathcal{D}_e^2F$. Since for the Euclidean Dirac operator we have $\mathcal{D}_e^2=-\Delta_n$, where $\Delta_n$ is the Laplace operator in $n$ dimensions, we obtain
$$
\int_\Omega\overline{(1+|x|^2)^2\mathcal{D}_e^2F}Gdx=
-\int_\Omega (1+|x|^2)^2\Delta_n \overline{F} Gdx,
$$
where $\Omega$ is an open (bounded) set in $\mathbb{R}^n$ with smooth boundary $\partial\Omega$.
\begin{lemma}
Let  $\mathcal{D}_e$ be the Euclidean Dirac defined in (\ref{DIRAC}) and $E$ the Euler operator
as in (\ref{EULERO}).
Then, the integration by parts formula
\begin{align}\label{INTBYPATE}
\int_\Omega\overline{(1+|x|^2)^2\mathcal{D}_e^2F}Gdx=
\int_\Omega(1+|x|^2)^2\sum_{i=1}^n\partial_{x_i}\overline{F}\partial_{x_i}Gdx
+4\int_\Omega(1+|x|^2)E\overline{F}Gdx,
\end{align}
holds for any smooth functions $F$ and $G:\Omega\subseteq \mathbb{R}^n\to \mathbb{R}_n$ with  $G=0$ on $\partial\Omega$.
\end{lemma}
\begin{proof}
Since we assume that $G=0$ on $\partial\Omega$, by Remark \ref{INTBYPARTFORM}, the integration part formula
\begin{align}\label{INT_BY_PARTS_PHERE}
&-\sum_{i=1}^n\Big\langle \partial_{x_i}
\Big(\partial_{x_i}F\Big),(1+|x|^2)^2G\Big\rangle_{L^2(\Omega, \mathbb{R}_n)}
=\sum_{i=1}^n\Big\langle
\partial_{x_i}F,\partial_{x_i}\Big((1+|x|^2)^2G\Big)\Big\rangle_{L^2(\Omega, \mathbb{R}_n)}\nonumber
\\
&
-\sum_{i=1}^n\Big\langle\nu_i\partial_{x_i}F,\Big((1+|x|^2)^2G\Big)\Big\rangle_{L^2(\partial\Omega, \mathbb{R}_n)},
\end{align}
where $\vec{\nu}=(\nu_1,\dots,\nu_n)$ outer unit normal vector to the boundary smooth boundary  $\partial\Omega$ of $\Omega$,
 gives
\begin{align}\label{BYPARTSDQUAD}
\int_\Omega\overline{(1+|x|^2)^2D^2F}Gdx&=
\int_\Omega\sum_{i=1}^n\partial_{x_i}\overline{F}\partial_{x_i}\Big((1+|x|^2)^2G\Big)dx\nonumber
\\
&
=
\int_\Omega\sum_{i=1}^n\Big[
\partial_{x_i}\overline{F}
\Big(4x_i(1+|x|^2)\Big)G+(1+|x|^2)^2\partial_{x_i}\overline{F}\partial_{x_i}G\Big]dx\nonumber
\\
&
=
\int_\Omega\sum_{i=1}^n\Big[4x_i(1+|x|^2)
\partial_{x_i}\overline{F}
G+(1+|x|^2)^2\partial_{x_i}\overline{F}\partial_{x_i}G\Big]dx\nonumber
\\
&
=
4\int_\Omega(1+|x|^2)E\overline{F}Gdx
+\int_\Omega(1+|x|^2)^2\sum_{i=1}^n\partial_{x_i}\overline{F}\partial_{x_i}Gdx,
\end{align}
so we get (\ref{INTBYPATE}).
\end{proof}
So now we can give the weak formulation to the equation
$$
\int_\Omega\overline{\mathcal{D}_S^2F}Gdx-2s_0 \int_\Omega\overline{\mathcal{D}_SF} G dx+|s|^2\int_\Omega\overline{F} G dx=\int_\Omega\overline{f} G dx
$$
with homogeneous Dirichlet boundary conditions,
 replacing $\mathcal{D}_S^2F$ given by (\ref{DDUEQUAD}) we get
we have
\begin{align*}
&\int_\Omega\Big(\overline{(1+|x|^2)^2\mathcal{D}_e^2F+2x(1+|x|^2)\mathcal{D}_eF+2n(1+|x|^2)EF+n^2F}\Big)Gdx
\\
&
-2s_0 \int_\Omega\overline{\mathcal{D}_SF} G dx+|s|^2\int_\Omega\overline{F} G dx=\int_\Omega\overline{f} G dx
\end{align*}
and also
\begin{align*}
&\int_\Omega
(1+|x|^2)^2\overline{\mathcal{D}_e^2F}  Gdx
\\
&
+2\int_\Omega (1+|x|^2)\overline{x\mathcal{D}_eF}Gdx
+2n\int_\Omega(1+|x|^2)E\overline{F}Gdx +
n^2\int_\Omega\overline{F}Gdx
\\
&
-2s_0 \int_\Omega\overline{\mathcal{D}_SF} G dx+|s|^2\int_\Omega\overline{F} G dx=\int_\Omega\overline{f} G dx
\end{align*}
Taking into account the integration by parts in (\ref{INTBYPATE})  we get
\begin{align*}
&\int_\Omega(1+|x|^2)^2\sum_{i=1}^n\partial_{x_i}\overline{F}\partial_{x_i}Gdx
+4\int_\Omega(1+|x|^2)E\overline{F}Gdx
\\
&
+2\int_\Omega (1+|x|^2)\overline{x\mathcal{D}_eF}Gdx
+2n\int_\Omega(1+|x|^2)E\overline{F}Gdx +
n^2\int_\Omega\overline{F}Gdx
\\
&
-2s_0 \int_\Omega\overline{\mathcal{D}_SF} G dx+|s|^2\int_\Omega\overline{F} G dx=\int_\Omega\overline{f} G dx
\end{align*}
and with simplifications we obtain
\begin{align*}
&\int_\Omega(1+|x|^2)^2\sum_{i=1}^n\partial_{x_i}\overline{F}\partial_{x_i}Gdx
+2(2+n)\int_\Omega(1+|x|^2)E\overline{F}Gdx
\\
&
+2\int_\Omega (1+|x|^2)\overline{x\mathcal{D}_eF}Gdx
-2s_0 \int_\Omega\overline{\mathcal{D}_SF} G dx+(|s|^2+n^2)\int_\Omega\overline{F} G dx
\\
&
=\int_\Omega\overline{f} G dx.
\end{align*}
We now consider the  sesquilinear form
\begin{align}\label{FORMSFERIC}
q_s(F,G):=
&\int_\Omega(1+|x|^2)^2\sum_{i=1}^n\partial_{x_i}\overline{F}\partial_{x_i}Gdx
+2(2+n)\int_\Omega(1+|x|^2)E\overline{F}Gdx\nonumber
\\
&
+2\int_\Omega (1+|x|^2)\overline{x\mathcal{D}_eF}Gdx
-2s_0 \int_\Omega\overline{\mathcal{D}_SF} G dx+(|s|^2+n^2)\int_\Omega\overline{F} G dx .
\end{align}

\begin{problem}
Let $\mathcal{D}_S$ be the Dirac operator on the spherical space given by (\ref{DIREUC}).
Recall that $\mathcal{D}_e:=\sum_{i=1}^ne_i\partial_{x_i}$ is the Euclidean Dirac operator as in (\ref{DIRAC})  and
$E=\sum_{i=1}^nx_i\partial_{x_i}$ is the Euler operator as in (\ref{EULERO}).
Using the scalar product $\inner{F,G}_2$ defined in (\ref{SCLAL2}) we write the  sesquilinear form
$q_s(F,G)$ defined in (\ref{FORMSFERIC})
 as
\begin{align}\label{FORMWITHSCALPROD2}
q_s(F,G) & = \sum_{i=1}^n\inner{ (1+|x|^2)^2\partial_{x_i}F,\partial_{x_i}G}_2
+2(2+n)\inner{(1+|x|^2)EF,G}_2\nonumber
\\
&
+2\inner{(1+|x|^2)x\mathcal{D}_eF,G}_2
-2s_0 \inner{\mathcal{D}_SF, G}_2
+(|s|^2+n^2)\inner{F,G}_2
\end{align}
with $\textup{dom } (q_s):= H^1_0(\Omega, \mathbb{R}_n) \times H^1_0(\Omega, \mathbb{R}_n)$.
Show that for some values of the spectral parameter
$s\in\mathbb{R}^{n+1},$ for every $f\in L^2(\Omega, \mathbb{R}_n)$ there exists a unique solution $F_f\in H^1_0(\Omega, \mathbb{R}_n)$ such that
	\begin{equation*}
		q_s(F_f,G) = \inner{f,G}_2, \quad \text{for all} \ G\in H^1_0(\Omega, \mathbb{R}_n).
	\end{equation*}
	Furthermore, determine $L^2$- and $D$-estimates of $F_f$, depending on the parameter $s\in \mathbb{R}^{n+1}$.
\end{problem}

Let us assume that $\Omega$ is such that $0\leq m<M<+\infty$ where $n,M$ are the constants defined by
\begin{equation}\label{CONTnM2}
\inf_{x\in\Omega}|x|=m,\ \ \ \ \sup_{x\in\Omega}|x|=M.
\end{equation}

\begin{lemma}\label{STIMEExDeD}
Let $\mathcal{D}_e$ be the Euclidean Dirac operator defined in (\ref{DIRAC}), $E$ the Euler operators as in (\ref{EULERO}) and $\mathcal{D}_S$ the Dirac operator on spherical space in (\ref{DIREUC}). Then,
for all $F\in H^1(\Omega, \mathbb{R}_n)$,  we have the following estimates

\begin{align}
\label{STIMADIRAC}
\|\mathcal{D}_eF\|_2&\leq \sqrt{n}\|F\|_D,
\\
\label{STIMAXDIRAC}
\|x\mathcal{D}_eF\|_2&\leq \sqrt{n}M\|F\|_D,
\\
\label{STIMAEULERO}
\|EF\|_2&\leq \sqrt{n}M\|F\|_D,
\\
\label{STIMADIRACEUCL}
\|\mathcal{D}_SF\|_2&\leq  \textcolor{black}{(1+M^2)\sqrt{n}\|F\|_D+nM\|F\|_2},
\end{align}
where  $\sup_{x\in\Omega}|x|=M$.
\end{lemma}
\begin{proof}
Estimate (\ref{STIMADIRAC}) follows from
$$
\|\mathcal{D}_eF\|_2=
\|\sum_{i=1}^ne_i\partial_{x_i}F\|_2\leq \sum_{i=1}^n \|\partial_{x_i}F\|_2\leq(\sum_{i=1}^n1^2)^{1/2}
\Big(\sum_{i=1}^n \|\partial_{x_i}F\|_2^2\Big)^{1/2}=\sqrt{n}\|F\|_D
$$
since $x$ is a 1-vector, estimate (\ref{STIMAXDIRAC}) follows from Lemma \ref{lem_Properties} ii) and estimate (\ref{STIMADIRAC})
$$
\|x\mathcal{D}_eF\|_2=|x| \|\mathcal{D}_eF\|_2\leq \Big(\sup_{x\in\Omega}|x|\Big) \sqrt{n}\|F\|_D=\sqrt{n}M\|F\|_D.
$$
Estimate (\ref{STIMAEULERO}) follows from
$$
\|EF\|_2=\|\sum_{i=1}^nx_i\partial_{x_i}F\|_2\leq
\sum_{i=1}^n|x_i|\|\partial_{x_i}F\|_2\leq \sqrt{n}M\|F\|_D
$$
and finally (\ref{STIMADIRACEUCL}) follows from
$$
\|\mathcal{D}_SF\|_2\leq \|(1+|x|^2)\mathcal{D}_eF\|_2+n|x|\|F\|_2\leq (1+M^2)\sqrt{n}\|F\|_D+nM\|F\|_2.
$$
\end{proof}

\begin{theorem} \label{THFORTHESPHER}
		Let $\Omega\subset \mathbb{R}^{n}$ be a bounded set with smooth boundary and let
$\mathcal{D}_S$ the Dirac operator on spherical space in (\ref{DIREUC}).
Let us consider also that there exist two constants $0\leq m<M<+\infty$ such that
\begin{equation*}
\inf_{x\in\Omega}|x|=m,\ \ \ \ \sup_{x\in\Omega}|x|=M.
\end{equation*}
Let $s\in\BR^{n+1}$ be such that
\begin{equation}\label{super}
|s|^2+n^2-2|s_0| nM +(1+m^2)^2>0
\end{equation}
and
\begin{equation}\label{SECONDAPPPP}
\begin{split}
&|\underline{s}|^2(1+m^2)^2-s_0^2\Big(n(1+M^2)^2-(1+m^2)^2\Big)
-2nM\Big((3+n)(1+M^2)^2+ (1+m^2)^2 \Big)|s_0|
\\
&
>n(3+n)^2M^2(1+M^2)^2-n^2(1+m^2)^2.
\end{split}
\end{equation}
Then, for every $f\in L^2(\Omega)$ there exists a unique $F_f\in H_0^1(\Omega)$ such that
		\begin{equation}
			q_s(F_f,G) = \inner{f,G}_2, \qquad \text{for all } G\in H_0^1(\Omega).
		\end{equation}
where $q_s(F,G) $ is the sesquilinear form defined in (\ref{FORMWITHSCALPROD2}),
Moreover, this solution satisfies the bounds
		\begin{equation} \label{Eq:norm_est_hyper5}
			\|Q_s(\mathcal{D}_S)^{-1}f\|_2\leq \frac{1}{H_{n,m,M}(s)}\, \norm{f}_2,
		\ \ \ {\rm and}\ \
			\|Q_s(\mathcal{D}_S)^{-1}f\|_D\leq \frac{1}{H_{n,m,M}(s)} \,\norm{f}_2,
		\end{equation}
where

\begin{align}\label{COSTHNMM}
H_{n,m,M}(s)&=
\frac 12\left(|s|^2+n^2\textcolor{black}{-2|s_0| nM }+(1+m^2)^2\right)
\\
& -\frac 12 \sqrt{\Big((|s|^2+n^2\textcolor{black}{-2|s_0| nM })- (1+m^2)^2\Big)^2+4\Big(\sqrt{n} \textcolor{black}{ (1+M^2) }\Big((3+n)M+|s_0|\Big)\Big)^2}.
\nonumber
\end{align}

	\end{theorem}
\begin{proof} We split the proof in steps.

\medskip
{\em Step 1}. The continuity follows from Lemma \ref{lem_Properties} and Lemma \ref{STIMEExDeD}, in fact we have
\begin{align}\label{continuity}
|q_s(F,G)| & \leq \Big|\sum_{i=1}^n\inner{ (1+|x|^2)^2\partial_{x_i}F,\partial_{x_i}G}_2\Big|
+2(2+n)|\inner{(1+|x|^2)EF,G}_2|\nonumber
\\
&
+2|\inner{(1+|x|^2)x\mathcal{D}_eF,G}_2|
+2|s_0|\ |\inner{\mathcal{D}_SF, G}_2|
+(|s|^2+n^2)|\inner{F,G}_2|
\end{align}
observe that for the first term on the right hand side of (\ref{continuity}) we have
\begin{align}\label{ESTIMSECORDER}
\Big|\sum_{i=1}^n\inner{ (1+|x|^2)^2\partial_{x_i}F,\partial_{x_i}G}_2\Big|
&\leq(1+M^2)^2\sum_{i=1}^n\Big|\inner{ \partial_{x_i}F,\partial_{x_i}G}_2\Big|\nonumber
\\
&
\leq(1+M^2)^2\sum_{i=1}^n 2^{n/2}\|\partial_{x_i}F\|_2 \|\partial_{x_i}G\|_2\nonumber
\\
&
\leq(1+M^2)^2 2^{n/2}\sum_{i=1}^n \Big(\|\partial_{x_i}F\|^2_2\Big)^{1/2} \nonumber \sum_{i=1}^n\Big(\|\partial_{x_i}G\|^2_2\Big)^{1/2}
\\
&
\leq(1+M^2)^2 2^{n/2}\|F\|_D  \|G\|_D,
\end{align}
so we get
\begin{align}\label{continuity1}
|q_s(F,G)| & \leq 2^{n/2}(1+M^2)^2\|F\|_D  \|G\|_D\nonumber
\\
&
+2(2+n)2^{n/2}\|(1+|x|^2)EF\|_2\ \|G\|_2\nonumber
\\
&
+2^{n/2+1}\|(1+|x|^2)x\mathcal{D}_eF\|_2  \|G\|_2\nonumber
\\
&
+\textcolor{black}{2^{n/2+1}|s_0| \|\mathcal{D}_SF\|_2 \|G\|_2}\nonumber
\\
&
+2^{n/2}(|s|^2+n^2)\|F\|_2 \|G\|_2.
\end{align}

Now using Lemma \ref{STIMEExDeD} we obtain

\begin{align}\label{continuity2}
|q_s(F,G)| & \leq 2^{n/2}(1+M^2)^2\|F\|_D  \|G\|_D\nonumber
\\
&
+2(2+n)2^{n/2}(1+M^2)\sqrt{n}M\|F\|_D\ \|G\|_2\nonumber
\\
&
+2^{n/2+1}(1+M^2)\sqrt{n}M\|F\|_D  \|G\|_2\nonumber
\\
&
+2^{n/2+1}|s_0|\textcolor{black}{\Big( (1+M^2)\sqrt{n}\|F\|_D+nM\|F\|_2 \Big)}\|G\|_2\nonumber
\\
&
+2^{n/2}(|s|^2+n^2)\|F\|_2 \|G\|_2
\end{align}
from which we deduce that there exists a positive constant
$C_1(s,n,M)$ such that
\begin{equation}\label{CONTINSPHER}
|q_s(F,G)| \leq C_1(s,n,M)\|F\|_{H_0^1}\|G\|_{H_0^1},
\end{equation}
since we have set $\|F\|^2_{H_0^1}:= \norm{F}_2^2+\norm{F}_D^2$,
 the  sesquilinear form is continuous from $H^1_0(\Omega, \mathbb{R}_n) \times H^1_0(\Omega, \mathbb{R}_n)$ to $\mathbb{R}_n$.

\medskip
{\em Step 2}. Coercivity. The coercivity regards the scalar part of $q_s(F,F)$, so we get

\begin{align}
{\rm Sc}(q_s(F,F)) & = {\rm Sc}\sum_{i=1}^n\inner{ (1+|x|^2)^2\partial_{x_i}F,\partial_{x_i}F}_2
+2(2+n){\rm Sc}\inner{(1+|x|^2)EF,F}_2\nonumber
\\
&
+2{\rm Sc}\inner{(1+|x|^2)x\mathcal{D}_eF,F}_2
-2s_0 {\rm Sc}\inner{\mathcal{D}_SF, F}_2
+(|s|^2+n^2){\rm Sc}\inner{F,F}_2
\end{align}
so that

\begin{align*}
{\rm Sc}(q_s(F,F))&\geq (1+m^2)^2\|F\|^2_D+(|s|^2+n^2)\|F\|^2_2
\\
&
-2(2+n)\Big|{\rm Sc}\inner{(1+|x|^2)EF,F}_2\nonumber\Big|
\\
&
-2\Big|{\rm Sc}\inner{(1+|x|^2)x\mathcal{D}_eF,F}_2\Big|
-2|s_0| \Big|{\rm Sc}\inner{\mathcal{D}_SF, F}_2\Big|.
\end{align*}
We also have

\begin{align*}
{\rm Sc}(q_s(F,F))&\geq (1+m^2)^2\|F\|^2_D+(|s|^2+n^2)\|F\|^2_2
\\
&
-2(2+n)(1+M^2)\|EF\|_2 \|F\|_2
\\
&
-2(1+M^2)\|x\mathcal{D}_eF\|_2 \|F\|_2
-2|s_0| \|\mathcal{D}_SF\|_2 \|F\|_2.
\end{align*}
Lemma \ref{STIMEExDeD} allows to deduce

\begin{align*}
{\rm Sc}(q_s(F,F))&\geq (1+m^2)^2\|F\|^2_D+(|s|^2+n^2)\|F\|^2_2
\\
&
-2(2+n)(1+M^2)\sqrt{n}M\|F\|_D \|F\|_2
\\
&
-2(1+M^2)\sqrt{n}M\|F\|_D \|F\|_2
-2|s_0| \textcolor{black}{\Big( (1+M^2)\sqrt{n}\|F\|_D+nM\|F\|_2 \Big)} \|F\|_2
\end{align*}
and

\begin{align*}
{\rm Sc}(q_s(F,F))&\geq (1+m^2)^2\|F\|^2_D+(|s|^2+n^2)\|F\|^2_2
\\
&
-2(2+n)(1+M^2)\sqrt{n}M\|F\|_D \|F\|_2
\\
&
-2(1+M^2)\sqrt{n}M\|F\|_D \|F\|_2
\\
&
\textcolor{black}{ -2|s_0| (1+M^2)\sqrt{n}\|F\|_D\|F\|_2}
\textcolor{black}{-2|s_0| nM\|F\|_2^2 }.
\end{align*}
Collecting the terms we obtain

\begin{align*}
{\rm Sc} (q_s(F,F))&\geq (1+m^2)^2\|F\|^2_D+(|s|^2+n^2\textcolor{black}{-2|s_0| nM } )\|F\|^2_2
\\
&
-2\sqrt{n}\Big((2+n)(1+M^2)M+(1+M^2)M+|s_0|  \textcolor{black}{ (1+M^2) }
\Big)\|F\|_D \|F\|_2
\end{align*}
and also
\begin{align*}
{\rm Sc} (q_s(F,F))&\geq (1+m^2)^2\|F\|^2_D+(|s|^2+n^2\textcolor{black}{-2|s_0| nM } )\|F\|^2_2
\\
&
-2\sqrt{n} \textcolor{black}{ (1+M^2) }\Big((3+n)M+|s_0|
\Big)\|F\|_D \|F\|_2.
\end{align*}
For  $\delta>0$, using the inequality
		\begin{equation*}
			\norm{F}_D\norm{F}_2 \leq \frac{1}{2\delta}\norm{F}_D^2 + \frac{\delta}{2}\norm{F}_2^2,
		\end{equation*}
we get
\begin{align*}
{\rm Sc}(q_s(F,F))&\geq (1+m^2)^2\|F\|^2_D+(|s|^2+n^2\textcolor{black}{-2|s_0| nM } )\|F\|^2_2
\\
&
-\sqrt{n} \textcolor{black}{ (1+M^2) }\Big((3+n)M+|s_0|
\Big)\Big(\frac{1}{\delta}\norm{F}_D^2 + \delta\norm{F}_2^2\Big)
\end{align*}
which gives
\begin{align}\label{COERCSTEPONESPHERE}
{\rm Sc}(q_s(F,F))&\geq
\Big[ (1+m^2)^2-\frac{1}{\delta}\sqrt{n} \textcolor{black}{ (1+M^2) }\Big((3+n)M+|s_0|
\Big)\Big]\|F\|^2_D\nonumber
\\
&
+\Big[(|s|^2+n^2\textcolor{black}{-2|s_0| nM })-\delta\sqrt{n} \textcolor{black}{ (1+M^2) }\Big((3+n)M+|s_0|
\Big) \Big]\|F\|^2_2 .
\end{align}
Now we reason as in the case of the hyperbolic space and we choose $\delta$ such that
the coefficients of $\|F\|^2_D$ and $\|F\|^2_2$ turn out to be equal.
 For this purpose we have to determine $\delta>0$ such that
\begin{equation}
\label{EQUA_FOR_DELTA_SPHERE_DIRIC}
(1+m^2)^2-\frac{1}{\delta}\sqrt{n} \textcolor{black}{ (1+M^2) }\Big((3+n)M+|s_0|
\Big)=
(|s|^2+n^2\textcolor{black}{-2|s_0| nM })-\delta\sqrt{n} \textcolor{black}{ (1+M^2) }\Big((3+n)M+|s_0|
\Big).
\end{equation}
Set
\begin{equation}\label{AzeroSphere}
A(s_0):=\sqrt{n} \textcolor{black}{ (1+M^2) }\Big((3+n)M+|s_0|\Big)
\end{equation}
the equation in $\delta$ becomes
$$
(1+m^2)^2-\frac{1}{\delta}A(s_0)=
(|s|^2+n^2\textcolor{black}{-2|s_0| nM })-\delta A(s_0)
$$
that is
$$
\delta^2 A(s_0)-\delta \Big((|s|^2+n^2\textcolor{black}{-2|s_0| nM })- (1+m^2)^2\Big)-A(s_0)=0.
$$
By setting
$$
a:=A(s_0), \ \ b:=-\Big((|s|^2+n^2\textcolor{black}{-2|s_0| nM })- (1+m^2)^2\Big), \ \ \ c:=-A(s_0)
$$
we obtain that
\begin{equation}\label{HnmM}
\Delta:=b^2-4ac=
\Big((|s|^2+n^2\textcolor{black}{-2|s_0| nM })- (1+m^2)^2\Big)^2+4\Big(\sqrt{n} \textcolor{black}{ (1+M^2) }\Big((3+n)M+|s_0|\Big)\Big)^2
\end{equation}
and the two solutions, that are real since $\Delta >0$ \textcolor{black}{(indeed $\Delta =0$ would imply that the two quantities that are squared in \eqref{HnmM} both vanish and this is not possible) are:}
\begin{equation}\label{deltaplussphere}
\delta^\pm=\frac{1}{2A(s_0)}\Big(|s|^2+n^2\textcolor{black}{-2|s_0| nM }- (1+m^2)^2\pm\sqrt{\Delta}\Big).
\end{equation}
It is immediate to see that
only $\delta^+$ is positive, so we get
\begin{equation}\label{AESSEZERODELTA}
A(s_0)\delta^+=\frac{1}{2}\Big(|s|^2+n^2\textcolor{black}{-2|s_0| nM }- (1+m^2)^2+\sqrt{\Delta}\Big).
\end{equation}
\textcolor{black}{The coefficient of $\|F\|^2_2$ (and $\|F\|^2_D$), denoted by $H_{n,m,M}(s)$ is then
$$
H_{n,m,M}(s):=(|s|^2+n^2\textcolor{black}{-2|s_0| nM })-\delta^+\sqrt{n} \textcolor{black}{ (1+M^2) }\Big((3+n)M+|s_0|
\Big)
$$
and replacing $A(s_0)\delta^+$ in (\ref{AESSEZERODELTA}) it becomes
\[
\begin{split}
H_{n,m,M}(s)&=|s|^2+n^2\textcolor{black}{-2|s_0| nM }-\frac 12 \left(|s|^2+n^2\textcolor{black}{-2|s_0| nM }-(1+m^2)^2+\sqrt{\Delta}\right)\\
&=\frac 12\left(|s|^2+n^2\textcolor{black}{-2|s_0| nM }+(1+m^2)^2\right)
\\
& -\frac 12 \sqrt{\Big((|s|^2+n^2\textcolor{black}{-2|s_0| nM })- (1+m^2)^2\Big)^2+4\Big(\sqrt{n} \textcolor{black}{ (1+M^2) }\Big((3+n)M+|s_0|\Big)\Big)^2}.
\end{split}
\]
We now impose that $H_{n,m,M}$ is positive, namely
\[
\begin{split}
&|s|^2+n^2\textcolor{black}{-2|s_0| nM }+(1+m^2)^2>\\
&  \sqrt{\Big(|s|^2+n^2\textcolor{black}{-2|s_0| nM }- (1+m^2)^2\Big)^2+4\Big(\sqrt{n} \textcolor{black}{ (1+M^2) }\Big((3+n)M+|s_0|\Big)\Big)^2}
\end{split}
\]
from which we obtain the system
\begin{equation}
|s|^2+n^2\textcolor{black}{-2|s_0| nM }+(1+m^2)^2>0
\end{equation}
and
\[
\begin{split}
&\Big(|s|^2+n^2\textcolor{black}{-2|s_0| nM }+(1+m^2)^2\Big)^2
\\
&
>\Big((|s|^2+n^2\textcolor{black}{-2|s_0| nM })- (1+m^2)^2\Big)^2+4\Big(\sqrt{n} \textcolor{black}{ (1+M^2) }\Big((3+n)M+|s_0|\Big)\Big)^2.
\end{split}
\]
}
We then have
\[
\begin{split}
&\Big(|s|^2+n^2\textcolor{black}{-2|s_0| nM }\Big)(1+m^2)^2
>\Big(\sqrt{n} \textcolor{black}{ (1+M^2) }\Big((3+n)M+|s_0|\Big)\Big)^2,
\end{split}
\]
and
\[
\begin{split}
&\Big(|s|^2+n^2-2|s_0| nM \Big)(1+m^2)^2
\\
&
>n(1+M^2)^2((3+n)M)^2
+n(1+M^2)^2s_0^2
+2n(1+M^2)^2(3+n)M|s_0|,
\end{split}
\]
which rewrites as
\[
\begin{split}
&|s|^2(1+m^2)^2+n^2(1+m^2)^2-2|s_0| nM(1+m^2)^2
\\
&
>n(1+M^2)^2((3+n)M)^2
+n(1+M^2)^2s_0^2
+2n(1+M^2)^2(3+n)M|s_0|
\end{split}
\]
and
\[
\begin{split}
&|s|^2(1+m^2)^2-n(1+M^2)^2s_0^2
-2nM\Big((3+n)(1+M^2)^2+ (1+m^2)^2 \Big)|s_0|
\\
&
>n(3+n)^2M^2(1+M^2)^2-n^2(1+m^2)^2.
\end{split}
\]
Finally we get
\[
\begin{split}
&|\underline{s}|^2(1+m^2)^2-s_0^2\Big(n(1+M^2)^2-(1+m^2)^2\Big)
-2nM\Big((3+n)(1+M^2)^2+ (1+m^2)^2 \Big)|s_0|
\\
&
>n(3+n)^2M^2(1+M^2)^2-n^2(1+m^2)^2.
\end{split}
\]

{\em Step 3}. It follows the proof given in the case of Theorem \ref{Th:Main_Hyper}.

\end{proof}

\begin{remark}\label{regionn} The region identified in the previous theorem is such that on each complex plane $\mathbb C_I$ to which $s_0+Is_1$ belongs it is the intersection between two regions $C_1$ and $C_2$. $C_1$ is defined by $|s|^2+n^2-2|s_0| nM +(1+m^2)^2>0$ which rewrites as
$$
(|s_0|- nM)^2+s_1^2>n^2(M^2-1)- (1+m^2)^2.
$$
 It is the exterior of two circumferences when $n^2(M^2-1)- (1+m^2)^2>0$ while it is the whole complex plane (except at most $(\pm nM,0)$) when $s_0+Is_1$ when $n^2(M^2-1)- (1+m^2)^2\leq 0$. One circumference has center in $(nM,0)$, radius squared $R^2=n^2(M^2-1)+(1+m^2)^2$ and belongs to the right complex plane $s_0\geq 0$; the second has center $(-nM,0)$, radius squared $R^2=n^2(M^2-1)+(1+m^2)^2$ and belongs to the left complex plane $s_0<0$.
$C_2$ is a region with boundary given by two hyperbolas: one in the half plane $s_0>0$ with center of symmetry $(-k,0)$ and one in the half plane $s_0<0$ with center of symmetry $(k,0)$ where
$$
k=\frac{nM((3+n)(1+M^2)^2+(1+m^2)^2)}{n(1+M^2)^2-(1+m^2)^2}>0.
$$
We note that we can rewrite the hyperbolas in the form
$$
(n(1+M^2)^2-(1+m^2)^2))(|s_0|+k)^2 -(1+m^2)^2s_1^2<
$$
$$(n(1+M^2)^2-(1+m^2)^2))k^2-n(3+n)^2M^2(1+M^2)^2+n^2(1+m^2)^2;
$$
easy but cumbersome computations show that
 the right hand side is positive and that the centers of symmetry satisfy the inequality. Thus the portion $P$ of the plane identified by \eqref{SECONDAPPPP} is the unbounded region intersecting the $s_1$-axis.\\
We now show that condition \eqref{super} is superfluous, in fact either it does not impose further conditions or, when $n^2(M^2-1)- (1+m^2)^2>0$ holds, is implied by \eqref{SECONDAPPPP}. To this end we show that
$P$ contains $Q:=\{(s_0,s_1)\ |\ s_1^2 \geq R^2\}$. By the above discussion we see that it is enough to show that the points in which $P$ intersects the $s_1$-axis are such that $s_1^2 \geq R^2$. This is easily seen in fact such points satisfy
$$
(1+m^2)^2s_1^2= n(3+n)^2M^2(1+M^2)^2-n^2(1+m^2)^2
$$
and we have to check if
$$
n(3+n)^2M^2\frac{(1+M^2)^2}{(1+m^2)^2}-n^2> n^2(M^2-1)-(1+m^2)^2
$$
which rewrites as
$$
n(3+n)^2M^2\frac{(1+M^2)^2}{(1+m^2)^2}+(1+m^2)^2> n^2M^2
$$
which is clearly satisfied. Thus $P$ contains $Q$ and, a fortiori, the region outside the circumferences. In this case the condition \eqref{super} is then superfluous.
\end{remark}

\section{Dirac operator on the spherical space with Robin-like boundary conditions}\label{sphericalRobin}

The natural Robin-like boundary conditions for the Dirac operator
$\mathcal{D}_S$  in  (\ref{DIREUC}) on spherical space
 are given by
\begin{equation}\label{BONDARYROBIN}
(1+|x|^2)^2\sum_{i=1}^n\nu_i\partial_{x_i}F(x)+b(x)F(x)=0 \  \ {\rm on} \ \ \partial\Omega,
\end{equation}
where we assume that the real valued function $b$ is such that $b\in L^\infty(\partial\Omega)$.
To formulate the weak solution of the equation
$$
(\mathcal{D}_S^2 - 2s_0 \mathcal{D}_S + |s|^2)F(x)=f(x) \ \ {\rm in}\ \  \Omega
$$
with boundary conditions (\ref{BONDARYROBIN})
 we integrate by parts the second order term
using the considerations in Remark \ref{INTBYPARTFORM} and the integration part formula in
(\ref{INT_BY_PARTS_PHERE}) we obtain
the  sesquilinear form
$$
q_s^R(F,G):=q_s(F,G)-\sum_{i=1}^n\Big\langle\nu_i\partial_{x_i}F,
\Big((1+|x|^2)^2G\Big)\Big\rangle_{L^2(\partial\Omega)}
$$
where $q_s(F,G)$ defined in (\ref{FORMWITHSCALPROD2}) and we assume the domain of $q_s^R$ to be
 $\textup{dom } (q_s^R):= H^1(\Omega, \mathbb{R}_n) \times H^1(\Omega, \mathbb{R}_n)$.
In this section we investigate the following problem.
\begin{problem}\label{RobinSphere}
Let $\mathcal{D}_S$ be the Dirac operator  (\ref{DIREUC}) on spherical space.
Recall that $\mathcal{D}_e:=\sum_{i=1}^ne_i\partial_{x_i}$ is the Euclidean Dirac operator as in (\ref{DIRAC})  and
$E=\sum_{i=1}^nx_i\partial_{x_i}$ is the Euler operator as in (\ref{EULERO}).
Using the scalar product $\inner{F,G}_2$ defined in (\ref{SCLAL2}) and the  sesquilinear form
$q_s(F,G)$ defined in (\ref{FORMWITHSCALPROD2}) we define the Robin  sesquilinear form
\begin{align}\label{ROBINqsfromsphere}
q_s^R(F,G) &:=q_s(F,G)+\langle b\tau_DF,\tau_DG\rangle_{L^2(\partial\Omega)}
\\
&
= \sum_{i=1}^n\inner{ (1+|x|^2)^2\partial_{x_i}F,\partial_{x_i}G}_2
+2(2+n)\inner{(1+|x|^2)EF,G}_2\nonumber
\\
&
+2\inner{(1+|x|^2)x\mathcal{D}_eF,G}_2
-2s_0 \inner{\mathcal{D}_SF, G}_2
+(|s|^2+n^2)\inner{F,G}_2\nonumber
\\
&
+\langle b\tau_DF, \tau_DG\rangle_{L^2(\partial\Omega)},\nonumber
\end{align}
with $\textup{dom } (q_s^R):= H^1(\Omega, \mathbb{R}_n) \times H^1(\Omega, \mathbb{R}_n)$ and where
the operator $\tau_{D}:H^1(\Omega)\rightarrow L^2(\partial\Omega)$
is the bounded Dirichlet trace operator, see \cite[Equation (4)]{M87}.
Show that for some values of the spectral parameter
$s\in\mathbb{R}^{n+1},$ for every $f\in L^2(\Omega, \mathbb{R}_n)$ there exists a unique solution $F_f\in H^1(\Omega, \mathbb{R}_n)$ such that
	\begin{equation*}
		q_s^R(F_f,G) = \inner{f,G}_2, \quad \text{for all} \ G\in H^1(\Omega, \mathbb{R}_n).
	\end{equation*}
	Furthermore, determine $L^2$- and $D$-estimates of $F_f$, depending on the parameter $s\in \mathbb{R}^{n+1}$.
\end{problem}

\begin{theorem} \label{ROBENTHFORTHESPHER}
		Let $\Omega\subset \mathbb{R}^{n}$ be a bounded set with smooth boundary.
Let $q_s^R(F,G)$ be the  sesquilinear form in (\ref{ROBINqsfromsphere}) associated with the Dirac operator $\mathcal{D}_S$  (\ref{DIREUC}) on spherical space with the associated Robin-type boundary conditions (\ref{BONDARYROBIN}).
Let us assume  that $\Omega$ is such that the two constants
\begin{equation*}
\inf_{x\in\Omega}|x|=m,\ \ \ \ \sup_{x\in\Omega}|x|=M
\end{equation*}
satisfy $0\leq m<M<+\infty$.
Let $s\in\BR^{n+1}$ be such that
\begin{align}\label{QULI1ROBSP}
|s|^2+n^2\textcolor{black}{-2|s_0| nM }+ (1+m^2)^2
-2^{\frac{n}{2}+1}\Vert b\Vert_{L^\infty(\partial\Omega)}\Vert\tau_D\Vert^2
>0
\end{align}
and
\begin{align}\label{QULI2RBSPHplus}
&
|s|^2\Big((1+m^2)^2-2^{\frac{n}{2}}\Vert b\Vert_{L^\infty(\partial\Omega)}\Vert\tau_D\Vert^2\Big)-n(1+M^2)^2s_0^2\nonumber
\\
&
\textcolor{black}{-2|s_0| nM }
\Big((1+m^2)^2+(3+n)(1+M^2)^2-2^{\frac{n}{2}}\Vert b\Vert_{L^\infty(\partial\Omega)}\Vert\tau_D\Vert^2
\Big)\nonumber
\\
&
>
n(1+M^2)^2(3+n)^2M^2
+(1+m^2)^22^{\frac{n}{2}}\Vert b\Vert_{L^\infty(\partial\Omega)}\Vert\tau_D\Vert^2\nonumber
\\
&
-n^2
\Big((1+m^2)^2-2^{\frac{n}{2}}\Vert b\Vert_{L^\infty(\partial\Omega)}\Vert\tau_D\Vert^2\Big)
-2^{n}\Vert b\Vert_{L^\infty(\partial\Omega)}^2\Vert\tau_D\Vert^4.
\end{align}
Then, for every $f\in L^2(\Omega)$ there exists a unique $F_f\in H^1(\Omega)$ such that
		\begin{equation}
			q^R_s(F_f,G) = \inner{f,G}_2, \qquad \text{for all } G\in H^1(\Omega).
		\end{equation}
		Moreover, this solution satisfies the estimates
		\begin{equation} \label{Eq:norm_est_hyper4}
			\|Q_s(\mathcal{D}_S)^{-1}f\|_2\leq \frac{1}{H^R_{n,m,M}(s)}\, \norm{f}_2,
		\ \ \ {\rm and}\ \
			\|Q_s(\mathcal{D}_S)^{-1}f\|_D\leq \frac{1}{H^R_{n,m,M}(s)} \,\norm{f}_2,
		\end{equation}
where
\begin{align}\label{COERCONSTROBSP}
H^R_{n,m,M}(s):&
=\frac{1}{2}(|s|^2+n^2\textcolor{black}{-2|s_0| nM })
+\frac{1}{2} (1+m^2)^2
-2^{\frac{n}{2}}\Vert b\Vert_{L^\infty(\partial\Omega)}\Vert\tau_D\Vert^2\nonumber
\\
&
-\frac{1}{2}
\sqrt{\Big((|s|^2+n^2\textcolor{black}{-2|s_0| nM })- (1+m^2)^2\Big)^2+4\Big(\sqrt{n} \textcolor{black}{ (1+M^2) }\Big((3+n)M+|s_0|\Big)\Big)^2}.
\end{align}

	\end{theorem}

\begin{proof}
 We will verify that the  sesquilinear form $q_s^R$, defined in \eqref{ROBINqsfromsphere},  satisfies the assumptions of the Lemma~\ref{lem_Lax_Milgram}.
We recall that the boundary term is estimated as in \eqref{ESTIM_bound_TERMTAU} as
\begin{align}\label{BOUNDRUTERM}
|\langle b\,\tau_DF,\tau_DG\rangle_{L^2(\partial\Omega)}|&\leq 2^{\frac{n}{2}}\Vert b\Vert_{L^\infty(\partial\Omega)}\Vert\tau_D\Vert^2\Vert F\Vert_{H^1}\Vert G\Vert_{H^1}.
\end{align}
{\em Step 1.}
The continuity is established by the estimate
(\ref{CONTINHUNO})
\begin{equation}\label{CONTISPHERROBIN}
|q_s(F,G)| \leq C(s,n,M)\|F\|_{H^1}\|G\|_{H^1},
\end{equation}
 that holds for the space $H^1(\Omega)$ and
so it follows that from \eqref{ROBINqsfromsphere} and (\ref{FORMWITHSCALPROD2}) we obtain
\begin{align*}	
|q^R_s(F,G)|\leq |q_s(F,G)|&+|\langle b\,\tau_DF,\tau_DG\rangle_{L^2(\partial\Omega)}|
\\
&
\leq\Big(C(s,n,M) +2^{\frac{n}{2}}\Vert b\Vert_{L^\infty(\partial\Omega)}\Vert\tau_D\Vert^2\Big)\|F\|_{H^1}\|G\|_{H^1}.
\end{align*}

{\em Step 2.}
For the coercivity we observe that from the inequalities \eqref{Eq:Lower_bound_HyperROB} and \eqref{Eq:Lower_bound_HyperROBDUE} for $\Sc{q_s(F,F)}$
and the estimate of the boundary term in \eqref{BOUNDRUTERM} we have
\begin{align}\label{BUTOSLO}
|\langle b\,\tau_DF,\tau_DF\rangle_{L^2(\partial\Omega)}|&\leq 2^{\frac{n}{2}}\Vert b\Vert_{L^\infty(\partial\Omega)}\Vert\tau_D\Vert^2\Vert F\Vert_{H^1}^2\nonumber
\\
&
=
2^{\frac{n}{2}}\Vert b\Vert_{L^\infty(\partial\Omega)}\Vert\tau_D\Vert^2\Big[\Vert F\Vert_{D}^2+\Vert F\Vert_{L^2}^2\Big].
\end{align}
So the coercivity condition  for $q^R_s(F,F)$ becomes
follows from the coercivity of $q_s(F,F)$ given in (\ref{COERCSTEPONESPHERE}), that is

\begin{align*}
{\rm Sc}(q_s(F,F))&\geq
\Big[ (1+m^2)^2-\frac{1}{\delta}\sqrt{n} \textcolor{black}{ (1+M^2) }\Big((3+n)M+|s_0|
\Big)\Big]\|F\|^2_D
\\
&
+\Big[(|s|^2+n^2\textcolor{black}{-2|s_0| nM })-\delta\sqrt{n} \textcolor{black}{ (1+M^2) }\Big((3+n)M+|s_0|
\Big) \Big]\|F\|^2_2
\end{align*}
and from (\ref{BUTOSLO}) we get

\begin{align}	\label{Eq:Lower_bound_SPHEREROB}
\Sc{q^R_s(F,F)}&\geq \Sc{q_s(F,F)}- |\Sc{\langle b\,\tau_DF,\tau_DF\rangle_{L^2(\partial\Omega)}}|
\nonumber
\\
&
\geq
\Big[ (1+m^2)^2-\frac{1}{\delta}\sqrt{n} \textcolor{black}{ (1+M^2) }\Big((3+n)M+|s_0|
\Big)\Big]\|F\|^2_D\nonumber
\\
&
+\Big[(|s|^2+n^2\textcolor{black}{-2|s_0| nM })-\delta\sqrt{n} \textcolor{black}{ (1+M^2) }\Big((3+n)M+|s_0|
\Big) \Big]\|F\|^2_2\nonumber
 \\
 &
 -2^{\frac{n}{2}}\Vert b\Vert_{L^\infty(\partial\Omega)}\Vert\tau_D\Vert^2\Big[\Vert F\Vert_{D}^2+\Vert F\Vert_{L^2}^2\Big].
\end{align}
 We can write (\ref{Eq:Lower_bound_SPHEREROB}) as

\begin{align}	\label{ALSOEq:Lower_bound_SPHEREROB}
&\Sc{q^R_s(F,F)}\geq \Sc{q_s(F,F)}- \Sc{\langle b\,\tau_DF,\tau_DF\rangle_{L^2(\partial\Omega)}}
\nonumber
\\
&
\geq
\Big[ (1+m^2)^2-2^{\frac{n}{2}}\Vert b\Vert_{L^\infty(\partial\Omega)}\Vert\tau_D\Vert^2-\frac{1}{\delta}\sqrt{n} \textcolor{black}{ (1+M^2) }\Big((3+n)M+|s_0|
\Big)\Big]\|F\|^2_D\nonumber
\\
&
+\Big[(|s|^2+n^2\textcolor{black}{-2|s_0| nM })-2^{\frac{n}{2}}\Vert b\Vert_{L^\infty(\partial\Omega)}\Vert\tau_D\Vert^2-\delta\sqrt{n} \textcolor{black}{ (1+M^2) }\Big((3+n)M+|s_0|
\Big) \Big]\|F\|^2_2\nonumber
\end{align}
and we choose $\delta>0$ such that the two coefficients of $\norm{F}_D^2$ and $\norm{F}_2^2$
are equal. In doing so we observe that  we obtain the same relation as in
(\ref{EQUA_FOR_DELTA_SPHERE_DIRIC}) that is
\begin{equation*}
(1+m^2)^2-\frac{1}{\delta}\sqrt{n} \textcolor{black}{ (1+M^2) }\Big((3+n)M+|s_0|
\Big)=
(|s|^2+n^2\textcolor{black}{-2|s_0| nM })-\delta\sqrt{n} \textcolor{black}{ (1+M^2) }\Big((3+n)M+|s_0|
\Big)
\end{equation*}
because the terms $2^{\frac{n}{2}}\Vert b\Vert_{L^\infty(\partial\Omega)}\Vert\tau_D\Vert^2$ appear in both hand sides.
Setting $A(s_0)$ as in (\ref{AzeroSphere}) and $\delta^+$ as in (\ref{deltaplussphere}),
 the coefficient of $\norm{F}_2^2$
   is equal to the coefficient of $\norm{F}_D^2$
 and it is given by
$$
H^R_{n,m,M}(s):=(|s|^2+n^2\textcolor{black}{-2|s_0| nM })
-2^{\frac{n}{2}}\Vert b\Vert_{L^\infty(\partial\Omega)}\Vert\tau_D\Vert^2-\delta^+\sqrt{n} \textcolor{black}{ (1+M^2) }\Big((3+n)M+|s_0|
\Big)
$$
clearly in $H^R_{n,m,M}(s)$ there appear the contribution of the boundary term. Let us write it as
\begin{align*}
H^R_{n,m,M}(s):=
(|s|^2+n^2\textcolor{black}{-2|s_0| nM })
-2^{\frac{n}{2}}\Vert b\Vert_{L^\infty(\partial\Omega)}\Vert\tau_D\Vert^2
-\delta^+A(s_0)
\end{align*}
and with the substitution of $\delta^+A(s_0)$, given by (\ref{AESSEZERODELTA}) we have the expression
\begin{align*}
H^R_{n,m,M}(s):&=
|s|^2+n^2\textcolor{black}{-2|s_0| nM }-2^{\frac{n}{2}}\Vert b\Vert_{L^\infty(\partial\Omega)}\Vert\tau_D\Vert^2
-\delta^+A(s_0)
\\
&
=|s|^2+n^2\textcolor{black}{-2|s_0| nM }-2^{\frac{n}{2}}\Vert b\Vert_{L^\infty(\partial\Omega)}\Vert\tau_D\Vert^2
\\
&-
\frac{1}{2}\Big(|s|^2+n^2\textcolor{black}{-2|s_0| nM }- (1+m^2)^2+\sqrt{\Delta}\Big),
\end{align*}
so replacing the explicit value of  $\Delta$, see (\ref{HnmM}), we have the
 constant in (\ref{COERCONSTROBSP}):
\begin{align}
H^R_{n,m,M}(s):&
=\frac{1}{2}(|s|^2+n^2\textcolor{black}{-2|s_0| nM })
+\frac{1}{2} (1+m^2)^2
-2^{\frac{n}{2}}\Vert b\Vert_{L^\infty(\partial\Omega)}\Vert\tau_D\Vert^2\nonumber
\\
&
-\frac{1}{2}
\sqrt{\Big((|s|^2+n^2\textcolor{black}{-2|s_0| nM })- (1+m^2)^2\Big)^2+4\Big(\sqrt{n} \textcolor{black}{ (1+M^2) }\Big((3+n)M+|s_0|\Big)\Big)^2}.\nonumber
\end{align}
Now we have to impose that the coefficient $K^R_{n,m,M}(s)$ is positive. We observe that
 the
 inequality with the square root is always positive so the inequality
\begin{align}
&|s|^2+n^2\textcolor{black}{-2|s_0| nM }+ (1+m^2)^2
-2^{\frac{n}{2}+1}\Vert b\Vert_{L^\infty(\partial\Omega)}\Vert\tau_D\Vert^2\noindent
\\
&
>\sqrt{\Big((|s|^2+n^2\textcolor{black}{-2|s_0| nM })- (1+m^2)^2\Big)^2+4\Big(\sqrt{n} \textcolor{black}{ (1+M^2) }\Big((3+n)M+|s_0|\Big)\Big)^2}
\noindent
\end{align}
is equivalent to the system of two inequalities
\begin{align}
|s|^2+n^2\textcolor{black}{-2|s_0| nM }+ (1+m^2)^2
-2^{\frac{n}{2}+1}\Vert b\Vert_{L^\infty(\partial\Omega)}\Vert\tau_D\Vert^2
>0
\end{align}
(which is condition (\ref{QULI1ROBSP})),
and
\begin{align}\label{QULI2RBSPH}
&\Big(|s|^2+n^2\textcolor{black}{-2|s_0| nM }+ (1+m^2)^2
-2^{\frac{n}{2}+1}\Vert b\Vert_{L^\infty(\partial\Omega)}\Vert\tau_D\Vert^2\Big)^2\nonumber
\\
&
>\Big(|s|^2+n^2\textcolor{black}{-2|s_0| nM }- (1+m^2)^2\Big)^2+4\Big(\sqrt{n} \textcolor{black}{ (1+M^2) }\Big((3+n)M+|s_0|\Big)\Big)^2.
\end{align}
With some simplifications inequality (\ref{QULI2RBSPH}) becomes
\begin{align}\label{QULI2RBSPHplus1}
&
(|s|^2+n^2\textcolor{black}{-2|s_0| nM })\Big((1+m^2)^2
-2^{\frac{n}{2}}\Vert b\Vert_{L^\infty(\partial\Omega)}\Vert\tau_D\Vert^2\Big)\nonumber
\\
&
>
\Big(\sqrt{n} \textcolor{black}{ (1+M^2) }\Big((3+n)M+|s_0|\Big)\Big)^2
+(1+m^2)^22^{\frac{n}{2}}\Vert b\Vert_{L^\infty(\partial\Omega)}\Vert\tau_D\Vert^2\nonumber
\\
&
-2^{n}\Vert b\Vert_{L^\infty(\partial\Omega)}^2\Vert\tau_D\Vert^4,
\end{align}
with some computations we conclude that
\begin{align*}
&
|s|^2\Big((1+m^2)^2-2^{\frac{n}{2}}\Vert b\Vert_{L^\infty(\partial\Omega)}\Vert\tau_D\Vert^2\Big)-n(1+M^2)^2s_0^2\nonumber
\\
&
\textcolor{black}{-2|s_0| nM }
\Big((1+m^2)^2+(3+n)(1+M^2)^2-2^{\frac{n}{2}}\Vert b\Vert_{L^\infty(\partial\Omega)}\Vert\tau_D\Vert^2
\Big)\nonumber
\\
&
>
n(1+M^2)^2(3+n)^2M^2
+(1+m^2)^22^{\frac{n}{2}}\Vert b\Vert_{L^\infty(\partial\Omega)}\Vert\tau_D\Vert^2
\nonumber
\\
&
-n^2\Big((1+m^2)^2-2^{\frac{n}{2}}\Vert b\Vert_{L^\infty(\partial\Omega)}\Vert\tau_D\Vert^2\Big)
-2^{n}\Vert b\Vert_{L^\infty(\partial\Omega)}^2\Vert\tau_D\Vert^4\nonumber
\end{align*}
which is the condition (\ref{QULI2RBSPHplus}).
So the coercive estimate (\ref{ROBINqsfromsphere}) becomes
\begin{align}
\Sc{q_s^R(F,F)}  &\geq K^R_{n,m,M}(s)(\|F\|_D^2+\|F\|_2^2)
\end{align}
with coercive constant $H^R_{n,m,M}(s)$ given by  (\ref{COERCONSTROBSP}).

{\em Step 3}. The estimates (\ref{Eq:norm_est_hyper4}) are obtained as in the previous theorems.

\end{proof}

\begin{remark} By writing $s=s_0+Is_1$, the region described by \eqref{QULI1ROBSP} and \eqref{QULI2RBSPHplus} is qualitatively described by branches of conics in the left and right half plane and a case-by-case study according to the datum $b\in L^\infty(\partial\Omega)$ and the constants $m$, $M$ which depend on $\Omega$.
\end{remark}
\section{Concluding remarks}\label{sec_Concluding_remarks}

We conclude this paper with some considerations that contextualize the estimates we have derived within the framework of the $H^\infty$-functional calculus, which manifests in various settings. In the realm of hypercomplex settings, depending on the spectral theory under consideration,
this calculus takes on several distinct formulations and the estimates for sectorial and bi-sectorial operators are crucial for establishing $H^\infty$-functional calculi.

\medskip
We recall that the $H^\infty$-functional calculus for complex operators on Banach spaces, introduced by A. McIntosh \cite{McI1} and further investigated in \cite{MC10,MC97,MC06,MC98}, holds significant relevance across multiple disciplines. For more information and the developments one can consult the books \cite{Haase,HYTONBOOK1,HYTONBOOK2}. \medskip

Regarding the spectral theory on the $S$-spectrum \cite{ACS2016,AlpayColSab2020,FJBOOK,CGK,ColomboSabadiniStruppa2011,ColSab2006,JONAMEM,JONADIRECT}, which has its roots in the theory of slice hyperholomorphic functions, the $H^\infty$-functional calculus has been considered in the papers \cite{ACQS2016,CGdiffusion2018,MANTSCHLOSS}. \medskip

The spectral theory concerning the monogenic spectrum, introduced in \cite{JM}, relies on monogenic functions \cite{DSS}. Further exploration of this topic, including the $H^\infty$-functional calculus within this framework can be found in the books \cite{JBOOK,TAOBOOK}. \medskip

In recent times, a novel branch of spectral theory concerning the $S$-spectrum  has emerged in connection with the factorizations of the second mapping in the Fueter-Sce extension theorem, see \cite{ColSabStrupSce,Fueter,TaoQian1,Sce}. This branch, referred to as the fine structures on the $S$-spectrum, revolves around sets of functions characterized by integral representations, from which the corresponding functional calculi are defined. The development of this new theory and its associated functional calculi is elaborated in the papers \cite{BANJOUR,CDPS1,Fivedim,Polyf1,Polyf2}, while recent investigations into the $H^\infty$-functional calculus linked with the functional calculi of the fine structures are studied in \cite{MILANJPETER,MPS23}.

\section*{Declarations and statements}

{\bf Data availability}. The research in this paper does not imply use of data.

{\bf Conflict of interest}. The authors declare that there is no conflict of interest.


\begin{thebibliography}{99}

\bibitem{adler}
 S. Adler:
  {\em Quaternionic Quantum Mechanics and Quaternionic Quantum Fields}.
   Volume 88 of {\em International Series of Monographs on Physics}. Oxford University Press, New York
(1995).

\bibitem{Ammann}
		B. Ammann: A variational problem in conformal spin geometry. Habilitationsschrift, Universität Hamburg (2003).
		

\bibitem{ACK} D. Alpay, F. Colombo, D.P. Kimsey: \textit{The spectral theorem for quaternionic unbounded normal operators based on the $S$-spectrum}. J. Math. Phys. \textbf{57}(2), 023503, 27 p. (2016).

\bibitem{ACS2016} D. Alpay, F. Colombo, I. Sabadini: \textit{Slice hyperholomorphic Schur analysis}. Volume 256 of Operator Theory: Advances and Applications, Basel, Birkhäuser/Springer, xii, 362 p. (2016).



\bibitem{AlpayColSab2020} D. Alpay, F. Colombo, I. Sabadini: \textit{Quaternionic de Branges spaces and characteristic operator function}. SpringerBriefs in Mathematics, Springer, Cham  x, 116 p. (2020).

 \bibitem{AlpayColSab2024}
   D. Alpay, F. Colombo, I. Sabadini: \textit{
Quaternionic Hilbert spaces and slice hyperholomorphic functions}.
Operator Theory: Advances and Applications 304.
Cham: Birkhäuser, xi, 348 p. (2024).


\bibitem{ACQS2016} D. Alpay, F. Colombo, T. Qian, I. Sabadini: \textit{The $H^\infty$ functional calculus based on the $S$-spectrum for quaternionic operators and for $n$-tuples of noncommuting operators}. J. Funct. Anal. \textbf{271}(6)  1544--1584 (2016).

\bibitem{MC10} P. Auscher, A. Axelsson, A. McIntosh: \textit{On a quadratic estimate related to the Kato conjecture and boundary value problems}. Harmonic analysis and partial differential equations, Contemp. Math., 505, Amer. Math. Soc., Providence, RI  105-129 (2010).

\bibitem{MC97} P. Auscher, A. McIntosh, A. Nahmod: \textit{Holomorphic functional calculi of operators, quadratic estimates and interpolation}. Indiana Univ. Math. J. \textbf{46} 375--403 (1997).

\bibitem{MC06} A. Axelsson, S. Keith, A. McIntosh:
\textit{Quadratic estimates and functional calculi of perturbed Dirac operators}.
Invent. Math. \textbf{163}  455--497 (2006).

\bibitem{BARACCO} L. Baracco, F. Colombo, M.M. Peloso, S. Pinton: \textit{Fractional powers of higher-order vector operators on bounded and unbounded domains}. Proc. Edinb. Math. Soc. \textbf{4}(2) 65, 912-937 (2022).

\bibitem{BF} G. Birkhoff, J. von Neumann: \textit{The logic of quantum mechanics}.
Ann. of Math. (2) \textbf{37}(4)  823--843 (1936).
	
\bibitem{BREZIS}
H. Brezis: {\em Functional analysis, Sobolev spaces and partial differential equations}.
Universitext. Springer, New York,  xiv+599 p. (2011).

\bibitem{BANJOUR} F. Colombo, A. De Martino, S. Pinton: \textit{Harmonic and polyanalytic functional calculi on the $S$-spectrum for unbounded operators}. Banach J. Math. Anal. \textbf{17}, no. 4, Paper No. 84  41 p. (2023).

\bibitem{CDPS1} F. Colombo, A. De Martino, S. Pinton, I. Sabadini:
\textit{Axially harmonic functions and the harmonic functional calculus on the $S$-spectrum}.
J. Geom. Anal. \textbf{33}(2)  54 p. (2023).

\bibitem{Fivedim} F. Colombo, A. De Martino, S. Pinton, I. Sabadini: \textit{The fine structure of the spectral theory on the $S$-spectrum in dimension five}. J. Geom. Anal. \textbf{33}, no. 9, Paper No. 300  73 p. (2023).

\bibitem{ColomboDenizPinton2020} F. Colombo, D. Deniz Gonz\'alez, S. Pinton: \textit{Fractional powers of vector operators with first order boundary conditions}. J. Geom. Phys. \textbf{151}, 18, 103618 (2020).

\bibitem{ColomboDenizPinton2021} F. Colombo, D. Deniz Gonz\'alez, S. Pinton: \textit{The noncommutative fractional Fourier law in bounded and unbounded domains}. Complex Anal. Oper. Theory \textbf{15}(7), Paper No. 114, 27 (2021).

\bibitem{CGdiffusion2018} F. Colombo. J. Gantner:
\textit{An application of the $S$-functional calculus to fractional diffusion processes}.
Milan J. Math. \textbf{86}(2),  225--303 (2018).

\bibitem{FJBOOK} F. Colombo, J. Gantner:
\textit{Quaternionic closed operators, fractional powers and fractional diffusion processes}. Operator Theory: Advances and Applications \textbf{274}, Birkhäuser/Springer, Cham  viii+322 (2019).

\bibitem{CGK} F. Colombo, J. Gantner, D.P. Kimsey: \textit{Spectral theory on the $S$-spectrum for quaternionic operators}. Volume 270 of Operator Theory: Advances and Applications, Birkhäuser/Springer, Cham  ix+356 (2018).

\bibitem{ADVCGKS} F. Colombo, J. Gantner, D.P. Kimsey, I. Sabadini: \textit{Universality property of the $S$-functional calculus, noncommuting matrix variables and Clifford operators}. Adv. Math. \textbf{410}, Paper No. 108719,  39 p. (2022).

\bibitem{ColKim} F. Colombo, D.P. Kimsey: \textit{The spectral theorem for normal operators on a Clifford module}. Anal. Math. Phys., \textbf{12}(1), Paper No. 25, 92 (2022).

\bibitem{Gradient}
		F. Colombo, F. Mantovani, P. Schlosser:
{\em Spectral properties of the gradient operator with nonconstant coefficients.}
 Anal. Math. Phys. 14 , no. 5, Paper No. 108, 31 p. (2024).


\bibitem{ColomboPelosoPinton2019} F. Colombo, M.M. Peloso, S. Pinton:
\textit{The structure of the fractional powers of the noncommutative Fourier law}.
 Math. Methods Appl. Sci. \textbf{42}(18),  6259--6276 (2019).

\bibitem{MILANJPETER} F. Colombo, S. Pinton, P. Schlosser: \textit{The $H^\infty$-functional calculi for the quaternionic fine structures of Dirac type}. Milan J. Math., \textbf{92} , no. 1, 73-122 (2024).

\bibitem{ColSabStrupSce} F. Colombo, I. Sabadini, D.C. Struppa:
\textit{Michele Sce's works in hypercomplex analysis -- A translation with commentaries}. Birkh\"auser/Springer, Cham (2020).

\bibitem{ColomboSabadiniStruppa2011} F. Colombo, I. Sabadini, D.C. Struppa:
\textit{Noncommutative functional calculus. Theory and applications of slice hyperholomorphic functions}.
Volume 289 of Progress in Mathematics, Birkhäuser/Springer Basel AG, Basel vi+221 p. (2011).

\bibitem{ColSab2006} F. Colombo, I. Sabadini, D.C. Struppa: \textit{A new functional calculus for noncommuting operators}.
J. Funct. Anal. \textbf{254}(8),  2255--2274 (2008).

\bibitem{Polyf1} A. De Martino, S. Pinton: \textit{A polyanalytic functional calculus of order 2 on the $S$-spectrum}.
Proc. Amer. Math. Soc. \textbf{151},  2471--2488 (2023).

\bibitem{Polyf2} A. De Martino, S. Pinton:
\textit{Properties of a polyanalytic functional calculus on the $S$-spectrum}. Math. Nachr. \textbf{296}  5190-5226 (2023).

\bibitem{MPS23} A. De Martino, S. Pinton, P. Schlosser: \textit{The harmonic $H^\infty$-functional calculus based on the $S$-spectrum}.
J. Spectr. Theory. \textbf{14},  no. 1, 121-162 (2024).

\bibitem{DSS} R. Delanghe, F. Sommen, V. Sou\v cek: \textit{Clifford algebra and spinor-valued functions.} Mathematics and its Applications \textbf{53}, Kluwer Academic Publishers Group, Dordrecht (1992).

\bibitem{RD} N. Dunford, J. Schwartz: \textit{Linear operators, part I: general theory}. J. Wiley and Sons (1988).

\bibitem{MC98} E. Franks, A. McIntosh: \textit{Discrete quadratic estimates and holomorphic functional calculi in Banach spaces}. Bull. Austral. Math. Soc. \textbf{58},  271--290 (1998).

\bibitem{Fueter} R. Fueter: \textit{Die Funktionentheorie der Differentialgleichungen $\Delta u=0$ und $\Delta\Delta u=0$ mit vier reellen Variablen}. Comment. Math. Helv. \textbf{7}(1),  307--330 (1934).

\bibitem{JONAMEM} J. Gantner: \textit{Operator theory on one-sided quaternion linear spaces: intrinsic S-functional calculus and spectral operators.} Mem. Amer. Math. Soc. \textbf{267}, no. 1297, iii+101 p. (2020).

\bibitem{JONADIRECT}
J. Gantner: \textit{A direct approach to the $S$-functional calculus for closed operators}.
J. Operator Theory. \textbf{77} no.2, 287-331 (2017).

\bibitem{JONAQSTUD}
J. Gantner: \textit{On the equivalence of complex and quaternionic quantum mechanics}.
Quantum Stud. Math. Found. 5, No. 2, 357-390 (2018).

\bibitem{DiracHarm}
		J. E. Gilbert, M. A. M. Murray: {\em Clifford algebras and Dirac operators in harmonic analysis}. Cambridge Studies in Advanced Mathematics, 26. Cambridge University Press, Cambridge,  viii+334 p. (1991).
		

\bibitem{JBOOK} B. Jefferies: \textit{Spectral properties of noncommuting operators}.
Volume 1843, Lecture Notes in Mathematics. Springer-Verlag, Berlin (2004).

\bibitem{Haase} M. Haase: \textit{The functional calculus for sectorial operators}.
Operator Theory: Advances and Applications \textbf{169}, Birkhäuser Verlag, Basel (2006).

\bibitem{HYTONBOOK1} T. Hytönen, J. van Neerven, M. Veraar, L. Weis:
 \textit{Analysis in Banach spaces. Vol. II. Probabilistic methods and operator theory}.
  Ergebnisse der Mathematik und ihrer Grenzgebiete. 3. Folge. A Series of Modern Surveys in Mathematics \textbf{67},
  Springer, Cham  xxi+616 p. (2017).

\bibitem{HYTONBOOK2} T. Hytönen, J. van Neerven, M. Veraar, and L. Weis: \textit{Analysis in Banach spaces. Vol. I. Martingales and Littlewood-Paley theory}. Ergebnisse der Mathematik und ihrer Grenzgebiete. 3.
    Folge. A Series of Modern Surveys in Mathematics \textbf{63}, Springer, Cham xvi+614 p. (2016).

\bibitem{JM} B. Jefferies, A. McIntosh, J. Picton-Warlow: \textit{The monogenic functional calculus}.
Studia Math. \textbf{136}, 99-119 (1999).

\bibitem{Spin}
		 H. B. Lawson Jr., M. Michelsohn:
{\em  Spin geometry}. Princeton Mathematical Series, 38. Princeton University Press, Princeton, NJ,  xii+427 p. (1989).
	

\bibitem{TAOBOOK} P. Li, T. Qian: \textit{Singular integrals and Fourier theory on Lipschitz boundaries}. Science Press Beijing, Beijing, Springer, Singapore xv+306 p. (2019).

\bibitem{MANTSCHLOSS}  F. Mantovani, P. Schlosser:
{\em The $H^\infty$-functional calculus for bisectorial Clifford operators.}
To appear in Journal of Spectral Theory. arXiv:2409.07249.

\bibitem{M87} J. Marschall: \textit{The trace of Sobolev-Slobodeckij spaces on Lipschitz domains}. Manuscripta Math. \textbf{58}, 47-65 (1987).

\bibitem{McI1} A. McIntosh: \textit{Operators which have an $H^\infty$ functional calculus}. Miniconference on operator theory and partial differential equations (North Ryde, 1986), 210--231, Proc. Centre Math. Anal. Austral. Nat. Univ., 14, Austral. Nat. Univ., Canberra (1986).
\bibitem{Ni} L. Nicolaescu: \textit{Lectures on the geometry of manifolds}.  World Scientific Publishing Co. Pte. Ltd., Hackensack, NJ (2021).

\bibitem{TaoQian1} T. Qian: \textit{Generalization of Fueter's result to $\textbf{R}^{n+1}$}. Atti Accad. Naz. Lincei Cl. Sci. Fis. Mat. Natur. Rend. Lincei (9) Mat. Appl. \textbf{8}(2), 111--117 (1997).

\bibitem{Sce} M. Sce: \textit{Osservazioni sulle serie di potenze nei moduli quadratici}. Atti Accad. Naz. Lincei Rend. Cl. Sci. Fis. Mat. Nat. \textbf{23}(8), 220--225 (1957).




\end{thebibliography}
\end{document}